\documentclass[a4paper,10pt,leqno]{amsart}
\usepackage{geometry}
\usepackage{cite}
\usepackage{hyperref}
\usepackage{amsmath,amssymb,amsthm,color}
\hypersetup{linktocpage,colorlinks, citecolor={magenta}}
\usepackage[english]{babel}
\usepackage[utf8]{inputenc}
\usepackage[T1]{fontenc}
\usepackage{lmodern}
\usepackage{paralist}

\usepackage{enumitem}
\setlist{nosep}
\setlist{noitemsep}
\setlist[itemize]{leftmargin=*}

\numberwithin{equation}{section}

\renewcommand{\epsilon}{\varepsilon}

\def \R{\mathbb{R}}
\def \Cc{\mathbb{C}}

\theoremstyle{plain}
\newtheorem{theo}{Theorem}
\newtheorem{prop}{Proposition}[section]
\newtheorem{lem}[prop]{Lemma}
\newtheorem{coro}[prop]{Corollary}
\newtheorem{claim}[prop]{Claim}

\theoremstyle{definition}
\newtheorem{defi}[prop]{Definition}
\newtheorem{remark}[prop]{Remark}

\def \t0{\rightarrow 0} 
\def \ti{\rightarrow \infty}

\def \hal{\frac{1}{2}}
\def \C{\mathcal{C}} 

\def \Esp{\mathbf{E}} 

\def \tD{\widetilde{D}}

\def \B{\mathbf{B}} 

\def \Fluct{\mathrm{Fluct}}

\def \Lap{\mathcal{L}_P}

\def \psis{\psi_s}

\def \smax{\mathrm{s}_{\mathrm{max}}}

\def \Ms{\mathbf{M}_s}

\def \PV{\mathbf{PV}}
\def \phiL{\phi_{\Lambda}}
\def \0{\mathsf{0}}
\def \1{\mathsf{1}}
\def \2{\mathsf{2}}
\def \3{\mathsf{3}}
\def \4{\mathsf{4}}
\def \k{\mathsf{k}}

\def \Varphi{\mathsf{V}[\varphi]}

\def \L{\mathsf{L}}
\def \Phis{\Phi_s}
\def \Ms{\mathbf{M}_s}
\def \psis{\psi_s}

\def \p{\mathsf{p}}

\def \Move{\mathrm{Move}}
\def \C{\mathcal{C}}

\def \Discr{\mathrm{Discr}}

\def \Oun{O_{\bullet}}

\def \R{\mathbb{R}}
\def \Cc{\mathbb{C}}

\def \t0{\rightarrow 0} 
\def \ti{\rightarrow \infty}

\def \hal{\frac{1}{2}}

\def \C{\mathcal{C}} 

\def \Esp{\mathbf{E}} 

\def \tD{\widetilde{D}}

\def \B{\mathbf{B}} 

\def \Fluct{\mathrm{Fluct}}

\def \Lap{\mathcal{L}_P}

\def \psis{\psi_s}

\def \smax{\mathrm{s}_{\mathrm{max}}}

\def \Ms{\mathbf{M}_s}
\def \HLP{\mathfrak{H}_{\Lambda}^{\varphi}}

\def \PV{\mathbf{PV}}
\def \phiL{\phi_{\Lambda}}
\def \0{\mathsf{0}}
\def \1{\mathsf{1}}
\def \2{\mathsf{2}}
\def \k{\mathsf{k}}

\def \Varphi{\mathsf{V}[\varphi]}

\def \L{\mathsf{L}}
\def \Phis{\Phi_s}
\def \Ms{\mathbf{M}_s}
\def \psis{\psi_s}

\def \p{\mathsf{p}}

\def \Move{\mathrm{Move}}
\def \C{\mathcal{C}}

\def \Discr{\mathrm{Discr}}

\def \Oun{O_{\bullet}}
\def \Esp{\mathbf{E}}

\def \id{\mathrm{Id}}

\def \tmsL{\mathbf{m}_{s}}

\def \l{\ell}

\def \Er{\mathrm{Er}}

\def \diagc{(\Lambda \times \Lambda) \setminus \diamond}

\def \V{\mathsf{V}}
\def \phil{\varphi_{\l}}
\def \sineb{\mathrm{Sine}_{\beta}}
\def \HLP{\mathfrak{H}_{\lambda, \varphi}}

\def \ophi{\overline{\varphi}}
\def \tmsL{\mu_{s}}
\def \etas{\eta_s}
\def \Aa{\mathsf{A}}
\def \Bb{\mathsf{B}}
\def \Cc{\mathsf{C}}

\def \Main{\mathrm{Main}}

\def \P{\mathbb{P}}
\def \Esp{\mathbb{E}}
\def \Dt{\widetilde{D}}
\def \DL{\widetilde{D}^{\mathrm{Left}}}

\def \DR{\widetilde{D}^{\mathrm{Right}}}
\def \mulambda{\mathfrak{m}_{\lambda, \varphi}}

\def \muL{\widetilde{\mathfrak{m}}_{\lambda, \varphi}}

\def \ophiH{\|\ophi\|_{H^{\hal}}}
\def \Gam{\mathsf{LP}_{\lambda, \varphi}}
\def \Ms{F_s}
\def \id{\mathrm{Id}_{\lambda}}
\def \PNbeta{\mathbb{P}_{N, \beta}}
\def \ZNbeta{\mathrm{Z}_{N, \beta}}
\def \Event{\mathsf{Event}_{\lambda, \l}}

\def \Pp{\mathsf{P}}
\def \Tt{\mathsf{T}}
\def \Rr{\mathsf{R}}
\def \D{\mathsf{D}}
\def \size{\mathsf{size}}
\def \Varphi{\V_{\lambda, \varphi}}
\def \I{\mathrm{Flu}\mathsf{LP}_A}
\def \II{\mathrm{Flu}\mathsf{LP}_B}
\def \III{\mathrm{Flu}\mathsf{LP}_C}

\def \ErrorLog{\mathrm{Error}\Gam}

\def \ErrorLogL{\ErrorLog^{\mathrm{Left}}}
\def \ErrorLogR{\ErrorLog^{\mathrm{Right}}}
\def \Errorlog{\ErrorLog}
\def \ErrorVar{\mathrm{ErrorVar}}
\def \Vx{V_{x}}
\def \Vi{V_{i}}
\def \Vj{V_{j}}
\def \Vk{V_{k}}
\def \DF{\mathsf{DF}_s}
\def \ErrorDF{\mathrm{Error}\DF}
\def \pto{o_{\l, \lambda}(1)}

\def \Main{\mathsf{Main}_s}
\def \MainA{\Main^{A}}
\def \MainZ{\Main^{\circ}}
\def \MainB{\Main^{B}}
\def \MainC{\Main^{C}}
\def \MainD{\Main^{D}}
\def \RE{\mathsf{RE}_s}
\def \FluRE{\mathrm{Flu}\RE}
\def \Cs{\C_s}
\def \tri{\Delta_s}

\def \Ff{\mathsf{F}}
\def \Gg{\mathsf{G}}
\def \Hh{\mathsf{H}}
\def \Ggeta{\Gg_{\eta}}
\def \qt{\frac{4}{3}}
\def \La{\Lambda}
\def \tMove{\widetilde{\mathcal{M}}_{\Lambda}}
\def \Move{\mathcal{M}_{\Lambda}}
\def \GLa{\mathsf{Gibbs}_{\La, \beta}}
\def \B{\mathbf{B}}
\def \HLa{\mathsf{H}_{\La}}
\def \tHLa{\widetilde{\mathsf{H}}_{\La}}

\def \L{\l}

\begin{document}
\title{CLT for fluctuations of linear statistics in the Sine-beta process}
\author{Thomas Leblé}
\address{Courant Institute of Mathematical Sciences, 251 Mercer Street, New York University, New York, NY 10012-1110, USA}
\date{\today}
\email{thomasl@math.nyu.edu}

\begin{abstract}
We prove, for any $\beta >0$, a central limit theorem for the fluctuations of linear statistics in the $\sineb$ process, which is the infinite volume limit of the random microscopic behavior in the bulk of one-dimensional log-gases at inverse temperature~$\beta$.

If $\ophi$ is a compactly supported test function of class $C^4$, and $\C$ is a random point configuration distributed according to $\sineb$, the integral of $\ophi(\cdot / \l)$ against the random fluctuation $d\C - dx$, converges in law, as $\l$ goes to infinity, to a centered normal random variable whose standard deviation is proportional to the Sobolev $H^{1/2}$ norm of $\ophi$ on the real line.

The proof relies on the DLR equations for $\sineb$ established by Dereudre-Hardy-Maïda and the author, the Laplace transform trick introduced by Johansson, and a transportation method previously used for $\beta$-ensembles at macroscopic scale. 
\end{abstract}
\maketitle

\section{Introduction}

\subsection{The Sine-beta process}
The $\sineb$ process is obtained as the \textit{infinite volume}, or \textit{thermodynamic}, limit of the \textit{microscopic} behavior in the \textit{bulk} of a \textit{one-dimensional log-gas}.

Let $\beta > 0$ be a fixed value of the \textit{inverse temperature} parameter. For $N \geq 1$, the probability measure on $\R^N$ given by the density
\begin{equation}
\label{def:PNbeta}
d\PNbeta(x_1, \dots, x_N) := \frac{1}{\ZNbeta} \exp\left( - \beta \left( \sum_{i < j} - \log |x_i - x_j| + \sum_{i=1}^N \frac{x_i^2}{2} \right) \right), 
\end{equation}
with respect to the Lebesgue measure on $\R^N$, where $\ZNbeta$ is a normalization constant, is the \textit{canonical Gibbs measure} of a one-dimensional log-gas at (inverse) temperature $\beta$. It corresponds physically to a system of $N$ particles interacting via a pairwise repulsive logarithmic potential, and confined by some external field that we take here to be quadratic, for simplicity. 

For $\beta = 1, 2, 4$, the density $\PNbeta$ coincides with the joint law of the $N$ eigenvalues of certain classical models of \textit{random matrices}: the Gaussian orthogonal, unitary, and symplectic ensemble\footnote{With a correct choice of the variance, due to the presence of $\beta$ in front of $\sum_i x_i^2$.}, respectively. We refer to \cite{forrester2010log} for a comprehensive survey of this connection. In fact, for every $\beta > 0$, there exists a model of random matrices with independent entries, known as the “tridiagonal model”, discovered in \cite{dumitriu2002matrix}, whose random eigenvalues behave like the particles of a log-gas at inverse temperature $\beta$.

 From a statistical physics point of view, one-dimensional log-gases are interesting toy models due to the fact that interaction is singular and, most importantly, \textit{long-range}: in contrast to many pair potentials studied in the literature, the logarithmic interaction does not tend rapidly to zero with the distance between the particles (in fact, not at all). 

Under $\PNbeta$, it is known that the particles typically arrange themselves in an interval approximately given by $[-2N, 2N]$. We consider this as being the \textit{microscopic behavior} of the system\footnote{In contrast to another object of study, the \textit{macroscopic behavior}, which corresponds to rescaling the particles by a factor $1/N$ in order for them to stay in some bounded interval, or equivalently to change the $x_i^2$ term in \eqref{def:PNbeta} into $N x_i^2$.}. 

We can see the random $N$-tuple $\C_N := (x_1, \dots, x_N)$ as a random, finite, point configuration in $\R$. The existence of a limit, or even of limit points, in some interesting topology, to the law of $\C_N$ is a difficult question. It was shown in \cite{valko2009continuum}, and \cite{killip2009eigenvalue}\footnote{For a closely related model, whose limit turns out to be the same.} that when taking the thermodynamic/infinite volume limit, i.e. letting $N \to \infty$, the random, finite point configuration $\C_N$ converges in law to some random, infinite point configuration on $\R$, whose law is called the $\sineb$ process. In both cases, a description of $\sineb$ is given through a system of coupled stochastic differential equations.

Finally, since the topology of convergence is local, $\sineb$ only captures the microscopic behavior “near $0$”. One could ask instead for the limit of $\C_N$ translated by $cN$, where $c$ is some parameter. It turns out that for $c$ in $(-2, 2)$, the law of the limit is the same, up to a scaling on the average  density of points. We call this the \textit{bulk} behavior. For $c = \pm 2$, one obtains the \textit{edge} behavior, whose limit is named the $\mathrm{Airy}_\beta$ process. For $|c| > 2$, the limit point process is almost surely empty.

\subsection{Main result: CLT for fluctuations of linear statistics}
\subsubsection{Definitions}
If $\C$ is a point configuration on $\R$, and $\varphi$ a continuous, compactly supported test function, we will often use the notation $\int \varphi(x) d\C(x)$ for
$$
\int \varphi(x) d\C(x) := \sum_{p \in \C} \varphi(p).
$$

\begin{defi}[Fluctuations of linear statistics]
Let $\varphi$ be a function of class $C^0$, compactly supported on $\R$, and let $\C$ be a point configuration on $\R$. We define the fluctuation of the \textit{linear statistic} associated to $\varphi$ as the quantity
\label{def:fluctuations}
\begin{equation}
\Fluct[\varphi](\C) := \int \varphi(x) (d\C(x) - dx).
\end{equation}
\end{defi}

\begin{defi}[Rescaled function]
\label{defi:rescale}
Let $\ophi$ be a test function, and $\l > 0$. We define the associated rescaled test function $\phil$ as 
\begin{equation}
\label{def:varphil}
\phil : x \mapsto  \ophi \left( \frac{x}{\l} \right).
\end{equation}
\end{defi}

\begin{defi}[$H^{1/2}$ norm on the real line]
Whenever the following quantity is finite, we call it the $H^{1/2}$ norm of $\ophi$
\begin{equation}
\label{def:ophiH}
\ophiH := \frac{1}{2\pi} \left( \iint_{\R \times \R} \left( \frac{ \ophi(x) - \ophi(y)}{x-y}\right)^2 dx dy\right)^{1/2}.
\end{equation}
\end{defi}
We may observe, that e.g. when $\ophi$ is of class $C^1$ and compactly supported, then $\ophiH$ is finite. Moreover, it is easy to check that the $H^{1/2}$ norm is invariant under rescaling as in \eqref{def:varphil}.

\subsubsection{Statement of the result}
\begin{theo}[CLT for fluctuations of linear statistics under $\sineb$] \label{theo:fluctuations}
Let $\ophi$ be a fixed test function of class $C^4$, compactly supported on $\R$, and for $\l > 0$, let $\phil$ be the rescaled function, as in Definition \ref{defi:rescale}. Let $\C$ be a random point configuration of law $\sineb$.

The following convergence holds, in law, as $\l \to \infty$,
\begin{equation*}
\Fluct[\phil](\C) \implies \text{Gaussian r.v. of mean $0$ and variance $\frac{2}{\beta} \ophiH^2$}.
\end{equation*}
\end{theo}

\subsubsection{Notation}
{Henceforth, we let $\ophi$ be a fixed test function of class $C^4$, compactly supported in $\R$, and for $\l > 0$ we let $\phil$ be as in Definition \ref{defi:rescale}. For lightness of notation, we drop the subscript $\l$ and write $\varphi$ instead of $\phil$. Also, for simplicity, we assume that $\ophi$ is supported in $(-1,1)$, so that $\varphi = \phil$ is supported in $(-\l, \l)$.}

We work with two parameters $\l, \lambda$. We will always assume that $\l, \lambda$ satisfy
\begin{equation}
\label{lvslambda}
100 < \l < \frac{\lambda}{1000},
\end{equation}
and we will use the notation $a \preceq b$ as follows
$$
a \preceq b \iff |a| \leq C|b|, 
$$
where $C$ is some multiplicative constant \textit{independent} of $\l, \lambda$, provided \eqref{lvslambda} is satisfied. We will sometimes write $\Oun(b)$ to denote a quantity that is $\preceq b$. Most implicit constants will depend on the test function $\ophi$.

If $A$ is a quantity depending on $\l, \lambda$, we use the notation $A = \pto$ to denote the fact that
$$
\lim_{\lambda \ti} \lim_{\l \ti} A = 0.
$$

We let $\Lambda$ be the interval $(-\lambda, \lambda)$.

All the expectations, denoted by $\Esp$, are expectations under $\sineb$, and all the probabilities, denoted by $\P$, are probabilities for $\sineb$.

\subsection{Strategy of the proof and connection with other works}
\subsubsection{Strategy of proof}
The proof relies on three main ingredients:
\begin{enumerate}
\item The DLR equations of \cite{DLRpreprint}.
\item The Laplace transform trick of \cite{MR1487983}.
\item The transportation method inspired by \cite{bekerman2017clt}. 
\end{enumerate}

The DLR (for Dobrushin-Landford-Ruelle) equations provide a version of the Gibbs measure \eqref{def:PNbeta} for “$N~=~+~\infty$”, and thus give a representation $\sineb$ as an \textit{infinite-volume Gibbs measure}, allowing for a “statistical physics approach”. We state these equations precisely in Section \ref{sec:DLRequations}, let us think of them as describing $\sineb$, in any interval, as a mixture of Gibbs measures resembling $\PNbeta$.

The CLT for fluctuations of linear statistics of log-gases has been proven by \cite{MR1487983} in the context of Hermitian random matrices, stated as a limit in law as $N \to \infty$ of fluctuations at macroscopic scale. A key point of the proof is the following observation: forming the Laplace transform of the fluctuations of $\varphi$ amounts to computing the partition function of a log-gas with a perturbed external field, where $\hal x_i^2$ in \eqref{def:PNbeta} is replaced by $\hal x_i^2 + s \varphi(x_i)$, where $s$ is small, and related to the parameter of the Laplace transform. More precisely, one is led to consider the ratio of the perturbed partition function and the original partition function, and the argument boils down to proving fine estimates of this ratio.

One way to compare the partition functions is to use a change of variables, or transportation method, as in e.g. \cite{shchange}, \cite{bekerman2015transport}, \cite{bekerman2017clt}. It effectively shifts the focus from the external field to the associated \textit{equilibrium measure}, in the sense of logarithmic potential theory. Then the question becomes to compute the perturbed equilibrium measure, to push the original one onto the perturbed one by some change of variables (or transportation map), and to use this transport to estimate the ratio of the partition functions. This is closely related to the “loop equations” approach.

Our proof is in the same spirit, with several modifications:
\begin{itemize}
\item The papers cited above treat linear statistics at macroscopic scale, and consider the limit in law as $N \to \infty$ of
$\sum_{i=1}^N \varphi(x_i)$, when $(x_1, \dots, x_N)$ are distributed according to a Gibbs measure similar to $\PNbeta$, with a possibly more general choice of external field. The CLT is also known to hold at mesoscopic scale, when $\varphi$ is taken as $\ophi(\cdot /N^{\delta})$, for $\delta \in (0, \frac{1}{2})$, see \cite{bekerman2016mesoscopic}. 

In contrast, the present work deals with the \textit{microscopic scale} and is the first one to consider the fluctuations for the limit process $\sineb$ itself for arbitrary values of $\beta$. Since we take rescaled functions $\phil$ and let $\l \to \infty$ can think of Theorem \ref{theo:fluctuations} as a result about fluctuations of the log-gas at \textit{large microscopic scales}.
\item When comparing the partition functions, there is usually a term (here $\Main$, see \eqref{def:MainComp}) whose magnitude is \textit{a priori} of order $1$, and must then be studied more carefully to show that it is in fact $o(1)$. This can be done by a technical bootstrap argument, and, in fact, this way, one can even obtain an all-order expansion of the partition function as in \cite{BorGui1}; another approach uses the independent knowledge of the partition function up to order $N$, as in \cite{bekerman2017clt}. 

Here, we use discrepancy estimates for $\sineb$ and are able to show directly that $\Main$ is $o(1)$.
\end{itemize}

\subsubsection{Connections with other works}
When $\beta = 2$, the point process acquires a particularly rich \textit{determinantal} structure, allowing for many explicit computations. In this case, the CLT for fluctuations of smooth enough functions was known since \cite{spohn1987interacting}, see also \cite{soshnikov2000central}. Let us observe that, for $\beta = 2$, the CLT is known to hold as soon as the test function is in $H^{1/2}(\R)$, to be compared with the requirement that $\ophi \in C^4_c$ here. The optimal regularity condition needed in the general $\beta$ case is an open question.

For $\beta$ arbitrary, a CLT similar to Theorem \ref{theo:fluctuations} was announced by \cite{HLPVV}. The method is completely different, and uses the original description of $\sineb$ involving coupled stochastic differential equations.

Several facts concerning the number of points under $\sineb$ have been proven. A CLT follows from \cite{kritchevski2012scaling}, large deviations were proven in  \cite{holcomb2015large,holcomb2017overcrowding} and a maximal deviation result in \cite{HolcombePaquette18}. The transportation strategy does not accommodate well to non-smooth functions like indicator functions, and we are unable to easily retrieve these results with the present techniques. 

The \textit{rigidity} of the process in the sense of Ghosh-Peres, i.e. the fact that the knowledge of the configuration outside a given compact set almost surely prescribes the number of points in that set, was proven by \cite{chhaibi2018rigidity}, and also obtained in \cite{DLRpreprint} in a very different way. The proof of \cite{chhaibi2018rigidity} follows the approach of \cite{ghosh2017rigidity} and relies on the fact that the variance of linear statistics is controlled by the $H^{1/2}$ norm of the test function, which had been established for random matrix models and can be passed to the limit. We believe that our “statistical physics” approach could yield similar bounds, and hence the rigidity result, but one needs to go over all the estimates beyond the “rescaled cases”  $\varphi = \phil = \ophi(\cdot / \l)$ and state them in full generality, with controls depending more precisely on $\varphi$, we do not undertake this here.

\subsubsection{Plan of the paper}
\begin{itemize}
\item In Section \ref{sec:aprioribounds}, we discuss discrepancy estimates for the $\sineb$ process and state an \textit{a priori} bound on the fluctuations on linear statistics, in terms of the discrepancies. We will rely constantly on this bound in order to control the error terms in the Laplace transform expansion.
\item In Section \ref{sec:perturbationmeasure}, we define the \textit{perturbation measure}, which formally corresponds to the change induced on the average density of points when treating the test function $\varphi$ as an additional external field applied to each particle. This perturbation measure is slightly singular, and we work in fact with a regularized version, the \textit{approximate perturbation measure}
.\item In Section \ref{sec:transport}, we define the perturbed measure, the transport map from the original measure (the constant density) to the perturbed one, and we expand the energy along this transport.
\item In Section \ref{sec:compareener}, we compare the interaction energy before and after transport, and show that most terms are negligible.
\item In Section \ref{sec:finalproof}, we combine all previous elements to give the proof of the CLT.
\item Many parts of the argument are rather elementary, but involve some lengthy computations. For legibility, we have postponed most of the computations to Section \ref{sec:auxiliary}.
\end{itemize}

\subsection{Semi-norms}
We will often use $g^{(\k)}$ to denote the $\k$-th derivative of $g$.

\begin{defi}[Semi-norms and local semi-norms]
Let $g$ be a test function, compactly supported on $\R$. For $\k \geq 0$, if $g$ is assumed to be of class $C^\k$, we let 
$$
|g|_{\k} := \sup_{x} |g^{(\k)}(x)|,
$$
and for $x$ in $\R$, letting $\Vx$ denote the neighborhood $\Vx := [x-3, x+3]$, we write:
\begin{equation}
\label{def:Vdex}
|g|_{\k, \Vx} := \sup_{y \in \Vx} |g^{(\k)}(y)|.
\end{equation}
\end{defi}

The following bounds will be used repeatedly:
\begin{equation}
\label{bounds:rescale}
|\phil|_{\k} = \frac{1}{\l^\k} |\ophi|_{\k}, \quad \|\phil^{(\k)}\|_{L^{\p}} = \|\ophi\|_{L^{\p}} \l^{\frac{1}{\p} - \k}.
\end{equation}
\subsection{Discrepancy and discrepancy estimates}
Throughout the paper, an important role is played by the discrepancy estimates, for they provide an \textit{a priori} bound on the size of fluctuations that we will repeatedly use to control error terms. If $\C$ is a point configuration and $I$ is an interval, we denote by $\C_{I}$ the restriction of $\C$ to $I$.

\begin{defi}[Discrepancy]
\label{def:discrepancy}
Let $\C$ be a point configuration on $\R$, and let $I$ be an interval. The discrepancy of $\C$ in $I$ is the difference between the number of points of $\C$ in $I$ and its expected value, namely the length of $I$. We write
$$
\Discr_{I} := |\C_{I}| - |I| = \int \1_{I} (d\C(x) - dx).
$$
If $a,b$ are integers, with possibly $a > b$, we let
$$
\Discr_{[a,b]} := \int_{a}^b \1_{I} (d\C(x) - dx).
$$
\end{defi}
It is known, see e.g. \cite{Leble:2017mz}[Lemma 3.2] that, if $I$ has length at least $1$, we have
\begin{equation}
\label{discresA}
\Esp \left[ \left(\Discr_I\right)^2 \right] \preceq |I|, \quad \Esp \left[ | \Discr_{I} | \right| \preceq \sqrt{|I|}.
\end{equation}
Moreover, it was shown in \cite{Leble:2017mz}[Remark 3.3] that, for $\sineb$, it holds
$$
\liminf_{R \to \infty} \frac{1}{2R} \Esp \left[ \left(\Discr_{[-R,R]}\right)^2 \right] = 0,
$$
and careful inspection of the argument yields the stronger statement, proven in \cite{uniqueness}
\begin{equation}
\label{discrasym}
\Esp \left[ \left(\Discr_{[-R,R]}\right)^2 \right] = o(R).
\end{equation}
Of course, since $\sineb$ is stationary, it implies that the variance of the number of points in any interval of length $R$ is $o(R)$.

\subsection{A priori bound on the fluctuations}
\label{sec:aprioribounds}
We let $\Dt_i$ be the quantity
$$
\Dt_i := |\Discr_{[0, i]}| + |\Discr_{[i, i+1]}| + 1.
$$
\begin{prop}[A priori bound on the fluctuations]
\label{prop:aprioribounds}
Let $g$ be a test function of class $C^1$, compactly supported on $\R$. 
\begin{equation}
\label{aprioribound}
 \left| \int g(x) (d\C - dx) \right| \preceq \sum_{i=-\infty}^\infty |g|_{\1, \Vi}  \Dt_i[\C].	
\end{equation}
Moreover, for $\lambda$ fixed we may choose to replace $\Dt_i$ by either $\DL_i$ or $\DR_i$, with
$$
\DL_i := |\Discr_{[-\lambda, i]}| + |\Discr_{[i, i+1]}| + 1, \quad \DR_i  := |\Discr_{[i, \lambda]}| + |\Discr_{[i, i+1]}| + 1.
$$
\end{prop}
The proof of Proposition \ref{prop:aprioribounds} is elementary, we postpone it to Section \ref{sec:proofaprioribounds}.

\begin{remark}[Bounds on $\Dt, \DL, \DR$]
In view of \eqref{discresA}, for $|i| \geq 1$, we have
\begin{equation}
\label{discrestimateA}
\Esp \left[ \left(\Dt_i\right)^2 \right] \preceq |i|, \quad \Esp \left[ \Dt_i \right] \preceq \sqrt{|i|}, 
\end{equation}
and in fact we have, in view of \eqref{discrasym}, as $|i| \to \infty$
\begin{equation}
\label{discrestimateB}
\Esp \left[ \left(\Dt_i\right)^2 \right] = o\left(|i|\right), \quad \Esp \left[ \Dt_i \right] = o \left( \sqrt{|i|} \right).
\end{equation}
We obtain similar estimates for $\DL_i$, resp. $\DR_i$ when replacing $|i|$ by $|\lambda + i|$, resp. $|\lambda - i|$.
\end{remark}

\section{The perturbation measure}
\label{sec:perturbationmeasure}
\subsection{The Cauchy principal value}
\begin{defi}[Cauchy principal value] 
\label{defi:CauchyPV}
Let $g$ be a test function of class $C^1$, compactly supported on $\R$. For $x$ in $\R$, we define
\begin{equation}
\label{def:CauchyPV}
\PV \int \frac{g(t)}{t-x} dt := \int_{0}^{+\infty} \frac{g(x+u) - g(x-u)}{u} du,
\end{equation}
where $\PV$ stands for “principal value”.
\end{defi}

\begin{defi}[The quantity $\HLP$]
\label{sec:finiteHilb}
For $x$ in $\R$, we define $\HLP(x)$ as
\begin{equation}
\label{def:HLP}
\HLP(x) := \frac{1}{\pi} \PV \int \frac{\sqrt{\lambda^2-t^2} \varphi'(t)}{t-x} dt.
\end{equation}
\end{defi}

\begin{remark}
Since $\varphi$ is at least, $C^2$ and compactly supported in $(-\l, \l)$, we can see $\phiL$, defined by
\begin{equation}
\label{def:philambda}
\phiL : t \mapsto \sqrt{\lambda^2 - t^2} \varphi'(t),
\end{equation}
as a compactly supported function of class $C^1$, so the “principal value” notation in \eqref{def:HLP} makes sense, in view of Definition \ref{defi:CauchyPV}.
\end{remark}

\subsection{The perturbation measure}
\begin{defi}[The perturbation measure]
For $x$ in $(-\lambda, \lambda)$, we define $\mulambda(x)$ as
\begin{equation}
\label{def:mu}
\mulambda(x) := \frac{-1}{\pi \sqrt{\lambda^2-x^2}} \HLP(x).
\end{equation}
The density $\mulambda$ will be called the \textit{perturbation measure}.
\end{defi}

\begin{defi}[The logarithmic potential of $\mulambda$]
For $x$ in $\R$, we let
\begin{equation}
\label{def:Gam}
\Gam(x) := \int - \log |x-y| \mulambda(y) dy.
\end{equation}
\end{defi}

\begin{lem}[Properties of the perturbation measure] 
\label{lem:propofperturb}
The density $\mulambda$ is integrable on $\Lambda$, of total mass $0$. The logarithmic potential generated by $\mulambda$ is well-defined and satisfies the following equation for $x$ in $\Lambda$
\begin{equation}
\label{hmuDSE}
\left(\Gam\right)'(x) = \varphi'(x).
\end{equation}
\end{lem}
These properties are well-known and we refer to the book \cite{MR0094665}, see also Section~\ref{sec:proofpropofperturb}.

\subsection{Bounds on the perturbation measure}
\begin{lem}[Bounds on the perturbation measure $\mulambda$]
\label{lem:boundsonmulambda} 
We have
\begin{align}
\label{mulambda0} & \mulambda \preceq 
\begin{cases}
\frac{1}{\l} & |x| \leq 2\l, \\
\frac{\sqrt{\lambda} \l}{x^2 \sqrt{\lambda-|x|}} & |x| \geq 2\l
\end{cases}, \\
\label{mulambda1} 
& \mulambda^{(\1)} \preceq \begin{cases}
\frac{1}{\l^2} & |x| \leq 2\l, \\
\frac{\l}{|x| \sqrt{\lambda} (\lambda - |x|)^{3/2}} + \frac{\sqrt{\lambda} \l}{|x|^3 \sqrt{\lambda-|x|}} & |x| \geq 2\l
\end{cases}, \\
\label{mulambda2} 
& \mulambda^{(\2)} \preceq \begin{cases}
\frac{1}{\l^3} & |x| \leq 2\l, \\
\frac{\l}{\lambda^{3/2} (\lambda - |x|)^{5/2}} + \frac{\l}{x^2 \lambda^{1/2} (\lambda - |x|)^{3/2}} + \frac{\sqrt{\lambda} \l}{x^4 \sqrt{\lambda-|x|}} & |x| \geq 2\l
\end{cases}.
\end{align}
\end{lem}
Lemma \ref{lem:boundsonmulambda} follows from elementary computations, see Section \ref{sec:proofboundsonmu}.

\subsection{The approximate perturbation measure}
\label{sec:approximatemu}
The perturbation measure $\mulambda$ satisfies the exact relation \eqref{hmuDSE}, but is singular near $\pm \lambda$. We will work instead with an \textit{approximate} perturbation measure $\muL$, constructed below, which is more regular, and in fact vanishes near the endpoints. Of course, passing from $\mulambda$ to $\muL$ induces an error on the logarithmic potential, which we need to control.

\begin{lem}[The approximate perturbation measure]
\label{lem:approximatemu}
There exists a function $\muL$ of class $C^2$, compactly supported in $(-\lambda, \lambda)$, satisfying:
\begin{enumerate}
\item $\muL$ = $\mulambda$ on $[- \lambda + \L, \lambda - \L]$.
\item The masses of $\muL$ and $\mulambda$ coincide near each endpoint, i.e.
\begin{equation}
\label{massisconserved}
\int_{-\lambda}^{-\lambda + \L} \muL = \int_{-\lambda}^{-\lambda + \L} \mulambda, \quad \int_{\lambda - \L}^{\lambda} \muL = \int_{\lambda - \L}^{\lambda} \mulambda
\end{equation}
\item For $x$ in $[-\lambda, -\lambda + \L] \cup [\lambda - \L, \lambda]$, and for any $\k = \0, \1, \2$ we have the bound
\begin{equation}
\label{bound:muL}
|\muL^{(\k)}(x)| \preceq \frac{1}{\L^\k} \frac{\l}{\lambda^{3/2} \L^{1/2}} .
\end{equation}
with implicit multiplicative constants depending on $\k$ and $\ophi$, but not on $\l, \lambda, x$.
\item $\muL$ is identically $0$ on $[-\lambda, -\lambda + \L/4]$ and on $[\lambda - \L/4, \lambda]$.
\end{enumerate} 
\end{lem}
The construction of $\muL$ is given in Section \ref{sec:proofapproximatemu}. 

\begin{lem}[Additional properties of $\muL$]
\label{lem:boundsonmuL}
\begin{align}
\label{muL0} & \muL(x) \preceq
\begin{cases}
\frac{1}{\l} & |x| \leq 2\l, \\
\frac{\l}{x^2} & 2\l \leq |x| \leq \lambda/2 \\
\frac{\l}{\lambda^{3/2} \sqrt{\lambda - x}} & \lambda/2 \leq |x| \leq \lambda - \L, \\
\frac{\l}{\lambda^{3/2} \L^{1/2}} & \lambda - \L \leq |x| \leq \lambda
\end{cases}, \\
\label{muL1} & \muL^{(1)}(x) \preceq 
\begin{cases}
\frac{1}{\l^2} & |x| \leq 2 \l \\
\frac{\l}{|x|^3} & 2\l \leq |x| \leq \lambda/2, \\
\frac{\l}{\lambda^{3/2} (\lambda - |x|)^{3/2}} & \lambda/2 \leq |x| \leq \lambda - \L, \\
\frac{\l}{\lambda^{3/2} \L^{3/2}} & \lambda - \L \leq |x| \leq \lambda
\end{cases},
\\
\label{muL2} & \muL^{(2)}(x) \preceq 
\begin{cases}
\frac{1}{\l^3} & |x| \leq 2 \l \\
\frac{\l}{x^4} & 2\l \leq |x| \leq \lambda/2, \\
\frac{\l}{\lambda^{3/2} (\lambda - |x|)^{5/2}} & \lambda/2 \leq |x| \leq \lambda - \L, \\
\frac{\l}{\lambda^{3/2} \L^{5/2}} & \lambda - \L \leq |x| \leq \lambda
\end{cases},
\\
\label{muLL1} & \|\muL\|_{L^1} \preceq 1.
\end{align}
\end{lem}
\begin{proof}
The first three inequalities are consequences of \eqref{mulambda0}, \eqref{mulambda1}, \eqref{mulambda2}, for $|x| \leq 2\l$ and $2\l \leq |x| \leq \lambda - \L$, because we do not change the measure there. They follow from \eqref{bound:muL} for $\lambda - \L \leq |x| \leq \lambda$. 

To obtain \eqref{muLL1}, we split the integral into four parts:
$$
\int_{|t| \leq 2\l} |\muL(t)| dt + \int_{2\l \leq |t| \leq \lambda /2} |\muL(t)| dt + \int_{\lambda /2 \leq |t| \leq \lambda - \L} |\muL(t)| dt + \int_{\lambda - \L \leq |t| \leq \lambda}  |\muL(t)| dt,
$$
then \eqref{muLL1} follows from \eqref{muL0} and an elementary computation.
\end{proof}

\subsection{The error on the logarithmic potential}
\begin{defi}[Error on the logarithmic potential]
We introduce the quantity
\begin{equation}
\label{def:ErrorLog}
\ErrorLog(x) := \int - \log |x-y| \left( \muL(y) - \mulambda(y) \right) dy.
\end{equation}
\end{defi}

\begin{prop}[Error on the logarithmic potential]
\label{prop:logpotmuL}
We have
\begin{equation*}
\ErrorLog = \ErrorLogL + \ErrorLogR,
\end{equation*}
where $\ErrorLogL$, $\ErrorLogR$ satisfy:
\begin{equation}
\label{bound:ErrorLog0} 
\ErrorLogL(x) \preceq \frac{\l^{3/2} \log(\lambda)}{\lambda^{3/2}}, \quad \ErrorLogR(x) \preceq \frac{\l^{3/2} \log(\lambda)}{\lambda^{3/2}}, \quad \text{ for } |x| \leq 2 \lambda.
\end{equation}
and 
\begin{equation}
\label{bound:ErrorLog1}
\begin{cases}
 \left(\ErrorLogL\right)^{(\1)}(x) \preceq \frac{\l^{5/2}}{\lambda^{3/2} (-\lambda - x)^2} &  |x - (-\lambda)| \geq 2\L, \\
  \left(\ErrorLogR\right)^{(1)}(x) \preceq \frac{\l^{5/2}}{\lambda^{3/2} (\lambda - x)^2}, &  |\lambda - x| \geq  2\L.
\end{cases}
\end{equation}
\end{prop}

The proof of Proposition \ref{prop:logpotmuL} is given in Section \ref{sec:logputmuL}.

\subsection{The variance term}

\begin{lem}[The variance term]
\label{lem:variance}
We have the following identity
\begin{equation}
\label{energiemsLm}
\iint - \log |x-y| \muL(x) \muL(y) dx dy = 2 \ophiH^2 + \ErrorVar,
\end{equation}
with $\ErrorVar$ bounded as follows
\begin{equation}
\label{erreur:variance}
\ErrorVar \preceq \frac{\l^3 \log(\lambda)}{\lambda^{3}} + \frac{\l^2}{\lambda^2}.
\end{equation}
In particular we obtain
\begin{equation}
\label{small:ErrorVar} s^2 \ErrorVar = s^2 \pto.
\end{equation}
\end{lem}
The proof of Lemma \ref{lem:variance} is given in Section \ref{sec:prooflemvariance}.

\section{Transporting to the perturbed measure}
\label{sec:transport}
\subsection{The perturbed measure}
The $\sineb$ process has intensity $1$. Adding a perturbative external field will formally change the average density of points from a constant density to the perturbed density $(1 + \mulambda(x))dx$. Since we work with the approximate perturbation $\muL$, it leads to the following definition.

\begin{defi}[The perturbed measure]
Let $\smax$ be defined as
\begin{equation}
\label{def:smax}
\smax := \hal \max\left(1, |\muL|_{\0}, \|\muL\|_{L^1} \right)^{-1}.
\end{equation}
For any $s$ such that $|s| \leq \smax$, we define the \textit{perturbed measure} $\tmsL$ as
\begin{equation}
\tmsL(x) = 1 + s \muL(x).
\end{equation}
Of course, $\tmsL$ depends on $\lambda, \ophi, \l$ but for simplicity we only keep track of the parameter $s$. In the following, $s$ is always assumed to satisfy $|s| \leq \smax$.
\end{defi}

\begin{lem}[Properties of the perturbed measure]
The density $\tmsL$ is of class $C^2$, is bounded above and below on $[-\lambda, \lambda]$ by universal positive constants, and satisfies
$$
\int_{-\lambda}^{\lambda} \tmsL(x) dx = \int_{-\lambda}^{\lambda} 1 dx.
$$
The density $\tmsL$ is equal to $1$ on $[-\lambda, -\lambda + \L/4]$ and on  $[\lambda - \L/4, \lambda]$.
\end{lem}
\begin{proof}
This follows directly from the construction of $\muL$ as in Lemma \ref{lem:approximatemu}, and from the choice of $\smax$ as in \eqref{def:smax}.
\end{proof}

\subsection{Energy splitting}
Let $\Lambda$ be the interval $\Lambda = (-\lambda, \lambda)$, and let $\C$ be a point configuration in $\Lambda$. Let $\diamond$ be the diagonal in $\Lambda \times \Lambda$.

\begin{lem}[Energy splitting around $\tmsL$]
\label{lem:energyplusfluct}
The following identity holds:
\begin{multline}
\iint_{\diagc} - \log |x-y| (d\C(x) - dx)(d\C(y) - dy)
\\ = \iint_{\diagc} - \log |x-y| (d\C(x) - d\tmsL(x))(d\C(y) - d\tmsL(y)) \\
+ 2 s \int_{\Lambda} \Gam(x) (d\C - dx) + 2s \int_{\Lambda} \ErrorLog(x) (d\C-dx) \\
- 2 s^2 \ophiH^2 - s^2 \ErrorVar,
\end{multline}
where $\Gam$ is the logarithmic potential generated by $\mulambda$, as in \eqref{def:Gam}, the error term $\ErrorLog$ is defined in \eqref{def:ErrorLog}, and $\ErrorVar$ satisfies \eqref{erreur:variance}. 
\end{lem}

We postpone the proof of Lemma \ref{lem:energyplusfluct} to Section \ref{sec:proofenergyplusfluct}, it simply consists in putting together the various definitions given above.

\subsection{The transport map} 
\begin{defi}[Transport map]
\label{defi:TransportMap}
We let $\Ms$ be the cumulative distribution function of the density $\tmsL$ 
\begin{equation*}
\Ms := x \mapsto \int_{-\lambda}^x \tmsL(y) dy,
\end{equation*}
and we define the transport map $\Phis$ on $[-\lambda, \lambda]$ as
\begin{equation}
\label{def:Phis}
\Phis(x) := \Ms^{-1}(x + \lambda),
\end{equation}
so that $\Phis$ satisfies, for $x$ in $[-\lambda, \lambda]$, the identity
$$
\int_{-\lambda}^{\Phis(x)} \tmsL(y) dy = \int_{-\lambda}^x 1 dy.
$$
\end{defi}

\begin{lem}[Properties of the transport map]
\label{lem:transportmap}
The map $\Phis$ is a $C^1$, increasing bijection from $[-\lambda, \lambda]$ to itself. The push-forward of the constant density $1$ on $[-\lambda, \lambda]$ by $\Phis$ is equal to $\tmsL$, i.e. for any measurable function $f$ we have
$$
\int_{-\lambda}^{\lambda} f \circ \Phis(x) dx = \int_{-\lambda}^\lambda f(x) \tmsL(x) dx.
$$

Let $\id$ be the identity map from $[-\lambda, \lambda]$ to itself. The transport map $\Phis$ coincides with $\id$ near the endpoints, more precisely on $[-\lambda, - \lambda + \L/4]$ and $[\lambda - \L/4, \lambda]$. We define $\psis$ as 
\begin{equation}
\label{def:psis}
\psis := \Phis - \id.
\end{equation}
The map $\psis$ satisfies
\begin{equation}
\label{PhisA0}
\psis(x) =  - s \int_{-\lambda}^{\Phis(x)} \muL(y) dy,
\end{equation}
in particular, we have the rough control
\begin{equation}
\label{bound:psis0} |\psis|_{\0} \leq 1, \text{ i.e. }|\Phis - \id|_{\0} \leq 1,
\end{equation}
and we also obtain
\begin{align}
\label{bound:psis1} \psis^{(1)}(x) & \preceq s |\muL|_{\0, V(x)}, \\
\label{bound:psis2} \psis^{(2)}(x) & \preceq s |\muL|_{\1, V(x)}, \\
\label{bound:psi3} \psis^{(3)}(x) & \preceq s|\muL|_{\2, V(x)}.
\end{align}
\end{lem}
The proof of Lemma \ref{lem:transportmap} is given in Section \ref{sec:prooftransportmap}.

\begin{lem}[Finer bound on $\psis$]
\label{lem:additionalpropertiespsis}
We have
\begin{equation}
\label{psisplusprecis}
\psis(x) \leq s \int_{-\lambda}^{x + 1} |\muL(y)| dy, \quad \psis(x) \leq s \int_{x-1}^{\lambda} |\muL(y)| dy.
\end{equation}

We obtain the following bounds, which improve on \eqref{bound:psis0}:
\begin{equation}
\label{psisplusprecis2}
\psis(x) \preceq s \times 
\begin{cases}
1 & |x| \leq 10 \l\\
\frac{\l}{|x|} & 10 \l \leq |x| \leq \lambda /2, \\
\frac{\l}{\lambda^{3/2}} \sqrt{\lambda  - |x|} & |x| \geq \lambda/2.
\end{cases}
\end{equation}
\end{lem}
The proof of Lemma \ref{lem:additionalpropertiespsis} is given in Section \ref{sec:proofadditionalpropertiespsis}

\begin{defi}[The slope of the transport]
For $x,y$ in $\Lambda$ we define $\tri(x,y)$ as
\begin{equation}
\label{def:tri}
\tri(x,y) := \frac{\psis(y) - \psis(x)}{y-x},
\end{equation}
with the natural convention that $\tri(x,x) = \psis'(x)$.
\end{defi}

\subsection{Energy expansion along a transport}
We introduce the following notation.
\begin{align}
\label{def:MainComp} \Main(\eta) & := \iint_{\Lambda \times \Lambda} - \log | 1 + \tri(x,y) | (d\eta - dx) (d\eta -dy) \\
\label{def:RE} \RE & := - \int \log \tmsL(x) \tmsL(x) dx \\
\label{def:FluRE} \FluRE(\eta) & := -  \int \log \tmsL \circ \Phis(x) (d\eta - dx).
\end{align}
The term $\Main(\eta)$ will be the main term in the energy comparison below. The term $\RE$ is the relative entropy of $\tmsL$, which is independent on the point configuration, and $\FluRE(\eta)$ is the fluctuation of the relative entropy functional, which depends on $\eta$.

\begin{lem}[Energy expansion along a transport] 
\label{lem:energycomparison}
Let $\eta$ be a point configuration in $\Lambda$, let $\eta_s$ be the push-forward of the configuration $\eta$ by the map $\Phis$.

We have
\begin{multline}
\label{energycomparison}
\iint_{\diagc} - \log |x-y| (d\etas(x) - \tmsL(x) dx)(d\etas(y) - \tmsL(y)dy)  \\ 
= \iint_{\diagc} - \log |x-y| (d\eta(x) - dx)(d\eta(y) - dy)  + \Main(\eta)  + \RE + \FluRE(\eta).
\end{multline}
\end{lem}

\begin{proof}[Proof of Lemma \ref{lem:energycomparison}]
Since, by construction, $\Phis$ transports $\eta$ onto $\etas$ and the constant density $dx$ onto $\tmsL(x) dx$, we may write
\begin{multline*}
\iint_{\diagc} - \log|x-y| (d\eta_s(x) - \tmsL(x) dx)(d\eta_s(y) - \tmsL(y)dy) \\
= \iint_{\diagc} - \log |\Phis(x) - \Phis(y)| (d\eta - dx) (d\eta - dy)  \\
= \iint_{\diagc} - \log|x-y| (d\eta - dx) (d\eta - dy) \\
+ \iint_{\diagc} - \log | 1 + \tri(x,y) | (d\eta - dx) (d\eta -dy),
\end{multline*}
where we have used the definition $\psis = \Phis - \id$ and the definition of $\tri$ as in \eqref{def:tri}. 

Since $\tri$ is continuously extended by $\psis'$ on the diagonal, we may write
\begin{multline*}
\iint_{\diagc} - \log | 1 + \tri(x,y) | (d\eta - dx) (d\eta -dy) \\ 
= \iint_{\Lambda \times \Lambda} - \log | 1 + \tri(x,y) | (d\eta - dx) (d\eta -dy) 
+ \int \log | 1 + \psis'(x) | d\eta.
\end{multline*}
The first term in the right-hand side corresponds to the definition \eqref{def:MainComp} of $\Main$. We claim that
\begin{equation}
\label{intphisprime}
\int \log (1 + \psis'(x)) (d\eta - dx) = \int \log \Phis'(x) (d\eta - dx) = \RE + \FluRE(\eta).
\end{equation} 
To prove \eqref{intphisprime}, let us observe that $1 + \psis'(x) = \Phis'(x)$ and, by definition of a transport, we have
\begin{equation*}
1 + \psis'(x) = \Phis'(x) = \frac{1}{\tmsL \circ \Phis(x)}
\end{equation*}
We obtain
\begin{equation*}
\int \log  |1 + \psis'(x) | d\eta = \int \log \Phis'(x) = - \int \log \tmsL \circ \Phis(x) d\eta. 
\end{equation*}
Finally, let us write
$$
-\int \log \tmsL \circ \Phis(x) d\eta = - \int \log \tmsL \circ \Phis(x) dx -  \int \log \tmsL \circ \Phis(x) (d\eta - dx).
$$
The first term in the right-hand side can be seen, using the fact that $\Phis$ transports the Lebesgue density onto $\tmsL$, as
$$
- \int \log \tmsL \circ \Phis(x) dx = - \int \log \tmsL(x) \tmsL(x) dx,
$$
so we obtain
\begin{multline}
\label{PhisprimeC}
\int \log  |1 + \psis'(x) | d\eta(x) = \int \log \Phis'(x) d\eta(x) \\
= - \int \log \tmsL(x) \tmsL(x) dx -  \int \log \tmsL \circ \Phis(x) (d\eta - dx).
\end{multline}
Using the notation introduced above in \eqref{def:MainComp}, \eqref{def:RE}, \eqref{def:FluRE}, this concludes the proof of \eqref{energycomparison}.
\end{proof}

\section{Comparison of energies I: the interior-interior interaction}
\def \Lala{\Lambda \times \Lambda}
\label{sec:compareener}

\subsection{The main term in the comparison}
We have
\begin{prop}[The main term is often small]
\label{prop:mainterms}
$$
\Esp[ \left|\Main\right| ] = s \pto.
$$
\end{prop}
The proof of Proposition \ref{prop:mainterms} is rather elementary, but involves cumbersome computations. We postpone it to Section \ref{sec:proofpropmainterms}.

\subsection{The relative entropy term}
\begin{lem}[The term $\RE$ is small]
\label{lem:relativeentropy}
We have
\begin{equation}
\label{relativeentropy}
\RE = - \int \log \tmsL(x) \tmsL(x) dx  \preceq \frac{s^2}{\l}.
\end{equation}
In particular, we obtain
\begin{equation}
\label{smallRE} \RE = s^2 \pto.
\end{equation}
\end{lem}

\begin{proof}
We write $\tmsL = 1 + s \muL$, expand the $\log$ and use the fact that $\muL$ has total mass $0$. We obtain
$$
\int \log \tmsL(x) \tmsL(x) dx = \int \left(- s\muL + \Oun\left( s^2 \muL^2 \right)\right) (1 + s \muL) \preceq s^2 \|\muL^2\|_{L^1}.
$$
Using \eqref{muL0}, we see that $\|\muL^2\|_{L^1} \preceq \frac{1}{\l}$, which yields \eqref{relativeentropy}.
\end{proof}

\subsection{The fluctuations of the relative entropy term}
\begin{lem}[The fluctuations $\FluRE(\eta)$]
\label{lem:flucttmsLPhis}
We have
\begin{equation} \label{flucttmsLPhis}
\FluRE(\eta) = - \int \log \tmsL \circ \Phis(x) (d\eta - dx) \preceq s \sum_{i = -\lambda}^{\lambda} |\muL|_{\1, \V(x)} \tD_i.
\end{equation}
\end{lem}
\begin{proof}[Proof of Lemma \ref{lem:flucttmsLPhis}]
We start by computing the derivative of $x \mapsto \log \tmsL \circ \Phis(x)$ as
$$
\left(\log \tmsL \circ \Phis\right)'(x) = \frac{\left[ \tmsL' \circ \Phis(x)\right] \Phis'(x)}{\tmsL \circ \Phis(x)}.
$$
Using the fact that $\Phis$ is bounded by $1$, that $\Phis'$ and $\frac{1}{\tmsL}$ are bounded, and that $\tmsL' = \muL'$, we obtain
$$
\left(\log \tmsL \circ \Phis\right)'(x) \preceq s |\muL|_{\1, \V(x)}.
$$
Moreover, since $\Phis$ is the identity near the endpoints, and $\tmsL = 1$ near the endpoints, the map $x \mapsto \log \tmsL \circ \Phis(x)$  is compactly supported. Applying Proposition \ref{prop:aprioribounds}, we obtain \eqref{flucttmsLPhis}.
\end{proof}

\begin{coro}[The term $\FluRE(\eta)$ is often small]
\label{coro:smallFluRE}
We have 
\begin{equation}
\label{smallFluRE}
\Esp \left[ \left| \FluRE(\eta) \right| \right] = s \pto.
\end{equation}
\end{coro}
\begin{proof}
In view of \eqref{flucttmsLPhis}, we use the discrepancy estimate \eqref{discrestimateA} and the estimates \eqref{muL1} on the first derivative of $\muL$. We obtain
\begin{equation*}
\Esp\left[ \sum_{i = -\lambda}^{\lambda} |\muL|_{\1, \V(x)} \tD_i \right] \preceq \sum_{i=0}^{2\l} \sqrt{\l} \frac{1}{\l^2} + \sum_{i=2\l}^{\lambda / 2} \sqrt{i} \frac{\l \sqrt{\lambda}}{i^3 \sqrt{\lambda}}
+ \sum_{i=\lambda/2}^{\lambda - \L} \sqrt{\lambda}  \frac{\l }{\lambda^{3/2}  (\lambda -i)^{3/2}} + \sum_{i = \lambda-\L}^{\lambda} \sqrt{\lambda} \frac{\l}{\lambda^{3/2} \L^{3/2}},
\end{equation*}
and thus
$$
\Esp\left[ \sum_{i = -\lambda}^{\lambda} |\muL|_{\1, \V(x)} \tD_i \right] \preceq \frac{1}{\sqrt{\l}} + \frac{\l}{\sqrt{\L} \lambda} = \pto.
$$
\end{proof}

\subsection{Conclusion}
Combining Proposition \ref{prop:mainterms}, Lemma \ref{lem:relativeentropy}, Corollary \ref{coro:smallFluRE}, we obtain:
\begin{equation}
\label{MainpREpFluRE}
\Esp \left[  \left| \Main \right| + \left| \FluRE \right| \right]  + \left|\RE\right| = s \pto + s^2 \pto,
\end{equation}
which, in view of Lemma \ref{lem:energycomparison}, says that the interior-interior interactions before and after transport are often very close.

\section{Comparison of the energies II: the interior-exterior interaction}
\subsection{The difference field}
\begin{defi}[The difference field]
Let $\eta$ be a point configuration in $(-\lambda, \lambda)$, and let $\etas$ be the push-forward of $\eta$ by $\Phis$.

For $x \notin (-\lambda, \lambda)$, we let $\DF(\eta)(x)$ be the electrostatic field created at $x$ by the difference $\etas-\eta$, i.e.
\begin{equation}
\label{def:Es}
\DF(\eta)(x) :=  \int - \log |x-y| \left( d\etas(y) - d\eta(y) \right).
\end{equation}
\end{defi}

\begin{lem}[Decomposition of the difference field]
\label{lem:difffield}
We have
\begin{equation}
\label{Eps1}
\DF(\eta)(x) = s\Gam(x) + s\ErrorLog(x) + \ErrorDF(\eta)(x),
\end{equation}
with $\Gam$ as in \eqref{def:Gam}, $\ErrorLog$ as in \eqref{def:ErrorLog}, and $\ErrorDF(\eta)$ defined by
\begin{equation}
\label{def:ErrorField}
\ErrorDF(\eta)(x) = \int \log \left( 1 - \frac{\psis(y)}{x-y} \right) \left(d\eta(y) - dy \right).
\end{equation}
\end{lem}
\begin{proof}
We simply write $\etas - \eta = (\etas - \tmsL) + (\tmsL - dy) + (dy - d\eta)$. We have
$$
\int - \log |x-y| (d\etas(y) - \tmsL(y) dy) = \int - \log |x - \Phis(y)| (d\eta(y) - dy),
$$
and we define $\ErrorDF(\eta)(x)$ as the term such that
\begin{equation}
\label{def:ErrorDFA}
\int - \log |x - \Phis(y)| (d\eta(y) - dy) = \int - \log |x - y| (d\eta(y) - dy) + \ErrorDF(\eta)(x),
\end{equation}
which coincides with the expression given in \eqref{def:ErrorField}. By definition, we obtain
$$
\DF(\eta)(x) = \int - \log |x-y| (\tmsL(y) dy  - dy) + \ErrorDF(\eta)(x).
$$
Since $\tmsL(y) = 1 + s \muL(y)$, the first term in the right-hand side is the logarithmic potential generated by $s \muL$, which is given by the sum $s \Gam + s \ErrorLog$. 
\end{proof}

\begin{lem}
Assume $x \geq \lambda$. We have
\begin{align}
\label{ErrorDFpaprime}
\ErrorDF(\eta)(x) & \preceq \sum_{i=-\lambda}^{\lambda} \left(\frac{|\psis(i)|}{(x-i)^2} + \frac{|\psis'(i)|}{(x-i)} \right) \DR_i \\
\label{ErrorDFprime}
\left(\ErrorDF(\eta)\right)'(x) & \preceq \sum_{i = -\lambda}^{\lambda} \left(\frac{|\psis(i)|}{(x-i)^3} + \frac{|\psis'(i)|}{(x-i)^2} \right) \DR_i. 
\end{align}
\end{lem}
\begin{proof}
Let us introduce the auxiliary function 
\begin{equation}
\label{def:Gxy} \Hh(x,y) := \log \left( 1 - \frac{\psis(y)}{x-y} \right),
\end{equation}
we re-write \eqref{def:ErrorField} as
$$
\ErrorDF(\eta)(x) = \int \Hh(x,y) \left(d\eta(y) - dy\right).
$$  
In particular, for $x \geq \lambda$, we may differentiate under the integral sign and get
$$
\left(\ErrorDF(\eta)\right)'(x) = \int \partial_x \Hh(x,y) \left(d\eta(y) - dy\right).
$$ 

Since $\psis$ vanishes near the endpoints, for any $x$ the function $\Hh(x, \cdot)$ is compactly supported with respect to the second variable. Moreover since $\psis(y) = 0$ for $y \geq \lambda - \L/10$, and $x \geq \lambda$, we may write 
$$
x- y - \psis(y) \approx x - y.
$$
A direct computation shows that
\begin{align}
\label{pyH} \partial_y \Hh(x,y) & \preceq  \frac{|\psis'(y)|}{|x-y|} + \frac{|\psis(y)|}{|x-y|^2} \\
\label{pyxH} \partial^2_{yx} \Hh(x,y) & \preceq  \frac{|\psis'(y)|}{|x-y|^2} + \frac{|\psis(y)|}{|x-y|^3},
\end{align}
as can be checked informally by treating $\psis$ as a perturbation, writing
$$
\Hh(x,y) \approx \frac{-\psis(y)}{x-y}, 
$$
and differentiating.

Using the a priori bound on fluctuations of Proposition \ref{prop:aprioribounds} with discrepancy $\DR_i$, and \eqref{pyH}, resp. \eqref{pyxH}, we obtain \eqref{ErrorDFpaprime}, resp. \eqref{ErrorDFprime}.
\end{proof}

\begin{coro}[The contribution of $\ErrorDF$ is often small]
\label{coro:ErrorDFpetit}
In particular, 
\begin{equation}
\label{petitDF}
\Esp \left[ \left| \int_{\Lambda^c} \ErrorDF(\C)(x) (d\C - dx)  \right| \right] = s \pto
\end{equation}
\end{coro}
We give the proof of Corollary \ref{coro:ErrorDFpetit} in Section \ref{sec:proofErrorDF}.

\subsection{The logarithmic potential and its fluctuations}
In this section, we consider the logarithmic potential $\Gam$ generated by $\mulambda$, as defined in \eqref{def:Gam}.

We state some bounds on $\Gam$ and its first derivative.
\begin{lem}[Controls on $\Gam$]
\label{lem:controlonGam}
\begin{align}
\label{Gambord} & \Gam(x) \preceq \frac{\l \log^2(\lambda)}{\lambda}, \quad |x| \in [\lambda /2, 2 \lambda] \\
\label{Gampprimepres} & \Gam'(x) \preceq \frac{\l}{ \lambda^{3/2} \sqrt{x-\lambda}},  \quad \lambda \leq |x| \leq 4 \lambda \\
\label{Gamprimeloin} & \Gam'(x) \preceq \frac{\l \log(\lambda)}{x^2}, \quad |x| \geq 4\lambda.
\end{align}
\end{lem}
We give the proof of Lemma \ref{lem:controlonGam} in Section \ref{sec:proofcontrolGam}.

\begin{lem}[Fluctuations of the logarithmic potential]
\label{lem:fluctGamvarphi}
We have
\begin{equation}
\label{Gammoinsvarphi}
 \int (\Gam(x) - \varphi(x)) (d\C - dx) \preceq \I + \II + \III,
\end{equation}
with 
\begin{align}
\label{GamI}
\I & \preceq  \sum_{i = \lambda + 2 \sqrt{\lambda}}^{4 \lambda} \frac{\l}{\lambda^{3/2} \sqrt{i - \lambda}} \DR_i + \sum_{i = 4\lambda}^{+\infty} \frac{\l \log(\lambda)}{i^2} \DR_i 
+ \sum_{i= \lambda + 2 \sqrt{\lambda}}^{\lambda + 4 \sqrt{\lambda}} \frac{\l \log^2(\lambda)}{\lambda^{3/2}} \DR_i, \\
\label{GamII}
\II & \preceq  \frac{\l \log^2(\lambda)}{\lambda^{3/2}} \sum_{|i| = \lambda - 4 \sqrt{\lambda}}^{\lambda - 2 \sqrt{\lambda}} \tD_i,  \\
\label{GamIII}
\III & \preceq \frac{\l \log^2(\lambda)}{\lambda} \left( \sqrt{\lambda} + |\Discr_{[\lambda - 4 \sqrt{\lambda}, \lambda + 4 \sqrt{\lambda}]}| + |\Discr_{[-\lambda - 4 \sqrt{\lambda}, -\lambda + 4 \sqrt{\lambda}]}|\right).
\end{align}
\end{lem}
\begin{proof}
Let $\chi_1$ be a smooth non-negative function such that 
\begin{equation}
\label{bonchi2}
\chi_1 = 1 \text{ on } [\lambda + 4\sqrt{\lambda}, +\infty)  \quad  \chi_1 = 0 \text{ on } (-\infty, \lambda + 2 \sqrt{\lambda}],
\end{equation}
with $|\chi_1|_{\0} \leq 1$ and $|\chi_1|_{\1} \preceq \frac{1}{\sqrt{\lambda}}$.

Let $\chi_2$ be a smooth non-negative function such that 
\begin{equation}
\label{bonchi33}
\chi_2 = 1 \text{ on } [-\infty, - \lambda - 4\sqrt{\lambda})  \quad  \chi_2= 0 \text{ on } [-\lambda - 2 \sqrt{\lambda}, + \infty),
\end{equation}
with $|\chi_2|_{\0} \leq 1$ and $|\chi_2|_{\1} \preceq \frac{1}{\sqrt{\lambda}}$.

Finally, let $\chi_3$ be a smooth non-negative function such that 
\begin{equation}
\label{bonchi3}
\chi_3 = 1 \text{ on } [-\lambda + 4 \sqrt{\lambda}, \lambda - 4 \sqrt{\lambda}]  \quad  \chi_3 = 0 \text{ outside } [- \lambda + 2 \sqrt{\lambda}, \lambda - 2 \sqrt{\lambda}],
\end{equation} 
with $|\chi_3|_{\0} \leq 1$ and $|\chi_3|_{\1} \preceq \frac{1}{\sqrt{\lambda}}$.

We write trivially $\Gam(x)$ as the sum
$$
\Gam(x) \chi_1(x) + \Gam(x) \chi_2(x) + \Gam(x) \chi_3(x) +  \Gam(x) (1 - \chi_1(x) - \chi_2(x) - \chi_3(x)),
$$
and we integrate these terms against $d\C - dx$.

\textbf{The $\chi_1, \chi_2$ terms.} We have
$$
\left(\Gam \chi_1\right)'(x) = \Gam'(x) \chi_1(x) + \Gam(x) \chi_1'(x),
$$
where $\chi_1'$ is supported on $[\lambda + 2 \sqrt{\lambda}, \lambda + 4 \sqrt{\lambda}]$ and bounded by $\Oun\left( \frac{1}{\sqrt{\lambda}} \right)$. Applying Proposition \ref{prop:aprioribounds}, we obtain
\begin{equation*}
\int \left(\Gam \chi_1\right) (d\C - dx) \preceq \sum_{i = \lambda + 2 \sqrt{\lambda}}^{+ \infty} |\Gam|_{\1, V(i)}  \DR_i \\
+ \sum_{i= \lambda + 2 \sqrt{\lambda}}^{\lambda + 4 \sqrt{\lambda}} |\Gam|_{\0, V(i)} \frac{1}{\sqrt{\lambda}},
\end{equation*}
and using \eqref{Gambord}, \eqref{Gampprimepres}, \eqref{Gamprimeloin} we may write
\begin{multline*}
\int \left(\Gam \chi_1\right) (d\C - dx) \preceq \sum_{i = \lambda + 2 \sqrt{\lambda}}^{4 \lambda} \frac{\l}{\lambda^{3/2} \sqrt{i - \lambda}} \DR_i + \sum_{i = 4\lambda}^{+\infty} \frac{\l \log(\lambda)}{i^2} \DR_i \\
+ \sum_{i= \lambda + 2 \sqrt{\lambda}}^{\lambda + 4 \sqrt{\lambda}} \frac{\l \log^2(\lambda)}{\lambda^{3/2}} \DR_i.
\end{multline*}
Of course, $\Gam \chi_2$ satisfies the same inequality, with $\DL_i$ instead of $\DR_i$, and this yields \eqref{GamI} (we only keep track of the “right-hand” term, the estimates on “left-hand” term are the same).

\textbf{The $\chi_3$ term.}
We have 
$$
\left((\Gam - \varphi) \chi_3\right)'(x) = (\Gam - \varphi)'(x) \chi_3(x) + (\Gam(x) - \varphi) \chi_3'(x),
$$
but we know by \eqref{hmuDSE} that $(\Gam - \varphi)' = 0$ on the support of $\chi_3$, and moreover $\varphi$ is supported outside the support of $\chi_3'$, so we have in fact
$$
\left((\Gam - \varphi) \chi_3\right)'(x) = \Gam(x) \chi_3'(x),
$$
which is supported on $[-\lambda + 2 \sqrt{\lambda}, - \lambda + 4 \sqrt{\lambda}] \cup [\lambda - 4 \sqrt{\lambda}, \lambda - 2 \sqrt{\lambda}]$. We use Proposition \ref{prop:aprioribounds} and \eqref{Gambord} and the fact that $|\chi_3|_{\1} \preceq \frac{1}{\sqrt{\lambda}}$ to get
\begin{equation*}
 \int \left(\Gam - \varphi\right) \chi_3 (d\C - dx)  \preceq \sum_{|i| = \lambda - 4 \sqrt{\lambda}}^{\lambda - 2 \sqrt{\lambda}} \tD_i \frac{\l \log^2(\lambda)}{\lambda} \frac{1}{\sqrt{\lambda}},
\end{equation*}
which yields \eqref{GamII}.

\textbf{The $1-\chi_1-\chi_2 - \chi_3$ term.} 
The function $1-\chi_1-\chi_2 - \chi_3$ is supported near the endpoints, on $[\lambda - 4 \sqrt{\lambda}, \lambda + 4 \sqrt{\lambda}]$ and on the symmetric interval. We use \eqref{Gambord} to get
\begin{multline*}
\int \Gam (1-\chi_1-\chi_2 - \chi_3) (d\C - dx) \\
\preceq \frac{\l \log^2(\lambda)}{\lambda} \left( \sqrt{\lambda} + |\Discr_{[\lambda - 4 \sqrt{\lambda}, \lambda + 4 \sqrt{\lambda}]}| + |\Discr_{[-\lambda - 4 \sqrt{\lambda}, -\lambda + 4 \sqrt{\lambda}]}|\right),
\end{multline*}
where the last parenthesis is, up to a multiplicative constant, a bound on the the number of the points in the intervals, and on the length of the intervals. This yields \eqref{GamIII}.
\end{proof}

\begin{coro}[The contribution of $\Gam - \varphi$ is often small]
\label{coro:GamVPhismall}
We have
\begin{equation}
\label{espGammoinsvarphi}
\Esp \left[ \left| \int (\Gam(x)- \varphi(x)) (d\C - dx) \right| \right] \preceq \frac{\l \log^2(\lambda)}{\sqrt{\lambda}}.
\end{equation}
In particular, 
\begin{equation}
\label{small:LPmoinsvarphi}  \Esp \left[ \left| \int (\Gam(x)- \varphi(x)) (d\C - dx) \right| \right] = \pto.
\end{equation}
\end{coro}
\begin{proof}
For $\I$, we use the estimate \eqref{discrestimateB} in the form 
$$
\Esp[\DR_i] \preceq \sqrt{|i- \lambda|},
$$ 
and we get
$$
\Esp[\I] \preceq \sum_{i = \lambda + 2 \sqrt{\lambda}}^{4 \lambda} \frac{\l}{\lambda^{3/2}} \frac{ \sqrt{i - \lambda}}{\sqrt{i - \lambda}} + \sum_{i = 4 \lambda}^{+ \infty} \l \log(\lambda) \frac{\sqrt{i}}{i^2} + \sum_{i = \lambda + 2 \sqrt{\lambda}}^{\lambda + 4 \sqrt{\lambda}} \frac{\l \log^2(\lambda)}{\lambda^{3/2}} \sqrt{i - \lambda},
$$
which, after computation, gives
$$
\Esp[\I] \preceq \frac{\l \log(\lambda)}{\sqrt{\lambda}}.
$$

For $\II$, we use the estimate \eqref{discrestimateB}, in the form
$$
\Esp[\tD_i] \preceq \sqrt{|i|},
$$ 
and we get
$$
\Esp[\II] \preceq \frac{\l \log^2(\lambda)}{\lambda^{3/2}} \lambda = \frac{\l \log^2(\lambda)}{\sqrt{\lambda}}.
$$

Finally, using the discrepancy estimate \eqref{discresA}, we have
$$
\Esp[\III] \preceq \frac{\l \log^2(\lambda)}{\sqrt{\lambda}}.
$$

The dominant error term is thus $\frac{\l \log^2(\lambda)}{\sqrt{\lambda}}$, which proves the result.
\end{proof}

\subsection{Fluctuations of the error on the logarithmic potential}
\begin{lem}[Fluctuations of $\ErrorLog$]
\label{lem:fluctuationsofErrorLog}
We have
\begin{equation}
\label{FluctErrorLogL}
\int \ErrorLogL (d\C - dx) \preceq \Aa + \Bb + \Cc,
\end{equation}
with  
\begin{align*}
\Aa & \preceq \left( \sum_{i = -\lambda + 3\L}^{+ \infty} + \sum_{i = -\infty}^{- \lambda - 3 \L} \right) \frac{\L^{3/2} \l}{\lambda^{3/2} (- \lambda -i)^2} \DL_i \\
\Bb & \preceq \sum_{i=-\lambda - 4 \L}^{-\lambda + 4 \L} \frac{\l \log(\lambda)}{\lambda^{3/2} \sqrt{\L}} \DL_i \\
\Cc & \preceq \left( \L + |\Discr_{[-\lambda - 4\L, - \lambda + 4 \L]}| \right) \frac{\l \sqrt{\L} \log(\lambda)}{\lambda^{3/2}},
\end{align*}
and similarly for $\ErrorLogR$, replacing $\DL_i$ by $\DR_i$.
\end{lem}

\begin{proof}[Proof of Lemma \ref{lem:fluctuationsofErrorLog}]
Let $\chi$ be a smooth non-negative function such that 
\begin{equation}
\label{bonchi}
\chi = 1 \text{ on } [-\lambda - 3\L , -\lambda + 3\L], \quad  \chi = 0 \text{ on } [-\lambda - 4\L, -\lambda + 4\L],
\end{equation}
with $|\chi|_{\0} \leq 1$ and $|\chi|_{\1} \preceq \frac{1}{\L}$. We write trivially
\begin{equation*}
\int \ErrorLogL (d\C - dx) = \int \ErrorLogL(x) \chi(x) (d\C - dx) 
+ \int \ErrorLogL(x) \left(1-\chi(x)\right) (d\C - dx).
\end{equation*}
We have
\begin{equation}
\label{deriveeproduit}
\left( \ErrorLogL \chi \right)^{'}(x) = \left(\ErrorLogL\right)^{'}(x) \chi(x) + \ErrorLogL(x) \chi'(x),
\end{equation}
where $\chi'$ is supported on 
$[-\lambda - 4\L, -\lambda - 3\L] \cup [-\lambda + 3\L, -\lambda + 4\L]$ and is bounded by $\Oun\left(\frac{1}{\l} \right)$.

Using Proposition \ref{prop:aprioribounds}, we obtain
\begin{multline}
\label{partieAP}
\int \ErrorLogL(x) \chi(x) (d\C - dx) \leq \left( \sum_{i=-\lambda+3\L}^{+ \infty} + \sum_{i = - \infty}^{-\lambda - 3\L} \right) |\ErrorLogL|_{\1, \V(i)} \DL_i,
\\
+ \sum_{i=-\lambda - 4\L}^{-\lambda + 4\L} |\ErrorLogL|_{\0} \frac{1}{\L} \DL_i.
\end{multline}

On the other hand, since $1-\chi$ is supported on $[-\lambda-4\L, -\lambda + 4\L]$, we have a trivial bound
\begin{equation}
\label{partieNaive}
\int \ErrorLogL(x) \left(1-\chi(x) \right)( d\C - dx) \preceq \left(\L + |\Discr_{[-\lambda-4\L, -\lambda + 4\L]}| \right) |\ErrorLogL|_{\0},
\end{equation}
where $\L + |\Discr_{[-\lambda-4\L, -\lambda + 4\L]}|$ is (up to a multiplicative constant) a bound on the number of the points in the interval, and on the length of the interval.

We let $\Aa$ be the first line of \eqref{partieAP}, $\Bb$ be the second line of \eqref{partieAP} and $\Cc$ be the right-hand side of \eqref{partieNaive}, and we use the bounds of Proposition \ref{prop:logpotmuL} to obtain \eqref{FluctErrorLogL}. 
\end{proof}

\begin{coro}[The contribution of $\ErrorLog$ is often small]
\label{coro:ErrorLogsmall}
 We have 
\begin{equation}
\label{espFluctErrorLog}
\Esp \left[ \int \ErrorLog (d\C - dx) \right] \preceq \frac{\l^{5/2} \log(\lambda)}{\lambda^{3/2}}. 
\end{equation}
In particular,
\begin{equation}
\label{small:ErrorLog} \Esp \left[ \int \ErrorLog (d\C - dx) \right] = \pto.
\end{equation}
\end{coro}
\begin{proof}
It is of course enough to prove \eqref{espFluctErrorLog} for $\ErrorLogL$.  We use the discrepancy estimate \eqref{discrestimateA} in the form $\Esp[\DL_i] \preceq \sqrt{|i + \lambda|}$, and  write
$$
\Aa \preceq \frac{\L^{3/2} \l}{\lambda^{3/2}} \sum_{i=3 \L}^{+ \infty} \frac{\sqrt{i}}{i^2} = \frac{\L \l}{\lambda^{3/2}}.
$$
Using again $\Esp[\DL_i] \preceq \sqrt{|i + \lambda|}$, we have
$$
\Bb \preceq \L \frac{\l \log(\lambda)}{\lambda^{3/2} \L^{1/2}} \L^{1/2} = \frac{\L \l \log(\lambda)}{\lambda^{3/2}}.
$$
Finally, we use the discrepancy estimate \eqref{discresA} to get
$$
\Cc \preceq \frac{\l \L^{3/2} \log(\lambda)}{\lambda^{3/2}},
$$
and this is the dominant term.
\end{proof}

\section{Proof of the central limit theorem}
\label{sec:finalproof}
\subsection{A good event}
\begin{lem}[Defining a good event]
\label{lem:Event}
For any point configuration $\C$, and $\Lambda = (-\lambda, \lambda)$ fixed, let us decompose $\C$ as $\C = \nu \cup \gamma_{\La^c}$, where
$$
\nu = \C \cap \Lambda, \quad \gamma_{\La^c} = \C \cap \La^c.
$$
We let $\nu_s$ be the push-forward of $\nu$ by $\Phis$. We will consider $\C$ and $\Cs$, where\footnote{In fact, since $\Phis$ is the identity outside $\Lambda$, $\Cs$ itself is also the push-forward of $\C$ by $\Phis$.}
$$
\Cs := \nu_s \cup \gamma_{\La^c}.
$$

There exists an $\Event$ satisfying
\begin{equation}
\label{EventOften}
\P\left(\C \in \Event\right) = 1 - \pto, \quad \P\left(\Cs \in \Event\right) = 1 - \pto,
\end{equation}
such that, if $\C \in \Event$
\begin{align*}
\Main(\nu) & = s \pto,\\
 \FluRE(\nu)  & = s \pto, \\
 \int_{\La^c} \ErrorDF(\nu)(x) (d\gamma_{\La^c} - dx)   & = s \pto, \\
 s \int \left(\Gam(x) - \varphi\right) (d \left[\nu_s \cup \gamma_{\La^c}\right] - dx)  & = s \pto, \\
s \int \ErrorLog(x) (d\left[\nu_s \cup \gamma_{\La^c}\right] -dx)  & = s \pto,
\end{align*}
and moreover
\begin{equation}
\label{Ccapll}
|\C \cap (-\l, \l)| \preceq \l^2.
\end{equation}
\end{lem}
\begin{proof}
The control \eqref{Ccapll} is needed for technical reasons, in order to ensure that the number of poins in $(-\l, \l)$ is bounded. Since the mean number of points is $2\l$, the event \eqref{Ccapll} is of course very likely.

Using Proposition \ref{prop:mainterms}, Corollary \ref{coro:smallFluRE}, Corollary \ref{coro:ErrorDFpetit}, Corollary \ref{coro:GamVPhismall} and Corollary \ref{coro:ErrorLogsmall}, and applying Markov's inequality, we see that there exists an event $E$ of probability $1 - \pto$ on which the three first bounds hold, and moreover
\begin{align*}
s \int \left(\Gam(x) - \varphi\right) (d \left[\nu \cup \gamma_{\La^c}\right] - dx) & = s \pto, \\
s \int \ErrorLog(x) (d\left[\nu \cup \gamma_{\La^c}\right] -dx) & = s \pto.
\end{align*}

Moreover, we argue that
$$
\P\left( \left( \nu_s \cup \gamma_{\La^c} \right) \in E \right) = 1 - \pto.
$$
Indeed we know, by construction, that the transport map $\Phis$ is close to the identity map, with $\Phis - \id$ bounded by $1$, see \eqref{bound:psis0}. So if $\Cs = \nu_s \cup \gamma_{\La^c}$ is the push-forward of $\C = \nu \cup \gamma_{\La^c}$ by $\Phis$, we have for any $x,y \in \R$
$$
|\Discr_{[x, y]}|[\Cs] \preceq |\Discr_{[x-1,y+1]}|(\C) + 1.
$$
Any estimate involving the discrepancies of $\C$ can thus be converted into the estimate on $\Cs$. We then take $\Event$ to be the intersection 
$$
E \cap \left\lbrace \Cs \in E\right\rbrace, 
$$
for which the last two bounds hold as stated.
\end{proof}

\subsection{The DLR equations}
\label{sec:DLRequations}
The DLR formalism for $\sineb$ is a statistical physics representation of the point process as an infinite volume Gibbs measure. Before stating the result of \cite{DLRpreprint} in a convenient fashion for the present paper, we need to introduce some notation. 

\begin{defi}[Infinite volume Gibbs kernel]
Let $\lambda  > 0$, and let $\Lambda : = [-\lambda, \lambda]$. Let $\gamma$ be a point configuration in $\R$, and $\eta$ be a point configuration in $\Lambda$. We aim at defining the energy of the point configuration $\eta \cup \gamma_{\La^c}$  formed by $\eta$ in $\La$ and $\gamma$ in $\La^c := \R \backslash \La$. In fact, we only want to \textit{compare} these energies for a fixed $\gamma$ and a variable $\eta$, so we may work up to (possibly infinite) additive constants, which formally disappear in the comparison. 

The interaction energy of $\eta$ with itself is denoted by $\HLa(\eta)$.\begin{equation}
\label{def:HLa}
\tHLa(\eta) := \hal \iint_{\diagc} - \log |x-y| (d\eta(x) - dx) (d\eta(y) - dy).
\end{equation}

The following quantity encodes the interaction energy of the configuration $\eta$ in $\La$ with the configuration $\gamma$ outside $\La$. In fact, we compute the interaction of $\eta - \gamma$ in $\La$ with $\gamma - dx$ outside $\La$. The first modification only plays the role of a (possibly infinite) additive constant (for fixed $\gamma$), and the second modification is technical.
\begin{equation}
\label{def:tMove}
\tMove(\eta, \gamma) := \lim_{p \to \infty} \int_{x \in ([-p,p] \backslash \La)} \int_{y \in \Lambda}  - \log |x-y| (d\eta(y) - d\gamma_{\La}(y)) (d\gamma(x) - dx).
\end{equation}

We denote Bernoulli point processes by $\B$. In particular, $\B_{|\gamma_\La|, \La}$ is the law of the Bernoulli point process with $|\gamma_{\La}|$ points in $\La$, i.e. the law of a random point configuration made of $|\gamma_{\La}|$ points drawn uniformly and independently in $\La$.

We may now form the Boltzmann factor associated to the sum of these energies, given by
$$
\exp \left( - \beta \left(\tHLa(\eta) + \tMove(\gamma, \eta) \right) \right),
$$
and the associated partition function
\begin{equation}
\label{def:ZLagamma}
  Z_{\La, \beta}(\gamma) := \int \exp \left(- \beta\left( \tHLa(\eta) + \tMove(\eta, \gamma) \right) \right)  d\B_{|\gamma_\La|, \La}(\eta).
 \end{equation}
 
Finally, for $\gamma$ fixed, we define $\GLa(\eta; \gamma)$ as a probability measure on random point configurations $\eta$ in $\La$ given by:
 \begin{equation}
  \label{def:GLa}
d\GLa(\eta ; \gamma) := \frac{1}{ Z_{\La, \beta}( \gamma)} \exp \left(- \beta\left( \tHLa(\eta) + \tMove(\eta, \gamma) \right) \right)  d\B_{|\gamma_\La|, \La}(\eta).
\end{equation}

\end{defi}

The following is a re-writing of the main result in \cite{DLRpreprint}.
\begin{prop}[DLR equations for $\sineb$]
\label{prop:DLRequations}
Let $f$ be a bounded, measurable function on the space of point configurations, and $\lambda > 0$, we have
\begin{equation}
\label{DLR}
\Esp[f] = \int d\sineb(\gamma) \int f(\eta \cup \gamma_{\Lambda^c}) d\GLa(\eta;  \gamma). 
\end{equation}
\end{prop}
\begin{proof}
The only difference with \cite{DLRpreprint} is that we chose here to include the background in the definition of the energy. The result of \cite{DLRpreprint} is stated with $\HLa$ and $\Move$ instead of $\tHLa$ and $\tMove$ respectively, where
$$
\HLa(\eta) := \hal \iint_{\diagc} - \log |x-y| d\eta(x) d\eta(y),
$$
$$
\Move(\gamma, \eta) := \lim_{p \to \infty} \int_{x \in ([-p,p] \backslash \La)} \int_{y \in \Lambda}  - \log |x-y| (d\eta(y) - d\gamma_{\La}(y)) d\gamma(x).
$$
It is easy to check that the difference between these two formulations is an additive constant (for fixed $\La, \gamma$), which is absorbed by the partition function, plus the term
$$
\lim_{p \to \infty} \int_{y \in [-p, p]} \int_{x \in \La} - \log |x-y| (d\eta(x) - d\gamma(x)) dy,
$$
which is almost surely zero.
\end{proof}

\subsection{The Laplace transform of the fluctuations}
\def \Lap{\mathcal{L}_{\varphi, \l, \lambda}}
\def \Lapp{\widetilde{\mathcal{L}}_{\varphi, \l, \lambda}}
We introduce the function $\Lap$
$$
t \mapsto \Lap(t) := \Esp \left[ \exp\left(s \Fluct[\varphi](\C) \1_{\Event}(\C)\right) \right],
$$
which is the Laplace transform of the fluctuations of $\varphi$, up to the indicator function $\1_{\Event}$, and, by construction, $\Event$ is very likely.

Using $\P(\Event) = 1 - \pto$ as stated in \eqref{EventOften}, we may of course re-write $\Lap$ as
$$
\Lap(t) = \Esp \left[ \exp\left(t \Fluct[\varphi](\C) \right) \1_{\Event}(\C) \right] + \pto,
$$
and we now focus on the first term in the right-hand side, that we denote by
\begin{equation}
\label{def:Lapp}
\Lapp(t) := \Esp \left[ \exp\left(t \Fluct[\varphi](\C) \right) \1_{\Event}(\C) \right].
\end{equation}

The map $\C \mapsto \exp\left(t \Fluct[\varphi](\C) \right) \1_{\Event}(\C)$ is bounded, because by construction, on $\Event$, the number of points of $\C$ in the support of $\varphi$ is bounded, see \eqref{Ccapll}. Using DLR equations \eqref{DLR}, we write
\begin{multline}
\label{Lappv2}
\Lapp(t) = \int d\sineb(\gamma) \frac{1}{Z_{\La, \beta}(\gamma)} \\
\times \int \exp\left(t \Fluct[\varphi](\eta) \right) \1_{\Event}\left(\eta \cup \gamma_{\La^c}\right)   \exp \left(- \beta\left( \tHLa(\eta) + \tMove(\eta, \gamma) \right) \right)  d\B_{|\gamma_\La|, \La}(\eta),
\end{multline}
where we have used the fact that, since $\varphi$ is supported inside $\La$, we may write 
$$
\Fluct[\varphi](\eta \cup \gamma_{\La^c}) = \Fluct[\varphi](\eta).
$$
Combining both exponential terms and using the definition \eqref{def:HLa} of $\tHLa$, we obtain, in the exponent
$$
-\beta \left( \hal  \left( \iint_{\diagc} - \log |x-y| (d\eta(x) - dx) (d\eta(y) - dy) - \frac{2t}{\beta} \Fluct[\varphi](\eta) \right) + \tMove(\eta, \gamma) \right),
$$
and we let 
\begin{equation}
\label{def:s}
s := \frac{t}{\beta}.
\end{equation}

\subsection{Laplace transform I. Energy splitting}
In view of the “energy splitting” identity stated in Lemma \ref{lem:energyplusfluct}, we may write
\begin{multline*}
\iint_{\diagc} - \log |x-y| (d\eta(x) - dx)(d\eta(y) - dy) - 2s \Fluct[\varphi](\eta) \\
=
\iint_{\diagc} - \log |x-y| (d\eta(x) - d\tmsL(x))(d\eta(y) - d\tmsL(y)) \\
+ 2 s \int_{\Lambda} \left(\Gam(x) - \varphi\right) (d\eta - dx) + 2s \int_{\Lambda} \ErrorLog(x) (d\eta -dx) \\
- 2 s^2 \ophiH^2 - s^2 \ErrorVar.
\end{multline*}
The term $\ErrorVar$ is bounded as in \eqref{erreur:variance}, hence we have
\begin{multline}
\label{exponentdevelop}
\exp\left(t \Fluct[\varphi](\eta) \right)  \exp \left(- \beta\left( \tHLa(\eta) + \tMove(\eta, \gamma) \right) \right) \1_{\Event}(\eta \cup \gamma_{\La^c}) \\
= 
\exp \left( - \beta \left( \hal \iint_{\diagc} - \log |x-y| (d\eta(x) - d\tmsL(x))(d\eta(y) - d\tmsL(y)) \right) \right)\\
\times \exp\left( - \beta \left(  s \int_{\Lambda} \left(\Gam(x) - \varphi\right) (d\eta - dx)  + s \int_{\Lambda} \ErrorLog(x) (d\eta -dx) + \tMove(\eta, \gamma) \right) \right) \\
\times \exp\left(  s^2 \beta \ophiH^2 \right)
\times \exp\left( s \pto + s^2 \pto \right) \1_{\Event}\left(\eta \cup \gamma_{\La^c}\right).
\end{multline}
Inserting this expansion into \eqref{Lappv2}, we obtain
\begin{multline}
\label{Lappv3}
\Lapp(t) = \exp\left( \frac{s^2}{2} \left(2\beta \ophiH^2\right) + s o_{\l, \lambda}(1) + s^2 o_{\l, \lambda}(1) \right) \int d\sineb(\gamma) \frac{1}{Z_{\La, \beta}(\gamma)} \\
\times \int \exp \left( - \beta \left( \hal \iint_{\diagc} - \log |x-y| (d\eta(x) - d\tmsL(x))(d\eta(y) - d\tmsL(y)) \right) \right)\\
\times \exp\left( - \beta \left(  s \int_{\Lambda} \left(\Gam(x) - \varphi\right) (d\eta - dx)  + s \int_{\Lambda} \ErrorLog(x) (d\eta -dx) + \tMove(\eta, \gamma) \right) \right) \\
\times  \1_{\Event}\left(\eta \cup \gamma_{\La^c}\right) d\B_{|\gamma_\La|, \La}(\eta).
\end{multline}
\def \hPhi{\hat{\Phi}_s}

\subsection{Laplace transform II. Change of variables}
We now perform a change of variables on $\eta$. For $N$ fixed, we may consider the map
$\hPhi : \La^{N} \to \La^{N} $ given by
$$
\hPhi(x_1, \dots, x_N) := (\Phis(x_1), \dots, \Phis(x_N)),
$$
where $\Phis$ is the transport map from the constant density to $\tmsL$. Since $\Phis$ is a bijection, so is $\hPhi$. We let $\nu = \hPhi^{-1}(\eta)$, so that 
$\eta = \hPhi(\nu)$ is the push-forward of $\nu$ by $\Phis$, that we will now denote by $\nu_s$. The innermost integral in \eqref{Lappv3} becomes
\begin{multline}
\label{innermost1}
\int \exp \left( - \beta \left( \hal \iint_{\diagc} - \log |x-y| (d\nu_s(x) - d\tmsL(x))(d\nu_s(y) - d\tmsL(y)) \right) \right)\\
\times \exp\left( - \beta \left(  s \int_{\Lambda} \left(\Gam(x) - \varphi\right) (d\nu_s - dx)  + s \int_{\Lambda} \ErrorLog(x) (d\nu_s -dx) + \tMove(\nu_s, \gamma) \right) \right) \\
\times \exp\left(\int \log \Phis'(x) d\nu(x) \right) \1_{\Event}\left(\nu_s \cup \gamma_{\La^c}\right) d\B_{|\gamma_\La|, \La}(\nu),
\end{multline}
where the term $\exp\left(\int \log \Phis'(x) d\nu(x)\right)$ is the Jacobian of the transformation. 

In view of \eqref{intphisprime}, we have
$$
\int \log \Phis'(x) d\nu(x) = \RE + \FluRE(\nu),
$$
we know from \eqref{smallRE} that $\RE = s^2 \pto$, and we know from Lemma \ref{lem:Event} that
$$
\nu_s \cup \gamma_{\La^c} \in \Event \implies \FluRE(\nu) = s \pto,
$$
hence the Jacobian only contributes to an error term $\exp(s\pto + s^2 \pto)$.

\subsection{Laplace transform III. The interior-interior energy}
Using Lemma \ref{lem:energycomparison}, we have
\begin{multline*}
\iint_{\diagc} - \log |x-y| (d\nu_s(x) - d\tmsL(x))(d\nu_s(y) - d\tmsL(y)) \\
= \iint_{\diagc} - \log |x-y| (d\nu(x) - dx)(d\nu(y) - dy) + \Main(\nu) + \RE + \FluRE(\nu).
\end{multline*}
We know from \eqref{smallRE} that $\RE = s^2 \pto$, and we know from Lemma \ref{lem:Event} that
$$
\nu_s \cup \gamma_{\La^c} \in \Event \implies \FluRE(\nu) = s \pto, \ \Main(\nu) = s \pto.
$$
We may thus write \eqref{innermost1} as 
\begin{multline}
\label{innermost2}
\int \exp \left( - \beta \left( \hal \iint_{\diagc} - \log |x-y| (d\nu(x) - dx)(d\nu(y) - dy) \right) \right)\\
\times \exp\left( - \beta \left(  s \int_{\Lambda} \left(\Gam(x) - \varphi\right) (d\nu_s - dx)  + s \int_{\Lambda} \ErrorLog(x) (d\nu_s -dx) + \tMove(\nu_s, \gamma) \right) \right) \\
\times \exp\left(s \pto + s^2 \pto \right) \1_{\Event}\left(\nu_s \cup \gamma_{\La^c}\right) d\B_{|\gamma_\La|, \La}(\nu).
\end{multline}

\subsection{Laplace transform IV. The interior-exterior energy}
Let us recall that $\tMove$ is defined in \eqref{def:tMove} by
$$
\tMove(\eta, \gamma) := \lim_{p \to \infty} \int_{x \in ([-p,p] \backslash \La)} \int_{y \in \Lambda}  - \log |x-y| (d\eta(y) - d\gamma_{\La}(y)) (d\gamma(x) - dx).
$$
A direct computation shows that
\begin{equation}
\tMove(\nu_s, \gamma) = \tMove(\nu, \gamma) + \int_{\La^c} \DF(\nu)(x) (d\gamma - dx),
\end{equation}
where $\DF(\nu)$ is the difference field generated by $\nu_s - \nu$ as in \eqref{def:Es}. Using the decomposition 
$$
\DF(\nu) = s\Gam(x) + s\ErrorLog(x) + \ErrorDF(\nu)(x), 
$$
as in \eqref{Eps1}, we may write
\begin{multline}
\label{tMovedecompo}
\tMove(\nu_s, \gamma) = \tMove(\nu, \gamma) + s \int_{\La^c} \Gam(x) (d\gamma - dx) \\
+ s \int_{\La^c} \ErrorLog(x)  (d\gamma - dx) + \int_{\La^c} \ErrorDF(\nu)(x) (d\gamma - dx),
\end{multline}
so in particular, the middle line in \eqref{innermost2} reads as
\begin{multline}
s \int_{\Lambda} \left(\Gam(x) - \varphi\right) (d\nu_s - dx)  + s \int_{\Lambda} \ErrorLog(x) (d\eta -dx) + \tMove(\nu_s, \gamma) \\
= \tMove(\nu, \gamma) + 
s \int \left(\Gam(x) - \varphi\right) (d \left[\nu_s \cup \gamma_{\La^c}\right] - dx)  + s \int \ErrorLog(x) (d\left[\nu_s \cup \gamma_{\La^c}\right] -dx)  \\+ \int_{\La^c} \ErrorDF(\nu)(x) (d\gamma - dx).
\end{multline}

We know from Lemma \ref{lem:Event} that
$$
\nu_s \cup \gamma_{\La^c} \in \Event \implies 
\left\lbrace \begin{array}{ccc}
\displaystyle{\int_{\La^c} \ErrorDF(\nu)(x) (d\gamma - dx)} & = & s \pto \\
\displaystyle{s \int \left(\Gam(x) - \varphi\right) (d \left[\nu_s \cup \gamma_{\La^c}\right] - dx)} & = & s \pto \\
\displaystyle{s \int \ErrorLog(x) (d\left[\nu_s \cup \gamma_{\La^c}\right] -dx)} & = & s \pto
\end{array}\right.,
$$
so we may re-write \eqref{innermost2} as
\begin{multline}
\label{innermost3}
\int \exp \left( - \beta \left( \hal \iint_{\diagc} - \log |x-y| (d\nu(x) - dx)(d\nu(y) - dy) \right) \right) \times \exp\left( - \beta \left( \tMove(\nu, \gamma) \right) \right) \\
\times \exp\left(s \pto + s^2 \pto \right) \1_{\Event}\left(\nu_s \cup \gamma_{\La^c}\right) d\B_{|\gamma_\La|, \La}(\nu).
\end{multline}

\subsection{Conclusion}
Using \eqref{innermost3} recognizing $\tHLa(\nu)$ in the first exponent (as defined in \eqref{def:HLa}, and coming back to the expression \eqref{Lappv3} of $\Lapp(t)$, we get
\begin{multline}
\label{Lappv4}
\Lapp(t) = \exp\left( \frac{s^2}{2} \left(2\beta \ophiH^2\right) + s o_{\l, \lambda}(1) + s^2 o_{\l, \lambda}(1) \right) \int d\sineb(\gamma) \frac{1}{Z_{\La, \beta}(\gamma)} \\
\times \int \exp \left( - \beta \left(\tHLa(\nu) + \tMove(\nu, \gamma) \right) \right) \times  \1_{\Event}\left(\nu_s \cup \gamma_{\La^c}\right) d\B_{|\gamma_\La|, \La}(\nu).
\end{multline}

By the DLR equations \eqref{DLR}, we may write
\begin{multline*}
\int d\sineb(\gamma) \frac{1}{Z_{\La, \beta}(\gamma)} \\
\times \int \exp \left( - \beta \left(\tHLa(\nu) + \tMove(\nu, \gamma) \right) \right) \times  \1_{\Event}\left(\nu_s \cup \gamma_{\La^c}\right) d\B_{|\gamma_\La|, \La}(\nu) = \Esp \left[ \1_{\nu_s \cup \gamma_{\La^c} \in \Event} \right] \\
 = \P \left[ \left\lbrace\nu_s \cup \gamma_{\La^c} \in \Event \right\lbrace \right],
\end{multline*}
and by Lemma \ref{lem:Event} this quantity is $1 - \pto$. Doing a final replacement of $s$ by $\frac{t}{\beta}$, we obtain
\begin{equation}
\label{Lappv5}
\Lapp(t) = \exp\left( \frac{t^2}{2} \times \frac{2}{\beta} \ophiH^2 +  t o_{\l, \lambda}(1) + t^2 o_{\l, \lambda}(1) \right) ( 1 - \pto).
\end{equation}

In particular, for $t$ such that $\frac{|t|}{\beta} \leq \smax$ as in \eqref{def:smax}, we get, uniformly in $t$,
\begin{equation}
\label{Lappv6}
\lim_{\lambda \to \infty, \l \to \infty} \Lap(t) =  \exp \left( \frac{t^2}{2} \times \frac{2}{\beta} \ophiH^2 \right),
\end{equation}

We have thus obtained that, sending $\lambda \to \infty$ then $\l \to \infty$, the Laplace transform of the random variable
$$
\Fluct[\varphi](\C) \1_{\Event}(\C)
$$
converges to 
$$
t \mapsto \exp\left( \frac{t^2}{2} \times \frac{2}{\beta}\ophiH^2  \right),
$$
which is the Laplace transform of a centered Gaussian variable with variance $ \frac{2}{\beta} \ophiH^2$. The convergence is uniform for values of the parameter in some open interval around $0$. It is well-known that this convergence implies convergence in law. Moreover, since we know by \eqref{EventOften} that $\P(\Event) = 1 - \pto$, the convergence in law of $\Fluct[\varphi](\C) \1_{\Event}(\C)$ implies the convergence in law of the fluctuations themselves. This concludes the proof of the central limit theorem.

\begin{remark}[Lack of moderate deviations bounds]
Since $|\muL|_{\0} \preceq \frac{1}{\l}$, taking $s$ of order as large as $\l$ still guarantees that $\tmsL$ will be a positive density. The transport map $\Phis$ may now move points at a distance $\Oun(\l)$, but in fact this is harmless because a careful inspection reveals that our estimates are insensitive to a displacement of the points of order $\l$. Taking $s$ large is tempting because it yields a control on the Laplace transform of the fluctuations for large values of the parameter, which in turn implies strong concentration bounds with exponential (in $\l$) tails. However, our argument relies on the discrepancy estimate \eqref{discrasym}, which is not quantitative and raises an obstacle for obtaining such moderate deviations bounds on the fluctuations.
\end{remark}

\section{Auxiliary proofs}
\label{sec:auxiliary}
\subsection{Proof of Proposition \ref{prop:aprioribounds}}
\label{sec:proofaprioribounds}
\begin{proof}[Proof of Proposition \ref{prop:aprioribounds}]
We write:
$$
\int  g(x) (d\C - dx) = \sum_{k=-\infty}^{\infty} \int_{k}^{k+1} g(x) (d\C - dx).
$$
Since $g$ is assumed to be compactly supported, all the sums are finite. On $[k, k+1]$ we may write, using the mean value theorem, $g(x) = g(k) + \Oun\left(|g|_{\1,\Vk}\right)$, and we obtain
$$
\int_{k}^{k+1} g(x) (d\C - dx) = g(k) \Discr_{[k, k+1]} + \Oun\left(|g|_{\1,\Vk}\right)\left(1 + |\Discr_{[k, k+1]}|\right).
$$
We have of course
\begin{equation}
\label{ofcourse}
\Discr_{[k, k+1]} = \Discr_{[0, k+1]} - \Discr_{[0, k]}, 
\end{equation}
so a summation by parts yields
$$
\sum_{k=-\infty}^{\infty} g(k) \Discr_{[k, k+1]} = \sum_{k = -\infty}^{\infty} \left(g(k-1) - g(k) \right) \Discr_{[0, k]}.
$$
Using the mean value theorem again, we get 
$$
\left|\sum_{k = -\infty}^{\infty} \left(g(k-1) - g(k) \right) \Discr_{[0, k]}\right| \leq \sum_{k = -\infty}^{\infty} |g|_{\1, \Vk} |\Discr_{[0, k]}|.
$$
We have thus obtained
$$
\int  g(x) (d\C - dx) \preceq \sum_{k=-\infty}^{\infty} |g|_{\1, \Vk} \left(|\Discr_{[0, k]}| + |\Discr_{[k, k+1]}| +1 \right),
$$
which yields the result. 

Finally, if $\lambda$ is fixed, we could choose to write, instead of \eqref{ofcourse}
$$
\Discr_{[k, k+1]} = \Discr_{[-\lambda, k+1]} - \Discr_{[-\lambda, k]},  \quad  \Discr_{[k, k+1]} = \Discr_{[k+1, \lambda]} - \Discr_{[k, \lambda]}
$$
so we can replace $\Dt$ by $\DL$ or $\DR$ as claimed.
\end{proof}

\subsection{Proof of Lemma \ref{lem:propofperturb}}
\label{sec:proofpropofperturb}
\begin{proof}[Proof of Lemma \ref{lem:propofperturb}] 
It is easy to check that $\HLP$ is bounded, and $x \mapsto \frac{1}{\sqrt{\lambda^2-x^2}}$ is integrable, thus so is $\mulambda$. Moreover, for any $x$ in $(-\lambda, \lambda)$, the map
$$
y \mapsto \log |x-y| \frac{1}{\sqrt{\lambda^2-x^2}} 
$$
is also integrable, hence the logarithmic potential is well-defined. 

The fact that $\mulambda$ has total mass $0$ follows from the well-known identity
$$
\PV \int \frac{1}{\sqrt{\lambda^2-x^2}} \frac{1}{t-x} dx = 0 \text{ for $t$ in $(-\lambda, \lambda)$,}
$$
which can be proven by elementary means, see e.g. \cite[Sec. 4.3, eq. (7)]{MR0094665}.

The fact that the logarithmic potential satisfies \eqref{hmuDSE} is a also a well-known result, which can be obtained by integrating the identity
\begin{equation}
\label{airfoil}
\PV \int \frac{\mu(t)}{t-x} dt = \varphi'(x),
\end{equation}
valid for any $x$ in $(-\lambda, \lambda)$. This is known as the \textit{airfoil equation} and we refer again to \cite[Sec.~4.3,~eq.~(12)]{MR0094665}.
\end{proof}

\subsection{Proof of Lemma \ref{lem:boundsonmulambda}}
\label{sec:proofboundsonmu}
We start by the following bounds concerning $\HLP$.
\begin{lem}[Bounds on $\HLP$ and its derivatives]
\label{lem:boundsonHLP} 
We have
\begin{align}
\label{HLP0}
& \HLP(x) \preceq
\begin{cases}
 \frac{\lambda}{\l} & |x| \leq 2\l, \\
\frac{\lambda \l}{x^2} & |x| \geq 2\l.
\end{cases} \\
\label{HLP1}
& \HLP^{(\1)}(x) \preceq
\begin{cases}
 \frac{\lambda}{\l^2} & |x| \leq 2\l, \\
\frac{\lambda \l}{x^3} & |x| \geq 2\l.
\end{cases} \\
\label{HLP2}
& \HLP^{(\2)}(x) \preceq
\begin{cases}
 \frac{\lambda}{\l^3} & |x| \leq 2\l, \\
\frac{\lambda \l}{x^4} & |x| \geq 2\l.
\end{cases}
\end{align}
with implicit constants depending on $\ophi$.
\end{lem}

\begin{proof}[Proof of Lemma \ref{lem:boundsonHLP}] 
We start with the following claim.
\begin{claim}
Let $g$ be a test function of class $C^1$, supported on $(-\l, \l)$. Then for any $x$ such that $|x| \leq 2\l$ we have:
\begin{equation}
\label{bound:CauchyPVLp}
\PV \int \frac{g(t)}{t-x} dt \preceq  \l^{1/2} \|g'\|_{L^{2}}.
\end{equation}
\end{claim}
\begin{proof}[Proof of the claim.]
Let $x$ be such that $|x| \leq 2\l$. Let us use the definition \eqref{def:CauchyPV} of the Cauchy principal value, and write
\begin{equation*}
\PV \int \frac{g(t)}{t-x} dt = \int_{0}^{+ \infty} \frac{g(x+u) -g(x-u)}{u} du = \int_{u \in I_x}  \frac{g(x+u) -g(x-u)}{u} du,
\end{equation*}
where $I_x$ is the set of positive real numbers $u$ such that $g(x+u)$ or $g(x-u)$ is not zero. This set depends on $x$, but since $g$ is supported on $(-\l, \l)$, the set $I_x$ is included in a union of intervals whose total length is bounded by $4 \l$. Using the elementary identity
$$
g(x+u) - g(x-u) = \int_{x-u}^{x+u} g'(v) dv,
$$
and applying Fubini's theorem, we get ($x-u \leq v \leq x+u$ is equivalent to $u \geq |x-v|$):
\begin{equation*}
\int_{u \in I_x} \frac{g(x+u) -g(x-u)}{u} du = \int_{u \in I_x} \frac{du}{u} \int_{x-u}^{x+u} g'(v) dv 
= \int_{-\l}^{\l} g'(v) dv \int_{u \geq |x-v|, u \in I_{x}} \frac{1}{u} du.
\end{equation*}

Since $u \mapsto \frac{1}{u}$ is decreasing, and $I_x$ has its length bounded by $4 \l$, the innermost integral satisfies
$$
\int_{u \geq |x-v|, u \in I_{x}} \frac{1}{u} du \leq \int_{|x-v|}^{|x-v| + 4 \l} \frac{1}{u} du = \log \left(\frac{4 \l + |x-v|}{|x-v|}\right),
$$
so we have
$$
\int_{-\l}^{\l} g'(v) dv \int_{u \geq |x-v|, u \in I_{x}} \frac{1}{u} du \leq \int_{-\l}^{\l} |g'(v)|\log \left(\frac{4 \l + |x-v|}{|x-v|}\right) dv.
$$

Applying Cauchy-Schwarz's inequality, we get
$$
\PV \int \frac{g(t)}{t-x} dt \leq \|g'\|_{L^{\2}} \left( \int_{-\l}^\l \left|\ln \left(\frac{4 \l + |x-v|}{|x-v|} \right) \right|^{2} dv \right)^{1/2}.
$$
A linear change of variables $w = \frac{x-v}{\l}$ shows that, for $|x| \leq 2\l$, we have
$$
\left( \int_{-\l}^{\l} \left|\ln \left(\frac{4 \l + |x-v|}{|x-v|} \right) \right|^{2} dv \right)^{1/2} \preceq \l^{1/2},
$$
which proves \eqref{bound:CauchyPVLp}.
\end{proof}

We recall that, by definition, 
\begin{equation}
\label{rappelHLP}
\HLP(x) = \frac{1}{\pi} \PV \int \frac{\phiL(t)}{t-x} dt,
\end{equation}
with $\phiL$
 defined as 
$$
\phiL : t \mapsto \sqrt{\lambda^2-t^2} \varphi'(t),
$$
and we compute the first derivatives of $\phiL$ as
\begin{equation}
\label{phiL1} \phiL^{(1)}(t) = \frac{1}{\pi} \frac{-t}{\sqrt{\lambda^2-t^2}} \varphi^{(1)}(t) + \sqrt{\lambda^2 - t^2} \varphi^{(2)}(t) 
\end{equation}
\begin{equation}
\label{phiL2} \phiL^{(2)}(t) = \frac{1}{\pi} \left( \frac{-1}{\sqrt{\lambda^2 -t^2}} - \frac{t^2}{(\lambda^2-t^2)^{3/2}} \right) \varphi^{(1)}(t) + \left( \frac{-2t}{\sqrt{\lambda^2 - t^2}} \right) \varphi^{(2)}(t) + \sqrt{\lambda^2-t^2} \varphi^{(3)}(t),
\end{equation}
and (this is the only moment where we need the $C^4$ regularity of $\varphi$):
\begin{multline}
\label{phiL3}
\phiL^{(\3)}(t) = \frac{1}{\pi} \left( - \frac{3t}{(\lambda^2-t^2)^{3/2}} - \frac{3t^3}{(\lambda^2 - t^2)^{5/2}} \right) \varphi^{(\1)}(t) \\
+ \left( \frac{-3}{\sqrt{\lambda^2 -t^2}} - \frac{t^2}{(\lambda^2-t^2)^{3/2}} - \frac{2t^2}{(\lambda^2-t^2)^{5/2}} \right) \varphi^{(\2)}(t) \\
+ \left( \frac{-3t}{\sqrt{\lambda^2-t^2}} \right) \varphi^{(\3)}(t) + \sqrt{\lambda^2-t^2} \varphi^{(\4)}(t).
\end{multline}

Let us observe that, for $\k \geq 1$, if $g$ is a test function of class $C^{\k +1}$, we have
\begin{equation}
\label{deriveeHilbert}
\left( \PV \int \frac{g(t)}{t - \cdot} dt \right)^{(\k)}(x) = \int_0^{+\infty} \frac{g^{(\k)}(x+u) - g^{(\k)}(x-u)}{u} du = \PV \int \frac{g^{(\k)}(t)}{t-x} dt.
\end{equation}

In view of \eqref{bound:CauchyPVLp}, \eqref{rappelHLP}, \eqref{deriveeHilbert}, we get that, for $|x| \leq 2\l$
$$
\HLP(x) \preceq \l^{1/2} \| \phiL^{(1)} \|_{L^2}, \quad \left(\HLP\right)^{(1)}(x) \preceq \l^{1/2} \| \phiL^{(2)} \|_{L^2}, \quad \left(\HLP\right)^{(2)}(x) \preceq \l^{1/2}  \| \phiL^{(3)} \|_{L^2}.
$$
Since $\varphi$ is supported in $(-\l, \l)$ and $\l$ satisfies $0 < \l < \frac{1}{10} \lambda$, it is easy to check, from \eqref{phiL1}, \eqref{phiL2}, that:
\begin{align*}
& |\phiL^{(\1)}(t)| \leq \frac{\l}{\lambda} |\varphi^{(\1)}(t)| + \lambda |\varphi^{(\2)}(t)|, \\
& |\phiL^{(\2)}(t)| \leq \frac{1}{\lambda} |\varphi^{(\1)}(t)| + \frac{\l}{\lambda} |\varphi^{(\2)}(t)| + \lambda |\varphi^{(\3)}(t)| \\
& |\phiL^{(\3)}(t)| \leq \frac{\l}{\lambda^3} |\varphi^{(\1)}(t)| + \frac{1}{\lambda} |\varphi^{(\2)}(t)| + \frac{\l}{\lambda} |\varphi^{(\3)}(t)|  + \lambda |\varphi^{(\4)}(t)|.
\end{align*}
Using finally the homogeneity bounds \eqref{bounds:rescale}, we see that the dominant term is the last one in each line, and we obtain the controls for $|x| \leq 2\l$ as in \eqref{HLP0}, \eqref{HLP1}, \eqref{HLP2}. We now turn to the case $|x| \geq 2\l$. 

\textbf{Bound on $\left(\HLP\right)^{(\1)}$.} Since $\varphi$ is supported on $(-\l, \l)$, so is $\phiL$, and for $|x| \geq 2\l$ the integral defining \eqref{def:HLP} (or its derivatives) can be understood in the standard sense as a Riemann integral. In particular, we have
$$
\HLP(x) = \frac{1}{\pi} \int \frac{\sqrt{\lambda^2- t^2} \varphi'(t)}{t-x} dt = \frac{1}{\pi} \int \frac{\lambda \left(1 + \Oun\left(\frac{\l^2}{\lambda^2}\right) \right) \varphi'(t)}{x\left(1+ \Oun\left(\frac{\l}{x}\right) \right)} dt
$$
The first-order term vanishes because $\int \varphi'(t) = 0$. We are left with
$$
\HLP(x) \preceq \frac{\lambda}{x} \left(\frac{\l^2}{\lambda^2} + \frac{\l}{x}  \right) \|\varphi'\|_{L^1} \preceq \frac{\lambda \l}{x^2} \|\varphi'\|_{L^1},
$$
which yields the control for $|x| \geq 2\l$ as in \eqref{HLP0}.

\textbf{Bound on the first derivative.}
To treat $\left(\HLP\right)^{(\1)}(x)$, we write it as
$$
\left(\HLP\right)^{(\1)}(x) = \frac{1}{\pi} \int \frac{\phiL'(t)}{t-x} dt = \frac{1}{\pi} \frac{1}{x} \int \phiL^{(1)}(t) \left( 1 - \frac{t}{x} + \Oun\left( \frac{\l^2}{x^2} \right) \right) dt.
$$
The first-order term vanishes because $\int \phiL^{(1)}
(t) = 0$. Using \eqref{phiL1}, we may thus write
$$
\left(\HLP\right)^{(\1)}(x) = \frac{1}{\pi} \frac{1}{x} \int \left( \frac{-t}{\sqrt{\lambda^2-t^2}} \varphi^{(1)}(t) + \sqrt{\lambda^2 - t^2} \varphi^{(\2)}(t) \right) \left( - \frac{t}{x} + \Oun\left( \frac{\l^2}{x^2} \right) \right) dt.
$$
First, we compute
\begin{multline}
\label{compu1}
\frac{1}{x} \int \frac{-t}{\sqrt{\lambda^2-t^2}} \varphi^{(1)}(t)\left( \frac{t}{x} + \Oun\left( \frac{\l^2}{x^2} \right)  \right) dt = \frac{1}{x} \int \frac{-t}{\sqrt{\lambda^2-t^2}} \varphi^{(1)}(t)  \Oun\left( \frac{\l}{x} \right) dt \\
\preceq \frac{\l^2}{\lambda x^2} \| \varphi \|_{L^1} \preceq \frac{\l^2}{\lambda x^2}.
\end{multline}
Next, we write
\begin{multline}
\label{compu2}
\frac{1}{x} \int \sqrt{\lambda^2 - t^2} \varphi^{(\2)}(t) \left( - \frac{t}{x} + \Oun\left( \frac{\l^2}{x^2} \right) \right) dt \\ = \frac{1}{x} \int \lambda \left(1 + \Oun \left(\frac{\l^2}{\lambda^2}\right) \right) \varphi^{(\2)}(t)   \left( - \frac{t}{x} + \Oun\left( \frac{\l^2}{x^2} \right)  \right) dt. 
\end{multline}
The first order term vanishes because $\int t \varphi^{(2)}(t) = 0$. We are left with
$$
\frac{\lambda}{x} \int \varphi^{(2)}(t) \left( \frac{\l^3}{\lambda^2 x} + \frac{\l^2}{x^2} \right) dt \preceq \frac{\lambda \l^2}{x^3} \| \varphi^{(2)} \|_{L^1} \preceq \frac{\lambda  \l }{x^3}. 
$$
Combining \eqref{compu1} and \eqref{compu2}, the dominant term is $\frac{\lambda  \l }{x^3}$ and we obtain the control on $\left(\HLP\right)^{(\1)}(x)$ for $|x| \geq 2\l$, as in \eqref{HLP1}.
\end{proof}

\textbf{Bound on $\left(\HLP\right)^{(\2)}$.}
The proof is similar to the one for $\left(\HLP\right)^{(\1)}$, except that we push the expansions to the next order, and use the fact that $\int \phi_{\lambda}^{(2)}(t) = 0$ and $\int t^2 \varphi^{(3)}(t) = 0$.

We may now give the proof of Lemma \ref{lem:boundsonmulambda}.
\begin{proof}[Proof of Lemma \ref{lem:boundsonmulambda}]
We compute
\begin{align*}
& \mulambda(x) = \frac{-1}{\pi} \left[ \frac{1}{\sqrt{\lambda^2 - x^2}} \HLP(x) \right], \\
& \mulambda^{(\1)}(x) = \frac{-1}{\pi} \left[ \frac{-x}{(\lambda^2-x^2)^{3/2}} \HLP(x) + \frac{1}{\sqrt{\lambda^2-x^2}} \HLP^{(\1)}(x) \right], 
\end{align*}
and the second derivative is given by
\begin{multline*}
\mulambda^{(\2)}(x) = \frac{-1}{\pi} \big[ \left(\frac{-1}{(\lambda^2-x^2)^{3/2}} - \frac{3x^2}{(\lambda^2-x^2)^{5/2}} \right) \HLP(x) - \frac{3x}{(\lambda^2-x^2)^{3/2}} \HLP^{(\1)}(x) \\ + \frac{1}{\sqrt{\lambda^2-x^2}} \HLP^{(\2)}(x)\big],
\end{multline*}
and we use \eqref{HLP0}, \eqref{HLP1}, \eqref{HLP2} together with the simple observation that
$$
\frac{1}{\sqrt{\lambda^2-x^2}} \preceq \frac{1}{\sqrt{\lambda}\sqrt{\lambda - |x|}},
$$
which allows for a slight simplification in the formulas. We obtain
\begin{align*}
&\mulambda(x) \preceq \frac{1}{\lambda} \frac{\lambda}{l}, \quad |x| \leq 2\l, \\
&\mulambda(x) \preceq \frac{1}{\sqrt{\lambda} \sqrt{\lambda - |x|}} \frac{\lambda \l}{x^2}, \quad |x| \geq 2\l, \\
& \mulambda^{(\1)}(x) \preceq \frac{\l}{\lambda^3} \frac{\lambda}{\l} + \frac{1}{\lambda} \frac{\lambda}{\l^2}, \quad |x| \leq 2\l, \\
& \mulambda^{(\1)}(x) \preceq \frac{|x|}{\lambda^{3/2} (\lambda - |x|)^{3/2}} \frac{\lambda \l}{x^2} + \frac{1}{\sqrt{\lambda} \sqrt{\lambda - |x|}} \frac{\lambda \l}{x^3}, \quad |x| \geq 2\l, \\
& \mulambda^{(\2)}(x) \preceq \left( \frac{1}{\lambda^3} + \frac{\l^2}{\lambda^5} \right) \frac{\lambda}{\l} + \frac{\l}{\lambda^3} \frac{\lambda}{\l^2} + \frac{1}{\lambda}\frac{\lambda}{\l^3}, \quad |x| \leq 2\l, 
\end{align*}
and, for $|x| \geq 2\l$
\begin{equation*}
\mulambda^{(\2)}(x) \preceq \left( \frac{1}{\lambda^{3/2} (\lambda - |x|)^{3/2}}  + \frac{x^2}{\lambda^{5/2} (\lambda - |x|)^{5/2}}\right) \frac{\lambda \l}{x^2} + \frac{|x|}{\lambda^{3/2} (\lambda - |x|)^{3/2}} \frac{\lambda \l}{x^3} 
+ \frac{1}{\sqrt{\lambda} \sqrt{\lambda - |x|}} \frac{\lambda \l}{x^4}.
\end{equation*}
We obtain \eqref{mulambda0}, \eqref{mulambda1}, \eqref{mulambda2}.
\end{proof}

\subsection{Two intermediate results}
\begin{lem}[A decomposition of $\phiL$]
\label{lem:phiLEr}
Let $\phiL : t \mapsto \sqrt{\lambda^2-t^2} \varphi'(t)$. We have
\begin{equation}
\label{phiLEr}
\phiL(t) = \lambda \varphi'(t) + \Er(t),
\end{equation}
where $\Er$ is a $C^1$ function, supported in $[-\l, \l]$, satisfying
\begin{equation}
\label{Ercont}
|\Er|_{\0} \preceq \frac{\l}{\lambda}, \quad |\Er|_{\1} \preceq \frac{1}{\lambda}.
\end{equation}
\end{lem}
\begin{proof}[Proof of Lemma \ref{lem:phiLEr}]
To see that \eqref{phiLEr} holds with \eqref{Ercont} we simply expand
$$
\sqrt{\lambda^2-t^2}\varphi'(t) = \lambda \varphi'(t) + \Oun\left( \frac{\l^2}{\lambda} \right) \varphi'(t),
$$
and since $|\varphi|_{\1} \preceq \frac{1}{\l}$, we obtain the first bound in \eqref{Ercont}. 

We may then compute
$$
\phiL'(t) = \frac{-t}{\sqrt{\lambda^2-t^2}} \varphi'(t) + \sqrt{\lambda^2-t^2} \varphi''(t) =  \Oun\left( \frac{\l}{\lambda} \right) \varphi'(t) + \lambda \varphi''(t) + \Oun\left( \frac{\l^2}{\lambda} \right) \varphi''(t),
$$
which yields the second bound in \eqref{Ercont}. 
\end{proof}

\begin{lem}[The integral of $\HLP$ on large intervals.]
\label{lem:integraledeH}
Let $a$ be in $[10 \l, \lambda/2]$. We have
\begin{align}
\label{integraledeH}
& \int_{-a}^{a} \HLP(y) dy  \preceq \frac{\l \lambda}{a}, \\
\label{integraledeyyH}
 & \int_{-a}^a y^2 |\HLP(y)| dy \preceq \l \lambda a.
 \end{align}
\end{lem}
\begin{proof}[Proof of Lemma \ref{lem:integraledeH}]
\textbf{Preliminary.}
 We use the definition \eqref{def:CauchyPV} of the Cauchy principal value, and write
$$
\HLP(y) = \frac{1}{\pi} \PV \int \frac{\varphi'(t) \sqrt{\lambda^2-t^2}}{y-t} dt = \int_{0}^{+\infty} \frac{ \phiL(y+u) - \phiL(y-u)}{u} du,
$$
where $\phiL : t \mapsto \sqrt{\lambda^2-t^2} \varphi'(t)$. We use Lemma \ref{lem:phiLEr} and decompose $\phiL$ as $\phiL = \lambda \varphi'(t) + \Er(t)$.
We may thus write
\begin{equation}
\label{yleqPV}
\HLP(y)  = \frac{\lambda}{\pi}  \int_{u=0}^{+ \infty} \frac{\varphi'(y+u) - \varphi'(y-u)}{u} du +  \int_{u=0}^{+ \infty} \frac{\Er(y+u) - \Er(y-u)}{u} du .
\end{equation}

The second term in the right-hand side of \eqref{yleqPV} can be bounded using \eqref{bound:CauchyPVLp}. We obtain
$$
\int_{u=0}^{+ \infty} \frac{\Er(y+u) - \Er(y-u)}{u} du = \PV \int \frac{\Er(t)}{y-t} dt \preceq \l^{1/2} \|\Er'\|_{L^2},
$$
and in view of the second inequality in \eqref{Ercont}, we get
\begin{equation}
\label{PVEr} \PV \int \frac{\Er(t)}{y-t} dt \preceq \frac{\l}{\lambda}.
\end{equation}

\textbf{Proof of \eqref{integraledeH}.}
We use \eqref{yleqPV} and write
\begin{equation}
\label{decompoaaHLP}
\int_{-a}^a \HLP(y) dy = \frac{\lambda}{\pi} \int_{-a}^a \int_{u=0}^{+ \infty} \frac{\varphi'(y+u) - \varphi'(y-u)}{u} du  +  \frac{1}{\pi} \int_{-a}^a \left(\PV \int \frac{\Er(t)}{y-t} dt\right) dy.
\end{equation}

We may bound the second term in the right-hand side of \eqref{decompoaaHLP}, using \eqref{PVEr}, as
\begin{equation}
\label{partieEr}
\int_{-a}^a \left(\PV \int \frac{\Er(t)}{y-t} dt\right) dy \preceq \frac{a \l}{\lambda}.
\end{equation}

We now turn to the first term in the right-hand side of \eqref{decompoaaHLP}. It can be expressed, using Fubini's theorem, as
\begin{multline*}
\lambda \int_{-a}^{a}  \int_{u=0}^{+ \infty} \frac{\varphi'(y+u) - \varphi'(y-u)}{u} du   \\
= \lambda \int_{0}^{+\infty} \frac{\varphi(u + a ) - \varphi(u-a) - \varphi(a-u) + \varphi(-a -u)}{u} du \\
= \lambda \left( \PV \int \frac{\varphi(t)}{a-t} dt - \PV \int \frac{\varphi(t)}{-a-t} dt\right).
\end{multline*}
Since $\varphi$ is supported on $(-\l, \l)$, and since $a \geq 10 \l$, these are standard Riemann integrals, and we have
$$
\PV \int \frac{\varphi(t)}{a-t} dt \preceq \frac{1}{a} \|\varphi\|_{L^1} \preceq \frac{\l}{a},
$$
and similarly for the other term. We get
\begin{equation}
\label{partiePhiL}
\lambda \int_{-a}^{a}  \int_{u=0}^{+ \infty} \frac{\varphi'(y+u) - \varphi'(y-u)}{u} du \preceq \frac{\l \lambda}{a}.
\end{equation}
Combining \eqref{partiePhiL} with \eqref{partieEr} (which is at most of the same order) yields \eqref{integraledeH}.

\textbf{Proof of \eqref{integraledeyyH}.}
We simply use the bounds of \eqref{HLP0}.
\end{proof}

\subsection{Proof of Lemma \ref{lem:approximatemu}}
\label{sec:proofapproximatemu}
\begin{proof}[Proof of Lemma \ref{lem:approximatemu}]
First, we define $\muL$ as $\mulambda$ on $[- \lambda + \L, \lambda - \L]$.

\textbf{On $[-\lambda + \L/2, - \lambda + \L]$.}
Let $S_a, S_b, S_c$ be three smooth, non-negative functions defined on $[-1,0]$ such that
\begin{align*}
& \forall \k \geq 0, S_a^{(\k)}(-1) = 0, \quad S_a(0) = 1, \quad \forall \k \geq 1, S_a^{(\k)}(0) = 0 \\
& \forall \k \geq 0, S_b^{(\k)}(-1) = 0, \quad S_b(0) = 1, \quad S_b^{(\1)}(0) = 1, \quad \forall \k \geq 2, S_b^{(\k)}(0) = 0 \\
& \forall \k \geq 0, S_c^{(\k)}(-1) = 0, \quad S_c(0) = 1, \quad S_c^{(\1)}(0) = 0, \quad S_c^{(\2)}(0) = 1,  \quad \forall \k \geq 2, S_b^{(\k)}(0) = 0.
\end{align*}
We let 
$$
\Pp = - \lambda + \L, \quad \size = \frac{\L}{2}
$$ 
and we define
$$
\D_0 = \mulambda\left(\Pp \right), \quad \D_1 = \mulambda^{(\1)}(\Pp), \quad \D_2 = \mulambda^{(\2)}(\Pp).
$$
We have, in view of \eqref{mulambda0}, \eqref{mulambda1}, \eqref{mulambda2}:
$$
\D_0 \preceq \frac{\l}{\lambda^{3/2} \L^{1/2}}, \quad \D_1 \preceq \frac{\l}{\lambda^{3/2} \L^{3/2}}, \quad \D_2 \preceq \frac{\l}{\lambda^{3/2} \L^{5/2}}.
$$
Finally, wet let $\Rr$ be the function
$$
\Rr(x) := \D_0 S_a\left( (x- \Pp) \frac{1}{\size} \right) S_b\left( (x- \Pp) \frac{\D_1}{\D_0} \right) S_c\left( (x- \Pp) \sqrt{\frac{\D_2}{\D_0}} \right).
$$
The function $\Rr$ is defined on $[\Pp - \size, \Pp] = [-\lambda + \L/4, - \lambda + \L]$, and on this interval we let $\muL(x) = \Rr(x)$.

By construction, the derivatives of order $\0, \1, \2$ of $\Rr$ and $\mulambda$ coincide at $\Pp$, so the piece-wise definition is $C^2$ at $\Pp$. Moreover, it can be checked that for $\k = \0, \1, \2$, we have:
\begin{equation}
\label{Rrk}
\Rr^{(\k)}(x) \preceq \frac{\l}{\L^{\k} \lambda^{3/2} \L^{1/2}}.
\end{equation}
This is easy to see for $\Rr^{(\0)}$, since $\D_0 \preceq \frac{\l}{\lambda^{3/2} \L^{1/2}}$. For the first and second derivatives, we use the fact that $\frac{1}{\size}, \frac{\D_1}{\D_0}$ and $\sqrt{\frac{\D_2}{\D_0}}$ have the same order $\frac{1}{\L}$.

\textbf{On $[-\lambda + \L/4, - \lambda + \L/2]$.}
Let $S_d$ be a smooth, non-negative function defined on $[-1, 0]$ such that 
$$
\forall \k \geq 0, S_d^{(\k)}(-1) = S_d^{(\k)}(0) = 0, \quad \int S_d(x) dx = 1.
$$
We overwrite the definition above and let now 
$$
\Pp := - \lambda + \L/2, \quad \size := \frac{\L}{4},
$$ 
and we introduce 
$$
\D_{-1} :=  \int_{-\lambda}^{-\lambda + \L} \mulambda(x) dx - \int_{-\lambda + \L/2}^{-\lambda + \L} \Rr(x) dx.
$$
Finally, we let $\Tt$ be the function 
$$
\Tt(x) =  \frac{\D_{-1}}{\size}  S_d\left( (x - \Pp) \frac{1}{\size} \right). 
$$
The function $\Tt$ is defined on $[-\lambda + \L/4, - \lambda + \L/2]$,
and on this interval we let $\muL(x) = \Tt(x)$. By construction, all the derivatives of $\Tt$ and $\Rr$ are equal (to $0$) at the point $\Pp$, so we still get a $C^2$ function.

Finally, we let $\muL(x) = 0$ on $[-\lambda, - \lambda + \L/4]$, and this connects with the previous definition in a $C^2$ way because all the derivatives of $\Tt$ vanish at $-\lambda + \L/4$. We define $\muL$ similarly near the other endpoint.

\textbf{Checking the statements.}
By construction, $\muL$ and $\mulambda$ coincide on a large interior part, and $\muL$ vanishes near the endpoints, so the first and fourth statements of the lemma are satisfied. Also by construction, we have
$$
\int \Tt(x) dx + \int \Rr(x)  = \int_{-\lambda}^{-\lambda + \L} \mulambda(x) dx,
$$
so the total masses of $\mulambda$ and $\muL$ are equal near the endpoints, and \eqref{massisconserved} holds.

We have already checked \eqref{bound:muL} for the first part of the construction, see \eqref{Rrk}. On $[-\lambda + \L/4, - \lambda + \lambda/2]$ we have
\begin{equation}
\label{muLkB}
|\Tt^{(\k)}(x)| \preceq \frac{1}{\size^{\k+1}} \D_{-1},
\end{equation}
and we observe that $\D_{-1}$ is of order $\size \times \D_0$, with $\D_0 \preceq \frac{\l}{\lambda^{3/2}\L^{1/2}}$, which yields the result.
\end{proof}

\subsection{Proof of Proposition \ref{prop:logpotmuL}}
\label{sec:logputmuL}
\begin{proof}[Proof of Proposition \ref{prop:logpotmuL}]
We recall that, by definition, we have
$$
\ErrorLog(x) : = \int - \log |x-y| (\muL- \mulambda(y))dy.
$$
We may split $\ErrorLog$ as the sum $\Errorlog(x) = \ErrorLogL(x) + \ErrorLogR(x)$, where $\ErrorLogL$ (resp. $\ErrorLogR$) is the contribution coming from the left (resp. right) endpoint, i.e.
\begin{align}
\label{def:ErrorLogL}
& \ErrorLogL(x) : = \int_{-\lambda}^{-\lambda + \L} \log |x-y| (\mulambda(y) - \muL(y))dy, \\
\label{def:ErrorLogR}
&  \ErrorLogR(x) : = \int_{\lambda - \L}^{\lambda} \log |x-y| (\mulambda(y) - \muL(y))dy.
\end{align}

\textbf{The sup norm, for $x$ close to the endpoints.}
Let us start with a rough bound:  
\begin{equation}
\label{ErrorLog0}
\ErrorLogL(x) \preceq  \int_{-\lambda}^{-\lambda + \L} \left|\log |x-y| \right| |\muL(y)| dy +  \int_{-\lambda}^{-\lambda + \L} |\log|x-y|| |\mulambda(y)| dy.
\end{equation}
Of course, if $x$ is far from the endpoints, this is sub-optimal because we do not use the fact that $\muL -\mulambda$ has mass zero, in fact we will use this inequality only for $x$ at distance $\Oun(\L)$ of an endpoint. Using \eqref{muL0}, we see that
\begin{equation}
\label{ErrorLog0muL}
\int_{-\lambda}^{-\lambda + \L} \left|\log |x-y| \right| |\muL(y)| dy \preceq \frac{\l \sqrt{\L} \log \lambda}{\lambda^{3/2}}.
\end{equation}
It remains to bound the second integral in \eqref{ErrorLog0}. We use \eqref{mulambda0} to write
$$
\int_{-\lambda}^{-\lambda + \L} |\log|x-y|| \mulambda(y) dy \preceq \int_{-\lambda}^{-\lambda + \L} \frac{\l \sqrt{\lambda} |\log|x-y||}{\lambda^2 \sqrt{\lambda - |y|}} dy \preceq \frac{\l}{\lambda^{3/2}} \int_{-\lambda}^{-\lambda + \L} \frac{|\log|x-y||}{\sqrt{\lambda - |y|}} dy.
$$
We use Hölder's inequality and write
$$
 \int_{-\lambda}^{-\lambda + \L}
 \frac{|\log|x-y||}{\sqrt{\lambda - |y|}} dy 
 \leq  
 \left(\int_{-\lambda}^{-\lambda + \L} |\log|x-y||^{3} \right)^{1/3} 
 \left(\int_{-\lambda}^{-\lambda + \L} \frac{1}{(\lambda - |y|)^{3/4}}\right)^{2/3} dy.
$$
An elementary computation shows that for $x$ in $(-2\lambda, 2\lambda)$
$\left(\int_{-\lambda}^{-\lambda + \L} |\log|x-y||^{3} \right)^{1/3} \preceq \L^{1/3} \log \lambda$, and that, on the other hand,
$\left(\int_{-\lambda}^{-\lambda + \L} \frac{1}{(\lambda - |y|)^{3/4}}\right)^{2/3} \preceq \L^{1/6}$, hence we obtain 
\begin{equation}
\label{ErrorLog0mu}
\int_{-\lambda}^{-\lambda + \L} |\log|x-y|| \mulambda(y) dy \preceq \frac{\l \sqrt{\L} \log \lambda}{\lambda^{3/2}}.
\end{equation}
Combining \eqref{ErrorLog0}, \eqref{ErrorLog0muL} and \eqref{ErrorLog0mu}, we obtain \eqref{bound:ErrorLog0}.

\textbf{The derivative, for $x$ far from the endpoints.}
We now turn to proving \eqref{bound:ErrorLog1}. Let $x$ such that $|x-\lambda| \geq 2\L$. We have, by definition,
$$
\ErrorLogR(x) = \int_{\lambda-\L}^{\lambda} - \log |x-y| (\muL(y) - \mulambda(y)) dy.
$$
We may differentiate under the integral sign and write
$$
\left(\ErrorLogR \right)'(x) = \int_{\lambda - \L}^{\lambda} \frac{-1}{x-y} (\muL(y) - \mulambda(y)) dy.
$$
A Taylor's expansion yields
\begin{multline}
\label{mulmuA}
\int_{\lambda - \L}^{\lambda} \frac{-1}{x-y} (\muL(y) - \mulambda(y)) dy
= \frac{1}{x-\lambda} \int_{\lambda-\L}^{\lambda } (\muL - \mulambda)(y) dy \\
+ \Oun\left( \frac{\L}{(\lambda-x)^2} \right) \int_{\lambda-\L}^{\lambda} (|\muL(y)| + |\mulambda(y)|)dy.
\end{multline}
The first term in the right-hand side of \eqref{mulmuA} vanishes because, by construction as in \eqref{massisconserved}, $\muL$ and $\mulambda$ have the same mass on $[\lambda - \L, \lambda]$. We can estimate the integral in the second term directly, and we obtain
\begin{equation*}
\int_{\lambda - \L}^{\lambda} \frac{-1}{x-y} (\muL(y) - \mulambda(y)) dy \preceq \frac{\L^{3/2} \l}{\lambda^{3/2} (\lambda - |x|)^2}.
\end{equation*}
The same argument holds near the other endpoint, which proves \eqref{bound:ErrorLog1}.
\end{proof}

\subsection{Proof of Lemma \ref{lem:variance}}
\label{sec:prooflemvariance}
\begin{proof}[Proof of Lemma \ref{lem:variance}]
Let us introduce $\Varphi$ as
\begin{equation}
\label{def:Varphi}
\Varphi := \frac{-1}{\pi^2} \int \frac{\varphi(x)}{\sqrt{\lambda^2-x^2}} \int\frac{\sqrt{\lambda^2-t^2} \varphi'(t)}{t-x} dt dx,
\end{equation}
which is equal to 
$$
\iint - \log |x-y| \mulambda(x) \mulambda(y) dx dy.
$$

\textbf{The error due to $\muL$.} The first step in the proof is to show that
\begin{equation}
\label{toVarphi}
\iint - \log |x-y| \muL(x) \muL(y) dx dy = \Varphi + \Oun\left( \frac{\l^2 \L \log(\lambda)}{\lambda^{3}} \right).
\end{equation}

We decompose the left-hand side of \eqref{energiemsLm} as
\begin{multline}
\label{decompositionenergietsmsLm}
\iint - \log |x-y| \muL(x) \muL(y) dx dy = 
\iint  - \log |x-y| \mulambda(x)\mulambda(y) dx dy \\
+ \iint - \log |x-y| (\muL - \mulambda)(x) (\muL - \mulambda)(y)
\\
 + 2\iint  - \log |x-y| \mulambda(y) (\muL - \mulambda)(x).
\end{multline}

Using the fact that $\mulambda$ satisfies \eqref{hmuDSE} and has total mass $0$ we may write
\begin{multline} \label{varianceA}
\iint  - \log |x-y| \mulambda(x)\mulambda(y) dx dy = \int \varphi(x) \mulambda(x) dx \\ = \frac{-1}{\pi^2} \int \frac{\varphi(x)}{\sqrt{\lambda^2-x^2}} \int\frac{\sqrt{\lambda^2-t^2} \varphi'(t)}{t-x} dt dx.
\end{multline}
which is equal to $\Varphi$ as defined in \eqref{def:Varphi}.

Using again \eqref{hmuDSE}, and the fact that, by construction, $\muL - \mulambda$ has total mass $0$ we write that
$$
\iint - \log |x-y| \mu(y) (\muL(x) - \mulambda(x)) dy dx  = \int \varphi(x) (\muL(x) - \mulambda(x)) dx, 
$$
but by construction $\varphi$ vanishes on the support of $\muL- \mulambda$, hence this is equal to $0$.

Finally, we write
\begin{multline*}
\iint - \log |x-y| (\muL(x) - \mulambda(x)) (\muL(y) - \mulambda(y))  dx dy \\
= \int \ErrorLog(x) (\muL(x) - \mulambda(x)) dx,
\end{multline*}
where $\ErrorLog = \ErrorLogL + \ErrorLogR$ as in Proposition \ref{prop:logpotmuL}. We can use \eqref{bound:ErrorLog0} and the fact that $\muL(x) - \mulambda(x)$ is supported near the endpoints of $(-\lambda, \lambda)$ to write
\begin{equation*}
\ErrorVar \preceq  \frac{\l \sqrt{L} \log(\lambda)}{\lambda^{3/2}}  \int_{-\lambda}^{-\lambda + \L} (|\muL(x)| + |\mulambda(x)|) dx \preceq \frac{\l \sqrt{\L} \log(\lambda)}{\lambda^{3/2}} \frac{\l \sqrt{\L}}{\lambda^{3/2}},
\end{equation*}
which yields \eqref{toVarphi}.

\textbf{The error due to $\lambda$ finite.}
Now, we compare $\Varphi$ with the norm $\ophiH$, we claim that:
\begin{equation}
\label{deVarphiaophi}
\Varphi = 2 \ophiH^2 + \Oun\left(\frac{\l^2}{\lambda^2}\right).
\end{equation} 
Indeed, we may write, since $\varphi$ is supported in $(-\l, \l)$
$$
 \int \frac{\varphi(x)}{\sqrt{\lambda^2-x^2}} \PV \int \frac{\sqrt{\lambda^2-t^2} \varphi'(t)}{t-x} dt dx = \int \varphi(x) \frac{1}{\lambda} \left( 1 + \Oun\left( \frac{\l^2}{\lambda^2} \right) \right) \HLP(x) dx, 
$$
 and we can use \eqref{HLP0} to write this as
 $$
 \frac{1}{\lambda} \int \varphi(x) \HLP(x) dx + \Oun\left( \frac{\l^2}{\lambda^2} \right).
 $$
 Now, we have, by definition, 
 $$
 \HLP(x) = \frac{1}{\pi}  \PV \int \frac{\phiL(t)}{x-t} dt, 
 $$
 and $\phiL$ admits the decomposition as in \eqref{phiLEr}, \eqref{Ercont}. It implies that
 $$
 \frac{1}{\lambda} \int \varphi(x) \HLP(x) dx = \frac{1}{\pi \lambda} \int \varphi(x) \PV \frac{\lambda \varphi'(t)}{x-t} dt dx + \Oun\left( \frac{\l^2}{\lambda^2} \right).
 $$
 We may thus write $\Varphi$ as
 $$
 \Varphi = \frac{-1}{\pi^2} \int \varphi(x) \PV \int \frac{\varphi'(t)}{x-t} dt dx + \Oun\left( \frac{\l^2}{\lambda^2} \right),
 $$
 and the result follows from the identity
 $$
 \frac{-1}{\pi^2} \int \varphi(x) \PV \int \frac{\varphi'(t)}{x-t} dt dx = 2 \ophiH^2,
 $$
with $\ophiH$ as in \eqref{def:ophiH}, which can be checked by elementary means.
\end{proof}

\subsection{Proof of Lemma \ref{lem:energyplusfluct}}
\label{sec:proofenergyplusfluct}
\begin{proof}[Proof of Lemma \ref{lem:energyplusfluct}]
For simplicity, we will use the notation  $\times$ as follows 
$$
A \times B = \iint_{\diagc} - \log |x-y| A(x) B(y).
$$ 
We have
\begin{multline*}
\left(d\C - \tmsL\right) \times \left(d\C - \tmsL \right) = \left(d\C - 1 - s \muL \right) \times \left(d\C - 1 - s \muL \right) \\
= \left(d\C - 1\right) \times (d\C - 1) + s^2 \cdot \muL \times \muL - 2 s \cdot \muL \times (d\C - 1).
\end{multline*}
By Lemma \ref{lem:variance}, we have
$$
\muL \times \muL = 2 \ophiH^2 + \ErrorVar. 
$$
Next, we write
$$
\muL \times (d\C - 1) = \mulambda \times (d\C - 1) + (\muL - \mulambda) \times (d\C -1).
$$
We recall that $\Gam$ is the logarithmic potential generated by $\mulambda$ and that $\ErrorLog$ is the  logarithmic potential generated by the difference $\muL- \mulambda$. So 
\begin{equation*}
\mulambda \times (d\C - 1) = \iint_{\diagc} -\log |x-y| \mulambda(y) (d\C(x) - dx) 
= \int_{\Lambda} \Gam(x) (d\C-dx),
\end{equation*}
and similarly
\begin{multline*}
(\muL - \mulambda) \times (d\C -1) = \iint_{\diagc} -\log |x-y| (\muL - \mulambda)(y) (d\C(x) - dx) \\
= \int_{\Lambda} \ErrorLog(x) (d\C-dx).
\end{multline*}
\end{proof}

\subsection{Proof of Lemma \ref{lem:transportmap}}
\label{sec:prooftransportmap}
\begin{proof}[Proof of Lemma \ref{lem:transportmap}]
Since $\muL$ is continuous and bounded as in Lemma \ref{lem:approximatemu}, and $\smax$ is chosen as in \eqref{def:smax}, we see that $1 + s \muL$ is a continuous, positive function on $\Lambda$. Consequently, $\Ms$ is $C^1$ and increasing, thus it is a $C^1$ bijection, and so is $\Phis$. The fact that $\Phis$ transports the constant density onto $\tmsL$ results from the definition, in fact $\Phis$ is the “monotone rearrangement” of the constant density onto $\tmsL$.

By construction, $\muL$ has total mass $0$ and vanishes near the endpoints, therefore
$\Ms(x) = x+\lambda$ near the endpoints, 
which implies that $\Phis$ coincides with the identity map near the endpoints.

We now turn to proving estimates on $\psi_s$. We may write, by definition, that for any $x$ in $[-\lambda, \lambda]$ we have
$$
\int_{-\lambda}^{\Phis(x)} \left(1 + s\muL(y)\right) dy = x + \lambda,
$$
and we thus obtain, as claimed in \eqref{PhisA0},
\begin{equation}
\label{PhisA}
\psis(x) = \Phis(x) - x = s \int_{-\lambda}^{\Phis(x)} \muL(y).
\end{equation}

\textbf{Bound on $\psis$.} We easily deduce $|\psis|_{\0} \leq s \|\muL\|_{L^1}$, and since $|s| \leq \smax$  as in \eqref{def:smax}, we have 
\begin{equation}
\label{psis0proof}
|\psis|_{\0} \leq 1.
\end{equation}
Finer bounds on $\psis$ are the goal of another lemma.

\textbf{Bound on $\psis^{(\1)}$.} Let us differentiate \eqref{PhisA} with respect to $x$:
$$
\Phis'(x) - 1 = -s \muL \circ \Phis(x) \cdot \Phis'(x),
$$
and we obtain
\begin{equation}
\label{PhisB}
\Phis'(x) = \frac{1}{1 + s \muL \circ \Phis(x)}.
\end{equation}
The numerator is bounded below by a positive constant, and a Taylor's expansion yields
$$
\left|\Phis'(x) - 1\right| \preceq \frac{1} s |\muL \circ \Phis(x)|,
$$
hence, since by definition $\psis' = \Phis' - 1$, we get
$$
\psis'(x) \preceq s |\muL \circ \Phis(x)|. 
$$
By definition $|\Phis(x) -x| = |\psis(x)|$, and \eqref{psis0proof} holds, we may thus write 
$$
|\muL \circ \Phis(x)| \leq \sup_{y \in [x-1, x+1]} |\muL(x)| \leq |\muL|_{\0, \Vx},
$$
with the notation of \eqref{def:Vdex}. This yields \eqref{bound:psis1}. In particular, $|\psis|_{\1} \preceq s|\muL|_{\0}$ and is thus bounded, and so is $\Phis'$.

\textbf{Bound on $\psis^{(2)}$.} We differentiate \eqref{PhisB} again and write
\begin{equation}
\label{PhisC}
\Phis^{(\2)}(x) = \frac{-s \muL' \circ \Phis(x) \Phis'(x)}{(1 + s\muL \circ \Phis(x))^2}.
\end{equation}
We have previously established that for $|s| \leq \smax$, we have $\Phis' \preceq 1$, and the quantity $(1 - s\muL \circ \Phis(x))$ is bounded above and below by a positive constant. We obtain
\begin{equation}
\label{prepsi2}
\psis^{(\2)}(x) = \Phis^{(\2)}(x) \preceq s \muL' \circ \Phis(x) \preceq s |\muL|_{\1, \Vx},
\end{equation}
which yields \eqref{bound:psis2}.

\textbf{Bound on $\psis^{(\3)}$.} Finally, differentiating \eqref{PhisC} again, we get
\begin{equation}
\label{PhisD}
\Phis^{(\3)}(x) = \frac{-s \muL^{(\2)} \circ \Phis(x) \left(\Phis'(x)\right)^2 - s \muL' \circ \Phis(x) \Phis^{(\2)}(x)}{(1+ s\muL \circ \Phis(x))^2} 
+ \frac{2s^2 \left(\muL' \circ \Phis(x)\right)^2 \left(\Phis'(x)\right)^2}{(1 + s\muL \circ \Phis(x))^3}.
\end{equation}
Using the fact that $\Phis'$ is bounded, that $\Phis^{(\2)}(x)$ is of order $s |\muL|_{\1, V(x)}$ (see \eqref{prepsi2}) and that the quantity $1 + s\muL \circ \Phis(x)$
is bounded below by a positive constant, we obtain
$$
\psis^{(\3)}(x) = \Phis^{(\3)}(x) \preceq s |\muL|_{\2, \Vx} + s^2 |\muL|^2_{\1, \Vx},
$$
and one can check from \eqref{muL1}, \eqref{muL2} that the dominant term in the right-hand side is the first one, which yields \eqref{bound:psi3}.
\end{proof}

\subsection{Proof of Lemma \ref{lem:additionalpropertiespsis}}
\label{sec:proofadditionalpropertiespsis}
\begin{proof}[Proof of Lemma \ref{lem:additionalpropertiespsis}]
The first inequality in \eqref{psisplusprecis} follows from \eqref{PhisA0} combined with \eqref{bound:psis0}, and the second one is obtained similarly, using the fact that $\muL$ has total mass $0$. We now turn to proving the inequalities of \eqref{psisplusprecis2}.

\textbf{The case $|x| \leq 10 \l$.}
Since $\|\muL\|_{L^1} \preceq 1$, as observed in \eqref{muLL1}, we have $|\psis|_{\0} \preceq s$, which in particular yields the bound for $|x| \leq 10\l$ as stated in \eqref{psisplusprecis2}. 

\textbf{The case $|x| \geq \lambda /2$.}
For $|x| \geq \lambda/2$, we may combine \eqref{psisplusprecis} with the estimates on $\muL$ as in \eqref{muL0}, and we obtain
$$
|\psis(x)| \preceq  s \frac{\l}{\lambda^{3/2}} \sqrt{\lambda + 1 - |x|}, 
$$
as stated in \eqref{psisplusprecis2}. 

\textbf{The case $10 \l \leq |x| \leq \lambda /2$.}
Finally, let us assume that $x$ is in $[10 \l, \lambda/2]$ (the case $x \in [-\lambda/2, -10\l]$ being, of course, similar). We may write
$$
\int_{-\lambda}^{\Phis(x)} \muL(t) dt = \int_{-\lambda}^{-\Phis(x)} \muL(t) dt + \int_{-\Phis(x)}^{\Phis(x)} \muL(t) dt.
$$
Using \eqref{muL0}, we see can write
\begin{equation}
\label{intlambdaPhis}
\int_{-\lambda}^{-\Phis(x)} \muL(t) dt \preceq \frac{\l^{3/2}}{\lambda^{3/2}} + \frac{\l}{\lambda} + \frac{\l}{\Phis(x)},
\end{equation}
and the dominant term is the last one. Next we write, for $|t|$ in $[10\l, \lambda/2]$,
$$
\muL(t) = \frac{1}{\lambda} \HLP(t) + \Oun\left( \frac{t^2}{\lambda^3} \right) |\HLP(t)|,
$$
and we use Lemma \ref{lem:integraledeH}. First, we apply \eqref{integraledeH} with $a = -\Phis(x)$ and $b = \Phis(x)$, we obtain
\begin{equation}
\label{appliintdeH}
\frac{1}{\lambda} \int_{-\Phis(x)}^{\Phis(x)} \HLP(t) dt \preceq \frac{\l}{\Phis(x)}.
\end{equation}
Secondly, we use \eqref{integraledeyyH} to get
$$
\frac{1}{\lambda^3} \int_{-\Phis(x)}^{\Phis(x)} t^2|\HLP(t)| dt \preceq \frac{\l \Phis(x)}{\lambda^2}. 
$$
The term in \eqref{appliintdeH} is the dominant one. Combining it with a similar one in \eqref{intlambdaPhis}, and since we know that $|\Phis(x) - x| \leq 1$, it yields, as desired
$$
|\psi(x)| \preceq \frac{\l}{|x|}.
$$
\end{proof}

\subsection{Proof of Proposition \ref{prop:mainterms}}
\label{sec:proofpropmainterms}
We extend the notation of \eqref{def:Vdex} as follows: if $g$ is a function of two variables, we let
$$
|g|_{V(x,y)} := \sup_{a \in V_x, b \in V_y} |g(x,y)|.
$$

We introduce the auxiliary function 
\begin{equation}
\label{def:F}
\Ff(x,y) := - \log \left( 1 + \tri(x,y) \right),
\end{equation}
so that, in view of definition \eqref{def:MainComp}, we have
$$
\Main(\eta) = \iint_{\Lala} \Ff(x,y) (d\eta(x)-dx) (d\eta(y)-dy).
$$

\begin{lem}[Energy comparison - the main term]
\label{prop:maincomparison}
We have
\begin{equation}
\label{bound:MainComp}
\Main(\eta) \preceq \MainZ(\eta) + \MainA(\eta) + \MainB(\eta) + \MainC(\eta) + \MainD(\eta),
\end{equation}
where the terms in the right-hand side are defined as
\begin{align*}
\MainZ(\eta) & = \sum_{i = - \lambda}^{\lambda} \sum_{j = -\lambda}^{\lambda} \tD_i \tD_j |\partial^2_{xy} F|_{V(i,j)}, \\
\MainA(\eta) & = \sum_{i=-\lambda}^{\lambda} \sum_{|j| = \lambda - \L}^{\lambda - \L/10} \tD_i \tD_j \frac{|\partial_x \Ff|_{V(i,j)}}{\L}, \\
\MainB(\eta) & = \sum_{|i| = \lambda - \L}^{\lambda - \L/10} \sum_{|j| = \lambda - \L}^{\lambda - \L/10} \tD_i \tD_j \frac{|\Ff|_{V(i,j)}}{\L^2}, \\
\MainC(\eta) & = \left( \L + \left|\Discr_{|x| \in [\lambda - \L/8, \lambda]}\right| \right) \sum_{j=-\lambda}^\lambda \sup_{|x| \in [\lambda - \L/8, \lambda]} |\partial_y \Ff|_{V(x,j)} \tD_j, \\
\MainD(\eta) & = \left( \L + \left|\Discr_{|x| \in [\lambda - \L/8, \lambda]}\right| \right) \sum_{|j| = \lambda - \L}^{\lambda - \L/10} \frac{\sup_{|x| \in [\lambda - \L/8, \lambda]} |\Ff|_{V(x,j)}}{\L} \tD_j. 
\end{align*}
\end{lem}

\begin{proof}
Let $\chi$ be a cut-off function equal to $1$ on $[-\lambda + \L/4, \lambda - \L/4]$, vanishing outside $[-\lambda + \L/8, \lambda - \L/8]$, bounded by $1$ and whose derivative is bounded by $\Oun\left(\frac{1}{\L}\right)$.
We may write:
\begin{multline}
\label{MainComp1}
\Main(\eta) = \iint_{\Lala} \chi(x) \Ff(x,y) \chi(y) (d\eta - dx) (d\eta -dy) \\
+ 2 \iint_{\Lala} \left(1-\chi(x)\right) \Ff(x,y) \chi(y)  (d\eta - dx) (d\eta -dy)\\
+ \iint_{\Lala} \left(1 - \chi(x)\right) \Ff(x,y) \left(1- \chi(y) \right) (d\eta - dx) (d\eta -dy).
\end{multline}
The last term in the right-hand side vanishes, because $\psis$ vanishes on the support of $1-\chi$, and so $\tri(x,y) = 0$, and thus $\Ff(x,y) = 0$, when both $x$ and $y$ belong to the support of $1 - \chi$. We now study the two first terms in the right-hand side of \eqref{MainComp1} separately.

\begin{claim}[The “$\chi$,$\chi$” term]
\label{claim:chiFchi}
We claim that:
\begin{multline}
\label{chiFchi}
\iint_{\Lala} \chi(x) \Ff(x,y) \chi(y) (d\eta - dx) (d\eta -dy) \\
\preceq \sum_{i = - \lambda}^{\lambda} \sum_{j = -\lambda}^{\lambda} \tD_i \tD_j  |\partial^2_{xy} \Ff|_{V(i,j)} 
+ \sum_{i=-\lambda}^{\lambda} \sum_{|j| = \lambda - \L}^{\lambda - \L/10} \tD_i \tD_j \frac{|\partial_x \Ff|_{V(i,j)}}{\L}
+ \sum_{|i| = \lambda - \L}^{\lambda - \L/10} \sum_{|j| = \lambda - \L}^{\lambda - \L/10} \tD_i \tD_j \frac{|\Ff|_{V(i,j)}}{\L^2}.
\end{multline}
\end{claim}
\begin{proof}[Proof of Claim \ref{claim:chiFchi}]
For a fixed configuration $\eta$, and $x$ in $(-\lambda, \lambda)$, let us define
\begin{equation}
\label{def:Geta}
\Gg_{\eta}(x) := \int \chi(y) \Ff(x,y) (d\eta(y) - dy).
\end{equation}
We have
\begin{equation}
\label{chichiGeta}
\iint_{\Lala} \chi(x) \Ff(x,y) \chi(y) (d\eta - dx) (d\eta -dy)  = \int \Gg_{\eta}(x) \chi(x) (d\eta - dx).
\end{equation}
By construction, the map $x \mapsto \Gg_{\eta}(x) \chi(x)$ is compactly supported. Using the \textit{a priori} bounds of Proposition \ref{prop:aprioribounds}, 
we obtain
\begin{equation*}
\iint_{\Lala} \chi(x) \Ff(x,y) \chi(y) (d\eta - dx) (d\eta -dy) 
\preceq \sum_{i = -\lambda}^{\lambda} \left| \Gg_{\eta} \chi \right|_{\1, \Vi} \tD_i.
\end{equation*}
We have of course, differentiating a product, 
$$
\left| \Gg_{\eta} \chi \right|_{\1, \Vi} \preceq  \left| \Gg_{\eta} \right|_{\1, \Vi} \left| \chi \right|_{\0, \Vi} +  \left| \Gg_{\eta} \right|_{\0, \Vi} \left| \chi \right|_{\1, \Vi}
$$
and we use the fact that $\chi$ is bounded by $1$, and that $\chi'(x)$ is bounded by $\L^{-1}$ and supported on $\{|x| \in [\lambda - \L/4, \lambda - \L/8]\}$. We obtain
\begin{equation}
\label{Getabound}
\int_{\Lambda} \Gg_{\eta}(x) \chi(x) (d\eta - dx) \preceq \sum_{i = -\lambda}^{\lambda} |\Gg_{\eta}|_{\1, \Vi} \tD_i + \sum_{|i| = \lambda - \L}^{\lambda - \L/2}\frac{|\Gg_{\eta}|_{\0, \Vi}}{\L} \tD_i.
\end{equation}

Let us now study $\Ggeta$ itself. We have
\begin{equation*}
\Ggeta(x) = \int \chi(y) \Ff(x,y) (d\eta(y) - dy), \quad \Ggeta'(x) = \int \chi(y) \partial_x \Ff(x,y) (d\eta(y) - dy).
\end{equation*}
We have of course, differentiating with respect to $y$ for $x$ fixed
$$
|\chi \Ff(x,\cdot)|_{1, \Vj} \preceq |\partial_y \Ff|_{V(x,j)} |\chi|_{0, \Vj}  + |\Ff|_{V(x,j)} |\chi|_{1, \Vj}, 
$$
and similarly
$$
|\chi \partial_x \Ff(x,\cdot)|_{1, \Vj} \preceq |\partial^2_{yx} \Ff|_{V(x,j)} |\chi|_{0, \Vj}  + |\partial_x \Ff|_{V(x,j)} |\chi|_{1, \Vj}.
$$
We use the \textit{a priori bounds} of Proposition \ref{prop:aprioribounds} again, and use again the fact that $\chi$ is bounded by $1$, that $\chi'(y)$ is zero outside $\{|y| \in [\lambda - \L/4, \lambda - \L/8]\}$ and bounded by $\L^{-1}$. We obtain
\begin{align}
\label{Geta} \Ggeta(x) & \preceq \sum_{j=-\lambda}^\lambda |\partial_y \Ff|_{V(x,j)} \tD_j + \sum_{|j| = \lambda - \L}^{\lambda - \L/10} \frac{|\Ff|_{V(x,j)}}{\L} \tD_j, \\
\label{Gpeta} \Ggeta'(x) & \preceq \sum_{j=-\lambda}^\lambda |\partial^2_{xy} \Ff|_{V(x,j)} \tD_j + \sum_{|j| = \lambda - \L}^{\lambda - \L/10} \frac{|\partial_x \Ff|_{V(x,j)}}{\L} \tD_j.
\end{align}

Combining \eqref{chichiGeta}, \eqref{Getabound} and \eqref{Geta}, \eqref{Gpeta}, we obtain the expression \eqref{chiFchi}.
\end{proof}

\begin{claim}[The “$\chi$,$(1-\chi)$” term]
\label{claim:chi1moinschi}
We claim that:
\begin{multline}
\label{chi1moinschi}
\iint_{\Lala} \left(1-\chi(x)\right) \Ff(x,y) \chi(y) (d\eta - dx) (d\eta -dy) \\ 
\preceq \left( \L + \left|\Discr_{|x| \in [\lambda - \L/8, \lambda]}\right| \right) \sum_{j=-\lambda}^\lambda \sup_{|x| \in [\lambda - \L/8, \lambda]} |\partial_y \Ff|_{V(x,j)} \tD_j 
\\
+ \left( \L + \left|\Discr_{|x| \in [\lambda - \L/8, \lambda]}\right| \right)  \sum_{|j| = \lambda - \L}^{\lambda - \L/10} \frac{\sup_{|x| \in [\lambda - \L/8, \lambda]} |\Ff|_{V(x,j)}}{\L} \tD_j.  
\end{multline}
\end{claim}
\begin{proof}[Proof of Claim \ref{claim:chi1moinschi}]
With the notation $\Ggeta$ of \eqref{def:Geta}, we write
\begin{equation}
\label{chi1moinschiA}
\iint_{\Lala} \left(1-\chi(x)\right) \Ff(x,y) \chi(y) (d\eta - dx) (d\eta -dy) = \int_{\Lambda} \left(1-\chi(x)\right) \Ggeta(x) (d\eta - dx).
\end{equation}

By construction, $1- \chi(x)$ is supported on 
$\{|x| \in [\lambda - \L/8, \lambda] \}$, so we have, using a rough bound on $\Ggeta$ and the mass of $d\eta - dx$ in $\{|x| \in [\lambda - \L/8, \lambda]\}$,
\begin{equation}
\label{chi1moinschiB}
\int_{\Lambda} \left(1-\chi(x)\right) \Ggeta(x) (d\eta - dx) \preceq \left( \L + \left|\Discr_{|x| \in [\lambda - \L/8, \lambda]}\right| \right) \sup_{|x| \in [\lambda - \L/8, \lambda]} \left| \Ggeta(x)\right|.
\end{equation}
Using \eqref{Geta} in \eqref{chi1moinschiB}, we obtain
\begin{multline*}
\int_{\Lambda} \left(1-\chi(x)\right) \Ggeta(x) (d\eta - dx)
\preceq \left( \L + \left|\Discr_{|x| \in [\lambda - \L/8, \lambda]}\right| \right) 
\sum_{j=-\lambda}^\lambda \sup_{|x| \in [\lambda - \L/8, \lambda]} |\partial_y \Ff|_{V(x,j)} \tD_j 
\\
+
\left( \L + \left|\Discr_{|x| \in [\lambda - \L/8, \lambda]}\right| \right)  \sum_{|j| = \lambda - \L}^{\lambda - \L/10} \frac{\sup_{|x| \in [\lambda - \L/8, \lambda]} |\Ff|_{V(x,j)}}{\L} \tD_j , 
\end{multline*}
which yields \eqref{chi1moinschi}.
\end{proof}

The estimate \eqref{bound:MainComp} is simply the combination of \eqref{MainComp1} and the two claims above.
\end{proof}

\begin{proof}[Proof of Proposition \ref{prop:mainterms}]
We recall that 
$$
\Ff(x,y) = - \log \left(1 + \frac{\psis(y)-\psis(x)}{y-x} \right).
$$
\begin{claim}[The magnitude of $\Ff$ and its derivatives]
\label{claim:magnitude}
We have
\begin{align}
\label{magni3} \Ff(x,y) & \preceq \frac{|\psis(x)| + |\psis(y)|}{|x-y|}   \\
\label{magni4}  \Ff(x,y)  & \preceq \sup_{t \in [x,y]} |\psis^{(\1)}(t)|   \\
\label{magni5} \partial_x \Ff(x,y) & \preceq \frac{|\psis'(x)|}{|y-x|} + \frac{|\psis(x)| + |\psis(y)|}{(y-x)^2}   \\
\label{magni6}\partial_x \Ff(x,y) & \preceq \sup_{t \in [x,y]} |\psis^{(\2)}(t)|.\\
\label{magni1}
\partial^2_{xy} \Ff(x,y) & \preceq   \frac{|\psis'(x)|}{(x-y)^2} + \frac{|\psis'(y)|}{(x-y)^2} + \frac{|\psis(x)|}{|x-y|^3} + \frac{|\psis(y)|}{|x-y|^3}, \\ 
\label{magni2} \partial^2_{xy} \Ff(x,y) & \preceq  \sup_{t \in [x,y]} |\psis^{(3)}(t)| + \sup_{t \in [x,y]} |\psis^{(2)}(t)|^2.
\end{align}
\end{claim}
\begin{proof}[Proof of Claim \ref{claim:magnitude}]
The bounds \eqref{magni3}, \eqref{magni4} are straightforward.

We then perform the following simple computation
\begin{align}
\label{pxF} \partial_x \Ff & = \frac{\partial_x \tri}{1 + \tri},\\
\label{p2xyFA}
\partial^2_{xy} \Ff & = \frac{ - \left(\partial^2_{xy} \tri\right) (1 + \tri) + \left(\partial_x \tri\right) \left(\partial_y \tri\right)}{(1 + \tri)^2}.
\end{align}

Moreover, we have
\begin{equation}
\label{partialx} \partial_x \tri(x,y) = \frac{- \psis'(x)}{y-x} + \frac{\psis(y) - \psis(x)}{(y-x)^2}, \quad  \partial_y \tri(x,y)  = \frac{ \psis'(y)}{y-x} - \frac{\psis(y) - \psis(x)}{(y-x)^2}.
\end{equation}
\begin{equation*}
\label{partialxy} \partial^2_{xy} \tri(x,y)  = \frac{\psis'(x) + \psis'(y)}{(y-x)^2} - 2 \frac{\psis(y) - \psis(x)}{(y-x)^3}.
\end{equation*}

From \eqref{pxF} and the fact that $1 + \tri$ is bounded below by a positive constant (because $|\tri|_{\0}$ is bounded by $|\psis'|_{\0}$, itself bounded by $s|\muL|_{\0}$, and $\smax$ is chosen as in \eqref{def:smax}) we see that
\begin{equation}
\label{usefultooo}
\partial_x \Ff \preceq \partial_x \tri, 
\end{equation}
and using \eqref{partialx} we obtain \eqref{magni5}.

Using again the fact that $1 + \tri$ is bounded below by a positive constant, we get
\begin{equation}
\label{usefultoo}
\partial^2_{xy} \Ff \preceq   - \left(\partial^2_{xy} \tri\right) (1 + \tri) + \left(\partial_x \tri\right) \left(\partial_y \tri\right),
\end{equation}
and after some algebra, we obtain
\begin{equation}
\label{p2xyFB}
\partial^2_{xy} \Ff(x,y) \preceq \frac{\psis'(x)}{(x-y)^2} + \frac{\psis'(y)}{(x-y)^2} - \frac{\psis'(x) \psis'(y)}{(x-y)^2} + \frac{\tri^2(x,y)}{(x-y)^2} - \frac{2\tri(x,y)}{(x-y)^2}.
\end{equation}
Since $\psis'$ is bounded, and so is $\tri(x,y) = \frac{\psis(x) - \psis(y)}{x-y}$, we may certainly write
\begin{equation*}
\partial^2_{xy} \Ff(x,y) \preceq \frac{|\psis'(x)|}{(x-y)^2} + \frac{|\psis'(y)|}{(x-y)^2} + \frac{|\psis(x)|}{|x-y|^3} + \frac{|\psis(y)|}{|x-y|^3},
\end{equation*}
which is \eqref{magni1}.

It remains to prove \eqref{magni6}, \eqref{magni2}. Using the identity 
$$
\tri(x,y) = \frac{1}{y-x} \int_{x}^y \psis'(s) ds,
$$ 
an elementary computation yields
\begin{equation}
\label{deriveesup}
\partial_x \tri(x,y)  \preceq \sup_{t \in [x,y]} \left|\psis^{(\2)}(t)\right|, \quad \partial_x \tri(x,y)  \preceq \sup_{t \in [x,y]} \left|\psis^{(\2)}(t)\right|, 
\partial^2_{xy} \tri(x,y)| \preceq \sup_{t \in [x,y]} \left|\psis^{(\3)}(t)\right|,
\end{equation}
We may then derive \eqref{magni6} from \eqref{usefultooo} and \eqref{deriveesup} and \eqref{magni2} from \eqref{usefultoo} and \eqref{deriveesup}.
\end{proof}

\paragraph{\textbf{General strategy, and convention for the proof.}}
We estimate the expectations of the all terms in Proposition \ref{prop:mainterms}. They involve (double) sums with coefficients of the type
$$
\tD_i \tD_j A(i,j),
$$
where $A(i,j)$ is a non-random quantity related to $\Ff$ or one of its derivatives. We will use the estimates of Claim \ref{claim:magnitude} to control the terms $A(i,j)$. Typically, the estimates \eqref{magni3}, \eqref{magni5}, \eqref{magni1} will be used when $i$ and $j$ are far away, and the estimates \eqref{magni4}, \eqref{magni6}, \eqref{magni2} will be used for $i$ and $j$ close to each other.

The expectation of $\tD_i \tD_j$ can be controlled using the discrepancy estimates \eqref{discrestimateA} and \eqref{discrestimateB}. Using Cauchy-Schwarz's inequality we see that\footnote{It is easy to check that the fact that, strictly speaking, the inequality is not true for $i= 0$ or $j=0$ is irrelevant.}
\begin{equation}
\label{CSmoyen}
\Esp\left[ \tD_i \tD_j \right] \preceq \sqrt{|i|} \sqrt{|j|}, 
\end{equation}
and we will replace all occurrences of $\tD_i$, resp. $\tD_j$ by $\sqrt{|i|}$, resp. $\sqrt{|j|}$. For most estimates, this is enough, and we obtain terms that are $\pto$.
A couple of terms are seen this way to be only bounded, but perhaps not vanishing, as $\lambda \to \infty, \l \to \infty$, which we denote by $O(1)$. For these terms, we use \eqref{discrestimateB} instead of \eqref{discrestimateA}, and write that
\begin{equation}
\label{mieuxCS}
\Esp\left[ \tD_i \tD_j \right] \preceq o_{|i| \to \infty} \left(\sqrt{|i|}\right) o_{|j| \to \infty} \left(\sqrt{|j|}\right),
\end{equation}
which allows us to improve the bound to $\pto$. 

\paragraph{\textbf{The term $\MainZ$.}}
We recall that
$$
\MainZ(\eta) = \sum_{i = - \lambda}^{\lambda} \sum_{j = -\lambda}^{\lambda} \tD_i \tD_j |\partial^2_{xy} F|_{V(i,j)}.
$$
Using symmetries, it is enough to study
$$
\sum_{i=0}^{\lambda} \sum_{i \leq |j| \leq \lambda} \tD_i \tD_j |\partial^2_{xy} F|_{V(i,j)}.
$$

\begin{enumerate}
\item Let us start with the region $0 < x < 2\l$ and $x < y < 4\l$. 
We want to prove that
$$
\Esp \left[ \sum_{i=0}^{2\l} \sum_{j = i}^{4\l} \tD_i \tD_j |\partial^2_{xy} F|_{V(i,j)} \right] = s \pto.
$$ 
In this case, since $i,j$ are close, we use \eqref{magni2} to control $|\partial^2_{xy} F|_{V(i,j)}$. By \eqref{bound:psis2}, \eqref{bound:psi3}, we know that $\psis^{(\2)}$ is controlled by $s \muL^{(\1)}$, and that $\psis^{(\3)}$ is controlled by $s \muL^{(\2)}$, and we refer to the bounds \eqref{muL1}, \eqref{muL2} to see that
$$
\sup_{|t| \leq 4\l} |\psis^{(\2)}(t)|^2 \preceq \frac{s^2}{\l^4} \quad \sup_{|t| \leq 4\l} |\psis^{(\3)}(t)| \preceq \frac{s}{\l^3},
$$
the dominant term is obviously the second one, so we may simply study
$$
\sum_{i=0}^{2\l} \sum_{j = i}^{4\l} \tD_i \tD_j \frac{s}{\l^3}.
$$
Taking the expectation and using \eqref{CSmoyen}, we are left with
$$
\sum_{i=0}^{2\l} \sum_{j = i}^{4\l} \sqrt{i} \sqrt{j} \frac{s}{\l^3} \preceq s \l^2 \times \l \times \frac{s}{\l^3} = s O(1).
$$
This is an example where the bound \eqref{CSmoyen} is not sufficient, and we replace it by \eqref{mieuxCS}. By well-known results on divergent series, we have
$$
\sum_{i=0}^{2\l} \sum_{j = i}^{4\l} o_{i}\left(\sqrt{i}\right) o_{j}\left(\sqrt{j}\right) = o_{l \to \infty} (\l^3), 
$$
and thus we have, as desired,
$$
\Esp \left[ \sum_{i=0}^{2\l} \sum_{j = i}^{4\l} \tD_i \tD_j \frac{1}{\l^3} \right] = s \pto.
$$
$$
\star \star \star
$$

\item For $2\l < x < \frac{\lambda}{2}$, $x < y < \frac{4}{3} x$. We study the expectation of
$$
\sum_{i=2\l}^{\lambda/2} \sum_{j = i}^{\qt i} \tD_i \tD_j |\partial^2_{xy} F|_{V(i,j)}.
$$
We use \eqref{magni2} to control $|\partial^2_{xy} F|_{V(i,j)}$. We control again $\psis^{(\2)}$ by $s\muL^{(\1)}$ (and read \eqref{muL1}), and $\psis^{(\3)}$ by $s\muL^{(\2)}$ (and read \eqref{muL2}), we get
$$
\sup_{t \in [x,y]} \left|\psis^{(\2)}(t) \right|^2 \preceq s^2 \frac{\l^2}{x^6}, \sup_{t \in [x,y]} \left| \psis^{(\3)}(t) \right| \preceq s\frac{\l}{x^4}.
$$
Since $x > \l$, the dominant term is the second one.

Finally, we take the expectation, use the discrepancy estimates and replace $\tD_i \tD_j$ by $\sqrt{i} \sqrt{j}$. Comparing the sum with an integral, we are left to study
$$
s \int_{2\l}^{\lambda/2} \int_{x}^{\frac{4}{3}x} \sqrt{x} \sqrt{y} \frac{\l}{x^4} dx dy.
$$
Replacing $\sqrt{y}$ by $\sqrt{x}$ (since $x < y < \frac{4}{3} x$), it yields
$$
s\int_{2\l}^{\lambda/2} \int_{x}^{\frac{4}{3}x} \sqrt{x} \sqrt{x}  \frac{\l}{x^4} dy dx = s\int_{2\l}^{\lambda/2} \frac{\l}{x^2} = sO(1).
$$
We are again in a case where \eqref{CSmoyen} is not enough and must be replaced by the discrepancy estimates \eqref{mieuxCS}, which improves the bound from $s O(1)$ to $s \pto$. 
$$
\star \star \star
$$

\item For $2\l < x < \frac{\lambda}{2}$, $\frac{4}{3} x < y$. We study the expectation of 
$$
\sum_{i=2\l}^{\lambda/2} \sum_{j = \qt i }^{\lambda} \tD_i \tD_j |\partial^2_{xy} F|_{V(i,j)}.
$$
Since $i$ and $j$ are far from each other, we use \eqref{magni1} to control $|\partial^2_{xy} F|_{V(i,j)}$. Taking the expectations, using \eqref{CSmoyen} and
comparing the sum with an integral, we are left to study
$$
\int_{2\l}^{\lambda/2} \sqrt{x} \int_{\frac{4}{3}x}^{\lambda} \sqrt{y} \left( \frac{|\psis'(x)|}{(x-y)^2} + \frac{|\psis'(y)|}{(x-y)^2} + \frac{|\psis(x)|}{(x-y)^3} + \frac{|\psis(y)|}{(x-y)^3} \right) dy.
$$

Since $\qt x < y$ we can replace $x-y$ by $y$, and we split the integrand in four parts.
\begin{enumerate}
\item Using \eqref{bound:psis1} to control $\psis'$ by $s\muL$, and \eqref{muL0} to control $\muL$, we get
$$
\int_{2\l}^{\lambda/2} \sqrt{x} |\psis'(x)| \int_{\frac{4}{3}x}^{\lambda} \frac{\sqrt{y}}{y^2} dy \preceq s \int_{2\l}^{\lambda /2} |\muL(x)| dx = s O(1).
$$
Using again \eqref{mieuxCS} instead of \eqref{CSmoyen}, we may replace $\sqrt{x}, \sqrt{y}$ by $o_x(\sqrt{x}), o_y(\sqrt{y})$, and we obtain in fact $s\pto$.

\item Using \eqref{bound:psis1} to control $\psis'$ by $s \muL$, and \eqref{muL0} to control $\muL$, and splitting the domain of integration in two parts, we see that
\begin{equation*}
\int_{2\l}^{\lambda/2} \sqrt{x} dx  \int_{\frac{4}{3}x}^{\lambda} \frac{\sqrt{y} |\psis'(y)|}{y^2} dy \preceq s \int_{2\l}^{\lambda /2} \sqrt{x} dx \left[\int_{\frac{4}{3}x}^{\lambda/2} \frac{\sqrt{y} \l}{y^{2} y^2} dy + \int_{\lambda/2}^{\lambda} \frac{\l \sqrt{y}}{\lambda^{3/2} y^2 \sqrt{\lambda - y}} dy  \right]
\end{equation*}
The first contribution is
$$
s\int_{2\l}^{\lambda/2} \sqrt{x} \int_{\frac{4}{3}x}^{\lambda/2} \frac{\sqrt{y} \l}{y^{2} y^2} dy dx = s\int_{2\l}^{\lambda/2} \sqrt{x} \int_{\frac{4}{3}x}^{\lambda/2} \frac{\l}{y^{7/2}} dy dx = s \int_{2\l}^{\lambda/2} \sqrt{x} \frac{\l}{x^{5/2}} dx = sO(1),
$$
and for the second one, since $y \in [\lambda/2,\lambda]$ we may replace $y$ by $\lambda$ and compute
$$
s \int_{2\l}^{\lambda/2} \sqrt{x} \int_{\lambda/2}^{\lambda} \frac{\l}{\lambda^{3} \sqrt{\lambda-y}} dy dx = s O(1).
$$
Again, this can be improved to $s \pto$.

\item Using \eqref{psisplusprecis2} to control $\psis(x)$, we have
$$
\int_{2\l}^{\lambda/2} \sqrt{x} |\psis(x)| \int_{\frac{4}{3}x}^{\lambda} \frac{\sqrt{y}}{y^{3}} dy  \preceq s \int_{2\l}^{\lambda/2} \sqrt{x} \frac{\l}{|x|} \frac{1}{x^{3/2}} dx = s O(1).
$$
Which can be improved to $\pto$.

\item Using \eqref{psisplusprecis2} to control $\psis(y)$ and splitting the domain of integration (on $y$) in two parts, we have
\begin{multline*}
\int_{2\l}^{\lambda/2} \sqrt{x}  dx \int_{\frac{4}{3}x}^{\lambda} \frac{\sqrt{y} |\psis(y)|}{y^{3}} dy \preceq \int_{2\l}^{\lambda/2} \sqrt{x}  dx \left[ \int_{\frac{4}{3}x}^{\lambda/2} \frac{\l}{y^{7/2}} dy + \int_{\lambda/2}^{\lambda} \frac{\l}{\lambda^{3/2}} \frac{\sqrt{\lambda - y}}{\lambda^{7/2}} dy \right] 
\\
\preceq s \int_{2\l}^{\lambda/2} \sqrt{x}  \left[ \frac{\l}{x^{5/2}} + \frac{\l}{\lambda^{5/2}} \right] dx = s O(1).
\end{multline*}
Similarly, this can be improved to $s \pto$.
\end{enumerate}
$$
\star \star \star
$$
\item For $\frac{\lambda}{2} < x < \lambda - \L$ and $y-x < \frac{1}{2} (\lambda - x)$.
We study the expectation of 
$$
\sum_{i=\lambda/2}^{\lambda - \L} \sum_{0 <j-i < \frac{1}{2} (\lambda - i)} \tD_i \tD_j |\partial^2_{xy} F|_{V(i,j)}.
$$
Since $i,j$ are close we use \eqref{magni2} to control $|\partial^2_{xy} F|_{V(i,j)}$. We see (the now usual way) that
$$
\sup_{t \in [x,y]} \left| \psis^{(2)}(t) \right|^2 \preceq s^2 \frac{\l^2}{\lambda^{3} (\lambda - x)^{3}}, \quad \sup_{t \in [x,y]} \left| \psis^{(3)}(t) \right| \preceq s \frac{\l}{\lambda^{3/2} (\lambda - x)^{5/2}},
$$
and the dominant term is the second one. 

We take the expectation, we use the discrepancy estimates, we compare the sum to an integral, we replace $\sqrt{x}, \sqrt{y}$ by $\sqrt{\lambda}$, and we are left to compute
\begin{equation*}
s \int_{\lambda/2}^{\lambda - \L} \lambda dx \int_{0 < y-x < \frac{1}{2} (\lambda - x)} \frac{\l}{\lambda^{3/2} (\lambda - x)^{5/2}}  dy \preceq s \int_{\lambda/2}^{\lambda - \L} \frac{\l}{\sqrt{\lambda}} \frac{1}{(\lambda -x)^{3/2}} \\
 \preceq \frac{\l}{\sqrt{\lambda}\sqrt{\L}} = s \pto.
\end{equation*}
$$
\star \star \star
$$

\item For $\frac{\lambda}{2} < x < \lambda - \L$ and $y-x > \frac{1}{2} (\lambda - x)$. We use \eqref{magni1} to control $|\partial^2_{xy} F|_{V(i,j)}$. We take the expectation, and compare the sum to a series, we are left to study
$$
\int_{\lambda/2}^{\lambda} dx \sqrt{x} \int_{y-x > \frac{\lambda -x}{2}} \sqrt{y} \left( \frac{|\psis'(x)|}{(x-y)^2} + \frac{|\psis'(y)|}{(x-y)^2} + \frac{|\psis(x)|}{(x-y)^3} + \frac{|\psis(y)|}{(x-y)^3} \right) dy.
$$
We replace $\sqrt{x}, \sqrt{y}$ by $\sqrt{\lambda}$ and split the integrand in four parts.
\begin{enumerate}
\item Using \eqref{bound:psis1} to control $\psis'(x)$ by $s\muL(x)$, and \eqref{muL0}, we have
$$
|\psis'(x)| \preceq s\frac{\l}{\lambda^{3/2} \sqrt{\lambda-x}}, 
$$
and thus consider
\begin{multline*}
\int_{\lambda/2}^{\lambda-\L} dx \sqrt{\lambda} \int_{y-x > \frac{\lambda -x}{2}} \sqrt{\lambda} \frac{|\psis'(x)|}{(y-x)^2} \preceq s \int_{\lambda/2}^{\lambda} dx \lambda \frac{\l}{\lambda^{3/2} \sqrt{\lambda-x}} \int_{y-x > \frac{\lambda -x}{2}} \frac{1}{(y-x)^2} dy \\
\preceq s \int_{\lambda/2}^{\lambda-\L} dx \lambda \frac{\l}{\lambda^{3/2} \sqrt{\lambda - x}} \frac{1}{\lambda-x} \preceq s \frac{\l}{\sqrt{\lambda} \sqrt{\L}} = s \pto.
\end{multline*}
\item Using \eqref{bound:psis1} to control $\psis'(y)$ by $s\muL(y)$, and \eqref{muL0}, we have
$$
|\psis'(y)| \preceq s\frac{\l}{\lambda^{3/2} \sqrt{\lambda-y}}, 
$$
and thus consider
\begin{equation*}
\int_{\lambda/2}^{\lambda-\L} dx \sqrt{\lambda} \int_{y-x > \frac{\lambda -x}{2}} \sqrt{\lambda} \frac{|\psis'(y)|}{(y-x)^2} dy \preceq s \int_{\lambda/2}^{\lambda-\L} dx \lambda \int_{y-x > \frac{\lambda -x}{2}} \frac{\l}{\lambda^{3/2} \sqrt{\lambda - y} (y-x)^2} dy.
\end{equation*}
Since $y-x > \frac{\lambda -x}{2}$, we may replace $\frac{1}{(y-x)^2}$ by $\frac{1}{(\lambda - x)^2}$, and we now study
\begin{multline*}
s \int_{\lambda/2}^{\lambda-\L} dx \frac{\lambda}{(\lambda-x)^2} \int_{x + \frac{\lambda-x}{2}}^{\lambda} \frac{\l}{\lambda^{3/2} \sqrt{\lambda-y}} dy \preceq s \int_{\lambda/2}^{\lambda-\L} dx \frac{\lambda}{(\lambda-x)^2}  \frac{\l}{\lambda^{3/2}} \sqrt{\lambda-x} dx\\
 \preceq s \frac{\l}{\sqrt{\lambda} \sqrt{\L}} = s \pto.
\end{multline*}

\item Using \eqref{psisplusprecis2} to control $\psis(x)$ by $s \frac{\l \sqrt{\lambda-x}}{\lambda^{3/2}}$, we write
\begin{multline*}
\int_{\lambda/2}^{\lambda-\L} dx \sqrt{\lambda} \int_{y-x > \frac{\lambda -x}{2}} \sqrt{\lambda}  \frac{|\psis(x)|}{(y-x)^3} dy \preceq s \int_{\lambda/2}^{\lambda-\L} dx \lambda \frac{\l \sqrt{\lambda -x}}{\lambda^{3/2}} \int_{|y-x| > \frac{\lambda - x}{2}} \frac{1}{(y-x)^3} dy \\
\preceq s \int_{\lambda/2}^{\lambda-\L} \lambda \frac{\l \sqrt{\lambda -x}}{\lambda^{3/2}} \frac{1}{(\lambda-x)^2} dx \preceq \frac{\l}{\sqrt{\L} \sqrt{\lambda}} = s \pto.
\end{multline*}

\item Using \eqref{psisplusprecis2} to control $\psis(y)$ by $s \frac{\l \sqrt{\lambda-y}}{\lambda^{3/2}}$, we write
\begin{multline*}
\int_{\lambda/2}^{\lambda} dx \sqrt{\lambda} \int_{y-x > \frac{\lambda -x}{2}} \sqrt{\lambda} \frac{|\psi(y)|}{(y-x)^3} dy \preceq s \int_{\lambda/2}^{\lambda} dx \frac{\lambda}{(\lambda-x)^3} \int_{x + \frac{\lambda-x}{2}}^{\lambda} \frac{\l}{\lambda^{3/2}} \sqrt{\lambda -y} dy \\
 \preceq s \int_{\lambda/2}^{\lambda} dx \frac{\lambda}{(\lambda-x)^3} \frac{\l}{\lambda^{3/2}} (\lambda - x)^{3/2} 
\preceq s \frac{\l}{\sqrt{\lambda} \sqrt{\L}} = s\pto.
\end{multline*}
$$
\star \star \star
$$
\end{enumerate}
\item For $\lambda - \L < x < y < \lambda$. We use \eqref{magni2}, the computation is similar to the case $\frac{\lambda}{2} < x < \lambda - \L$ and $y-x < \frac{1}{2} (\lambda - x)$ above, we obtain again an error as $s \frac{\l}{\sqrt{\L} \sqrt{\lambda}}$ which is $s \pto$.
$$
\star \star \star
$$
\item For $x,y$ in $(-4\l, 4\l)$ the proof is as in the very first case.
$$
\star \star \star
$$

\item For $0 < x < 2 \l$, and $4\l < -y < \lambda$, we use \eqref{magni2}, we write $\frac{1}{|x-y|} \leq \frac{1}{|y|}$ and we are left with
$$
\int_{0}^{2\l} \sqrt{x} dx \int_{4\l}^{\lambda} \sqrt{y} \left( \frac{|\psis'(x)|}{y^2} + \frac{|\psis'(y)|}{y^2} + \frac{|\psis(x)|}{y^3} + \frac{|\psis(y)|}{y^3} \right) dy.
$$ 
We have $\psis(x) \preceq s$, $\psis'(x) \preceq \frac{s}{\l}$ and the corresponding terms give
$$
\int_{0}^{2\l} \sqrt{\l} \int_{4\l}^{\lambda} \sqrt{y} \left( \frac{1}{\l y^2} + \frac{1}{y^3} \right) dy = s O(1).
$$
For the two other terms, we obtain
$$
\int_{0}^{2\l} \sqrt{x} dx \int_{4\l}^{\lambda} \sqrt{y} \frac{|\psis'(y)|}{y^2} dy \preceq s \l^{3/2} \int_{4\l}^{\lambda} \frac{|\muL(y)|}{y^{3/2}} dy,
$$
and using the bounds \eqref{muL0} we see that this is $s O(1)$.

All these terms are in fact improved to $s \pto$ as above.
$$
\star \star \star
$$

\item For $2\l < x < \frac{\lambda}{2}$ and $x < - y$, the computation is similar to the case $2\l < x < \frac{\lambda}{2}$ and $\frac{4}{3} x < y$, since we can  write $\frac{1}{|x-y|} \leq \frac{1}{|y|}$.
$$
\star \star \star
$$

\item Finally, for $\frac{\lambda}{2} < x < \lambda$ and $-\lambda < y < - \frac{\lambda}{2}$, we use \eqref{magni2} and we are left to bound, after replacing $\sqrt{x}, \sqrt{y}$ by $\sqrt{\lambda}$, and $|y-x|$ by $\lambda$, the quantity
$$
\int_{\lambda/2}^{\lambda} \sqrt{\lambda}  dx \int_{\lambda/2}^{\lambda} \sqrt{\lambda} \left( \frac{|\psis'(y)|}{\lambda^2} + \frac{|\psis(y)|}{\lambda^3}  \right) dy,
$$
where we use the symmetry in $x,y$ to forget about the $\psis'(x), \psis(x)$ terms. We have
$$
\lambda^2  \int_{\lambda/2}^{\lambda} \frac{|\psis'(y)|}{\lambda^2} dy \preceq \int_{\lambda/2}^{\lambda} |\muL(y)| dy = s O(1),
$$
and, using \eqref{psisplusprecis2}, 
$$
\lambda^2  \int_{\lambda/2}^{\lambda} \frac{|\psis(y)|}{\lambda^3} dy \preceq \frac{1}{\lambda} \frac{s \l}{\lambda^{3/2}} \int_{\lambda/2}^{\lambda} \sqrt{\lambda - y} dy \preceq \frac{s \l}{\lambda} = s \pto. 
$$
\end{enumerate}
This concludes the study of $\MainZ$.
$$
\star \star  \star \star \star 
$$

\textbf{The term $\MainA$.}
We recall that
$$
\MainA(\eta) = \sum_{i=-\lambda}^{\lambda} \sum_{|j| = \lambda - \L}^{\lambda - \L/10} \tD_i \tD_j \frac{|\partial_x \Ff|_{V(i,j)}}{\L}.
$$
Using symmetries, it is enough to study 
$$
\sum_{i = - \lambda}^{\lambda} \sum_{j = \lambda-\L}^{\lambda} \tD_i \tD_j \frac{|\partial_x \Ff|_{V(i,j)}}{\L}.
$$
When taking the expectation, we use the discrepancy estimates and replace $\tD_i \tD_j$ by $\sqrt{i} \sqrt{j}$, keeping in mind that any $O(1)$ can be improved to $\pto$ by using \eqref{mieuxCS} instead of \eqref{CSmoyen}.

We split the first sum into $i \leq \lambda - 3 \L$ and $i \geq \lambda - 3\L$.
\begin{enumerate}
\item  For the first sum, we use \eqref{magni5}, and study
$$
\sum_{i = - \lambda}^{\lambda-3\L} \sum_{j = \lambda-\L}^{\lambda} \frac{\sqrt{i} \sqrt{j}}{\L} \left( \frac{|\psis'(i)}{|j-i|} + \frac{|\psis(i)| + |\psis(j)|}{(j-i)^2} \right).
$$

Replacing $\sqrt{j}$ by $\sqrt{\lambda}$, $|\psis(j)|$ by $\frac{\l \sqrt{\L}}{\lambda^{3/2}}$ (in view of \eqref{psisplusprecis2}) and $j-i$ by $\lambda - i$,  we are left with
\begin{equation} 
\label{untermequirevient}
\sum_{i = - \lambda}^{\lambda-3\L} \frac{\L \sqrt{\lambda}}{\L} \sqrt{i} \left( \frac{|\psis'(i)|}{\lambda - i} + \frac{|\psis(i)|}{(\lambda-i)^2} + \frac{\sqrt{\L} \l}{\lambda^{3/2} (\lambda -i)^2} \right).
\end{equation}

We decompose the sum further
\begin{enumerate}
\item For $- \lambda \leq i \leq 2\l$, 
we use the fact that
$\psis'(i) \preceq s \muL(i)$ and $|\muL|_{L^1} \preceq 1$, that $\psis(i) \preceq s$, we replace $\lambda -i$ by $\lambda$ and we bound $\sqrt{|i|}$ by $\sqrt{\lambda}$. We obtain
$$
s \sum_{i = - \lambda}^{2\l} \frac{\L \sqrt{\lambda}}{\L} \sqrt{\lambda} \left( \frac{|\muL(i)|}{\lambda} + \frac{1}{\lambda^2} + \frac{\sqrt{\L} \l}{\lambda^{3/2} \lambda^2} \right) = s O(1).
$$
\item For $2\l \leq i \leq \lambda/2$, 
we use the fact that $\psis'(i) \preceq s \muL(i) \preceq s \frac{\l}{i^2}$, that $\psis(i) \preceq s \frac{\l}{i}$, we replace $\lambda -i$ by $\lambda$. We obtain
$$
s \sum_{i = 2\l}^{\lambda/2} \frac{\L \sqrt{\lambda}}{\L} \sqrt{i} \left( \frac{\l}{i^2 \lambda} + \frac{\l}{i \lambda^2} + \frac{\sqrt{\L} \l}{\lambda^{3/2} \lambda^2} \right) = s \pto.
$$
\item For $\lambda /2 \leq i \leq \lambda - 3\L$, 
we use the fact that $\psis'(i) \preceq s \muL(i) \preceq s\frac{\l}{\lambda^{3/2} \sqrt{\lambda -i}}$, that $\psis(i) \preceq s \frac{\l \sqrt{\lambda-i}}{\lambda^{3/2}}$, we replace $\sqrt{i}$ by $\sqrt{\lambda}$. We obtain
$$
s\sum_{i = \lambda/2}^{\lambda-3\L} \frac{\L \sqrt{\lambda}}{\L} \sqrt{\lambda} \left( \frac{\l}{\lambda^{3/2} (\lambda - i)^{3/2}} + \frac{\l \sqrt{\lambda -i}}{\lambda^{3/2} (\lambda-i)^2} + \frac{\sqrt{\L} \l}{\lambda^{3/2} (\lambda -i)^2} \right) = s\pto.
$$
\end{enumerate}
\item For the second sum, we use \eqref{magni6}, observe that near $\lambda$ we have, in view of \eqref{bound:psis2} and \eqref{muL1},
$$
\sup_{t \in [\lambda - 4 \L, \lambda]} |\psis^{(\2)}(t)| \preceq s \frac{\l}{\lambda^{3/2} \L^{3/2}}
$$
and study
$$
s \sum_{i = \lambda - 3\L}^{\lambda} \sum_{j = \lambda-\L}^{\lambda} \frac{\sqrt{i}\sqrt{j}}{\L} \frac{\l}{\lambda^{3/2} \L^{3/2}}.
$$
We replace $\sqrt{i} \sqrt{j}$ by $\lambda$ and get $s \pto$ by direct computation.
\end{enumerate}
This concludes the study of $\MainA$.
$$
\star \star  \star \star \star 
$$

\paragraph{\textbf{The term $\MainB$.}}
We recall that
$$
\MainB(\eta) = \sum_{|i| = \lambda - \L}^{\lambda - \L/10} \sum_{|j| = \lambda - \L}^{\lambda - \L/10} \tD_i \tD_j \frac{|\Ff|_{V(i,j)}}{\L^2}
$$
Using symmetries, it enough to study
$$
\sum_{i = \lambda - \L}^{\lambda - \L/10} \sum_{j = \lambda - \L}^{\lambda - \L/10} \tD_i \tD_j \frac{|\Ff|_{V(i,j)}}{\L^2} + \sum_{i = -\lambda }^{-\lambda + \L/10} \sum_{j = \lambda - \L}^{\lambda - \L/10} \tD_i \tD_j \frac{|\Ff|_{V(i,j)}}{\L^2}.
$$
When taking the expectation, we use the discrepancy estimates and replace $\tD_i \tD_j$ by $\sqrt{i} \sqrt{j}$. Here we replace further $\sqrt{i} \sqrt{j}$ by $\lambda$. 
\begin{enumerate}
\item For the first sum, we use \eqref{magni4} and observe that, near $\lambda$, we have (in view of \eqref{bound:psis1} and \eqref{muL0})
$$
\sup_{t \in \lambda - 2\L, \lambda} |\psis^{(\1)}(t)| \preceq s \frac{\l}{\lambda^{3/2} \sqrt{\L}},
$$
hence we obtain
\begin{equation*}
s \sum_{i = \lambda - \L}^{\lambda - \L/10} \sum_{j = \lambda - \L}^{\lambda - \L/10} \lambda \frac{|\Ff|_{V(i,j)}}{\L^2} \preceq \L^2 \lambda  \frac{\l}{\L^2 \lambda^{3/2} \sqrt{\L}} = s \pto.
\end{equation*}
\item For the second sum, we use \eqref{magni3}, and observe that, in view of \eqref{psisplusprecis2}, we have, for $i$ near $-\lambda$ and $j$ near $\lambda$,
$$
\frac{|\psis(i)| + |\psis(j)|}{(j-i)} \preceq s \frac{\l \sqrt{\L}}{\lambda^{3/2} \lambda}
$$
and we thus obtain
$$
s \sum_{i = -\lambda }^{-\lambda + \L/10} \sum_{j = \lambda - \L}^{\lambda - \L/10} \lambda \frac{\l \sqrt{\L}}{\lambda^{3/2} \L^2 \lambda} = s \pto.
$$
\end{enumerate}
This concludes the study of $\MainB$.
$$
\star \star  \star \star \star 
$$

\paragraph{\textbf{The term $\MainC$.}}
We recall that
$$
\MainC(\eta)  = \left( \L + \left|\Discr_{|x| \in [\lambda - \L/8, \lambda]}\right| \right) \sum_{j=-\lambda}^\lambda \sup_{|x| \in [\lambda - \L/8, \lambda]} |\partial_y \Ff|_{V(x,j)} \tD_j
$$
Taking the expectation, we use the discrepancy estimates and get
$$
\Esp \left[ \left( \L + \left|\Discr_{|x| \in [\lambda - \L/8, \lambda]}\right| \right) \tD_j \right] \preceq \L \sqrt{j},
$$
so we study
$$
\sum_{j=-\lambda}^\lambda \L \sqrt{j} \sup_{|x| \in [\lambda - \L/8, \lambda]} |\partial_y \Ff|_{V(x,j)}.
$$
We split the sum into $j \leq \lambda - 3\L$ and $j \geq \lambda - 3\L$.
\begin{enumerate}
\item For $j \leq \lambda - 3\L$, we use \eqref{magni5} (switching the roles of $x$ and $y$) and write, for $x$ in $[\lambda - \L/8, \lambda]$
$$
|\partial_y \Ff|_{V(x,j)} \preceq \frac{|\psis'(j)|}{|j-x|} + \frac{|\psis(x)|}{(j-x)^2} + \frac{|\psis(j)|}{(j-x)^2}.
$$
We replace $j-x$ by $\lambda -j$ and (in view of \eqref{psisplusprecis2}) $|\psis(x)|$ by $s \frac{\l \sqrt{\L}}{\lambda^{3/2}}$, and we obtain
$$
\sup_{|x| \in [\lambda - \L/8, \lambda]} |\partial_y \Ff|_{V(x,j)} \preceq \frac{|\psis'(j)|}{\lambda - j} + \frac{s \l \sqrt{\L}}{\lambda^{3/2} (\lambda-j)^2} + \frac{|\psis(j)|}{(\lambda-j)^2},
$$
we are thus left to study
$$
\sum_{j=-\lambda}^{\lambda - 3\L} \L \sqrt{j} \left( \frac{|\psis'(j)|}{\lambda - j} + \frac{s \l \sqrt{\L}}{\lambda^{3/2} (\lambda-j)^2} + \frac{|\psis(j)|}{(\lambda-j)^2} \right).
$$
This is actually much smaller than \eqref{untermequirevient}, which was already treated.
\item For $j \geq \lambda - 3\L$, we use \eqref{magni6} and write, for $x$ in $[\lambda - \L/8, \lambda]$
$$
|\partial_y \Ff|_{V(x,j)} \preceq \sup_{t \in [\lambda - 4\L, \lambda]} |\psis^{(\2)}(t)| \preceq s \frac{\l}{\lambda^{3/2} \L^{3/2}}, 
$$
and a direct computation gives 
$$
s \sum_{j=\lambda-3\L}^\lambda \L \sqrt{j} \sup_{|x| \in [\lambda - \L/8, \lambda]} |\partial_y \Ff|_{V(x,j)} \preceq s \L^2 \sqrt{\lambda} \frac{\l}{\lambda^{3/2} \L^{3/2}} = s \pto.
$$
\end{enumerate}
This concludes the study of $\MainC$.
$$
\star \star  \star \star \star 
$$

\paragraph{\textbf{The term $\MainD$}.}
We recall that
$$
\MainD(\eta)  = \left( \L + \left|\Discr_{|x| \in [\lambda - \L/8, \lambda]}\right| \right) \sum_{|j| = \lambda - \L}^{\lambda - \L/10} \frac{\sup_{|x| \in [\lambda - \L/8, \lambda]} |\Ff|_{V(x,j)}}{\L} \tD_j . 
$$
For the same reasons as above, we are led to study
$$
\sum_{|j| = \lambda - \L}^{\lambda - \L/10} \frac{\sup_{|x| \in [\lambda - \L/8, \lambda]} |\Ff|_{V(x,j)}}{\L} \L \sqrt{\lambda},
$$
and we split the sum in two parts: $j$ near $-\lambda$ and $j$ near $\lambda$. For the first part, we use \eqref{magni3}, and for the second part we use \eqref{magni4} to control $|\Ff|_{V(x,j)}$. After some computation, we obtain $s \pto$. 
This concludes the study of $\MainD$, and the proof of the proposition.
\end{proof}

\subsection{Proof of Corollary \ref{coro:ErrorDFpetit}}
\label{sec:proofErrorDF}
\begin{proof}[Proof of Corollary \ref{coro:ErrorDFpetit}]
We can split $\Lambda^c$ into $\{ x \geq \lambda \}$ and $\{ x \leq - \lambda \}$, both parts yield an equivalent contribution, so we only consider the first one. We need an adaptation of the \textit{a priori} bound \eqref{aprioribound} to a slightly different context. 

\begin{claim}[A priori bound - “hard edge” and decay assumption]
\label{claim:hardedge}
Let $g$ be a $C^1$ function such that
\begin{equation}
\label{decayassump}
\limsup_{x \to \infty} |x g(x)| < + \infty, \quad \limsup_{x \to \infty} x^2 |g'(x)| < + \infty,
\end{equation}
then, $\sineb$-a.s. both sides of the following inequality are finite, and the inequality holds
$$
\int_{\lambda}^{+\infty} g(x) (d\C - dx) \preceq \sum_{j=\lambda}^{+\infty} |g|_{\1, \Vj} \DR_j + g(\lambda) |\Discr_{[\lambda, \lambda+1]}|.
$$
\end{claim}
\begin{proof}[Proof of Claim \ref{claim:hardedge}]
We follow the same lines as for the proof of Proposition \ref{prop:aprioribounds}. We split the domain of integration into unit intervals and use the mean value theorem, in order to get, for $M > \lambda$ fixed
\begin{multline*}
\int_{\lambda}^M  g(x) (d\C - dx) = \sum_{k=\lambda}^{M-1} \int_{k}^{k+1} g(x) (d\C - dx) \\
\preceq \sum_{k = \lambda}^{M-1} g(k) \Discr_{[k, k+1]} + \Oun\left(|g|_{\1,\Vk}\right)\left(1 + |\Discr_{[k, k+1]}|\right).
\end{multline*}
We write, for any $k$, $\Discr_{[k, k+1]} = \Discr_{[\lambda, k+1]} - \Discr
_{[\lambda, k]}$, 
and perform a summation by parts to get 
\begin{equation*}
\sum_{k = \lambda}^{M-1} g(k) \Discr_{[k, k+1]} = \sum_{k = \lambda+1}^{M-1} \left( g(k-1) - g(k) \right) \Discr_{[\lambda, k]} 
+ g(M-1) \Discr_{[\lambda, M]} + g(\lambda) \Discr_{[\lambda, \lambda+1]}.  
\end{equation*}
In view of \eqref{decayassump}, the boundary term $g(M-1) \Discr_{[\lambda, M]}$ tends almost surely to $0$ as $M \to \infty$, because $\frac{1}{M} \Discr_{[\lambda, M]}$ tends almost surely to $0$.  On the other hand, the series 
$$
\sum_{k = \lambda}^{+\infty} |g|_{\1,\Vk} \left(1 + |\Discr_{[k, k+1]}|\right)
$$
is almost surely convergent, because we have, in view of \eqref{discrestimateA} and \eqref{decayassump}
$$
\limsup_{k \to \infty} \Esp \left[  \left( k^{2} |g|_{\1,\Vk} \left(1 + |\Discr_{[k, k+1]}| \right)\right)^2  \right] < + \infty.
$$
Sending $M \to \infty$ yields the result.
\end{proof}

We can easily check that $\ErrorDF$ satisfies the decay assumption \eqref{decayassump}. Using Claim \ref{claim:hardedge}, we get
$$
\int_{\lambda}^{+\infty} \ErrorDF(\C)(x) (d\C - dx) \preceq s \sum_{j = \lambda}^{+\infty} |\ErrorDF|_{\1, \Vj} \DR_j + |\ErrorDF(\lambda)| \Discr_{[\lambda, \lambda+1]}.
$$ 

\textbf{The boundary term.}
We claim that
\begin{equation}
\label{BTDF}
\Esp \left[ \left|\ErrorDF(\lambda) \Discr_{[\lambda, \lambda+1]}\right| \right] \preceq s \frac{\l \log(\lambda)}{\lambda^{3/2}} = s \pto.
\end{equation}
Indeed, using \eqref{ErrorDFpaprime} and the discrepancy estimates \eqref{discrestimateA} for $\DR_i$, we obtain
\begin{multline*}
\Esp \left[ |\ErrorDF(\lambda)| \Discr_{[\lambda, \lambda+1]} \right] \preceq \sum_{i = - \lambda}^{\lambda - \L} \left(\frac{|\psis(i)|}{(x-i)^2} + \frac{|\psis'(i)|}{(x-i)} \right) \Esp\left[\DR_i  \Discr_{[\lambda, \lambda+1]} \right] \\
\preceq \sum_{i = - \lambda}^{\lambda - \L} \left(\frac{|\psis(i)|}{(x-i)^2} + \frac{|\psis'(i)|}{(x-i)} \right) \sqrt{\lambda - i}.
\end{multline*}
We use \eqref{bound:psis1} and \eqref{muL0} to control the contribution of the $\psis'(i)$ terms, and \eqref{psisplusprecis2} to control the contribution of the $\psis(i)$ terms. For example, we have
$$
\sum_{i = \lambda/2}^{\lambda - \L} \frac{|\psis'(i)|}{(x-i)}  \sqrt{\lambda - i} \preceq s \sum_{i = \lambda/2}^{\lambda - \L} \frac{ \l}{\lambda^{3/2}} \frac{\sqrt{\lambda - i} }{(\lambda - i)^2} \sqrt{\lambda - i} \preceq s \frac{\l \log(\lambda)}{\lambda^{3/2}}.
$$ 

\textbf{The main contribution.}
We now claim that
\begin{equation}
\label{MTDF}
\Esp \left[ \sum_{j = \lambda}^{+\infty} |\ErrorDF|_{\1, \Vj} \DR_j \right] = s \pto.
\end{equation}
To prove \eqref{MTDF}, we use \eqref{ErrorDFprime} and write
$$
\sum_{j = \lambda}^{+\infty} |\ErrorDF|_{\1, \Vj} \DR_j \preceq \sum_{j = \lambda}^{+\infty} \sum_{i = - \lambda}^{\lambda - \L} \left(\frac{|\psis(i)|}{(j-i)^3} + \frac{|\psis'(i)|}{(j-i)^2} \right) \DR_i \DR_j.
$$
Taking the expectation and using Cauchy-Schwarz's inequality, we get
\begin{multline}
\label{CSA}
\Esp\left[ \sum_{j = \lambda}^{+\infty} |\ErrorDF|_{\1, \Vj} \DR_j \right]  \\ \preceq  \sum_{j = \lambda}^{+\infty} \sum_{i = - \lambda}^{\lambda - \L} \left(\frac{|\psis(i)|}{(j-i)^3} + \frac{|\psis'(i)|}{(j-i)^2} \right) \Esp\left[\left(\DR_i\right)^2\right]^{1/2} \Esp\left[\left(\DR_j\right)^2\right]^{1/2}
\end{multline}
Using the discrepancy estimate \eqref{discrestimateA} we obtain
$$
\Esp\left[ \sum_{j = \lambda}^{+\infty} |\ErrorDF|_{\1, \Vj} \DR_j \right] \preceq \sum_{j = \lambda}^{+\infty} \sum_{i = - \lambda}^{\lambda-\L} \left(\frac{|\psis(i)|}{(j-i)^3} + \frac{|\psis'(i)|}{(j-i)^2} \right) \sqrt{|\lambda - i|}\sqrt{|\lambda - j|}.
$$
Let us keep in mind that, as in the proof of Proposition \ref{prop:mainterms}, we may use the sharper discrepancy estimates \eqref{discrestimateB} instead of \eqref{discrestimateA}, and take advantage of the fact that
$$
\Esp\left[\left(\DR_j\right)^2\right] = o_{|j - \lambda| \to \infty}\left(j-\lambda\right).
$$

\begin{description}
\item[The terms $i$ far from $\lambda$] We first treat the case $-\lambda \leq i \leq \lambda /2$.
\begin{enumerate}
\item 
For any $j \geq \lambda$, using the estimates \eqref{psisplusprecis2} on $\psis$, we may write
$$
\sum_{i = - \lambda}^{\lambda /2} \frac{|\psis(i)|}{(j-i)^3} \sqrt{\lambda - i} \preceq s \frac{\sqrt{\lambda}}{j^3} \sum_{i = - \lambda}^{\lambda /2} |\psis(i)| \preceq s \frac{ \sqrt{\lambda} \l \log(\lambda)}{j^3}.
$$
We thus get:
\begin{equation}
\label{moinsllsur2A}
\sum_{j = \lambda}^{+\infty} \sum_{i = - \lambda}^{\lambda /2} \frac{|\psis(i)|}{(j-i)^3} \sqrt{\lambda - i} \sqrt{j - \lambda} \preceq s \sum_{j = \lambda}^{+\infty} \frac{ \sqrt{\lambda} \l \log(\lambda)}{j^{5/2}} \preceq s \frac{ \l \log(\lambda)}{ \lambda} = s \pto.
\end{equation}

\item For any $j \geq \lambda$, using \eqref{bound:psis1} and \eqref{muL0}, we write
$$
\sum_{i = - \lambda}^{\lambda /2} \frac{|\psis'(i)|}{(j-i)^2} \sqrt{\lambda - i} \preceq \frac{\sqrt{\lambda}}{j^2} \sum_{i = - \lambda}^{\lambda /2} |\psis'(i)| \preceq s \frac{\sqrt{\lambda}}{j^2}.
$$
We thus get:
$$
\sum_{j = \lambda}^{+\infty} \sum_{i = - \lambda}^{\lambda /2} \frac{|\psis'(i)|}{(j-i)^2} \sqrt{\lambda - i} \sqrt{j - \lambda} \preceq s \sum_{j = \lambda}^{+\infty} \frac{ \sqrt{\lambda} \sqrt{j - \lambda}}{j^2}. 
$$
A rough bound would only yield a $O(1)$ contribution here. Instead, we split the sum into
$$
\sum_{j = \lambda}^{\lambda + \log(\lambda)} \frac{ \sqrt{\lambda} \sqrt{j - \lambda}}{j^2} \preceq \frac{1}{\lambda^{3/2}} \sum_{j = 0}^{\log(\lambda)} \sqrt{k} = \pto,
$$
and the remainder where $j- \lambda \geq \log(\lambda)$, in which we use \eqref{discrestimateB} instead of \eqref{discrestimateA}, which allows us to replace $\sqrt{j - \lambda}$ by $o_{\lambda}(\sqrt{j - \lambda})$, and
$$
\sum_{j = \lambda + \log(\lambda)}^{+\infty} \frac{ \sqrt{\lambda} \times o_{\lambda} \left(\sqrt{j - \lambda}\right)}{j^2} = \pto.
$$
Hence
\begin{equation}
\label{moinsllsur2B}
\sum_{j = \lambda}^{+\infty} \sum_{i = - \lambda}^{\lambda /2} \frac{|\psis'(i)|}{(j-i)^2} \sqrt{\lambda - i} \times o\left( \sqrt{j - \lambda} \right) = s \pto.
\end{equation}
\end{enumerate}

Combining \eqref{moinsllsur2A} and \eqref{moinsllsur2B}, we see that the contribution in \eqref{CSA} coming from the terms “$i$ far from $\lambda$” i.e. here $-\lambda \leq i \leq \lambda/2$, is $s \pto$.

\item[The terms $i$ close to $\lambda$] We now consider $\lambda /2 \leq i \leq \lambda - \L$
\begin{enumerate}
\item  For any $j \geq \lambda$, using the estimates \eqref{psisplusprecis2} on $\psis$, we may write
$$
\sum_{i = \lambda/2}^{\lambda - \L} \frac{|\psis(i)|}{(j-i)^3} \sqrt{\lambda - i} \preceq s \sum_{i = \lambda/2}^{\lambda - \L} \frac{\l}{\lambda^{3/2}} \frac{\lambda -i}{(j-\lambda + \lambda-i)^3}.
$$
We distinguish the cases $\lambda - i \leq j - \lambda$ and $\lambda -i \geq j - \lambda$.

\begin{enumerate}
\item We have (the sum being non-empty only if $j - \lambda \geq \L$):
$$
\sum_{i \in  [\lambda/2, \lambda- \L] | \lambda -i \leq j - \lambda} \frac{\l}{\lambda^{3/2}} \frac{\lambda -i}{(j-\lambda + \lambda-i)^3} \preceq \frac{\l}{\lambda^{3/2}} \frac{1}{(j-\lambda)^2} \sum_{i \in  [\lambda/2, \lambda- \L] | \lambda -i \leq j - \lambda} 1 \preceq \frac{\l}{\lambda^{1/2}} \frac{1}{(j-\lambda)^2},
$$
\item 
On the other hand (the sum being non-empty only if $j - \lambda \leq \lambda/2$):
$$
\sum_{i \in  [\lambda/2, \lambda - \L] | \lambda -i > j - \lambda} \frac{\l}{\lambda^{3/2}} \frac{\lambda -i}{(j-\lambda + \lambda-i)^3} \preceq \frac{\l}{\lambda^{3/2}} \sum_{i \in  [\lambda/2, \lambda - \L] | \lambda -i > j - \lambda} \frac{1}{(\lambda-i)^2} \preceq \frac{\l}{\lambda^{3/2}} \frac{1}{j-\lambda},
$$
\end{enumerate}
We may thus write
\begin{equation*}
\sum_{j = \lambda}^{+ \infty} \sum_{i = \lambda/2}^{\lambda - \L} \frac{|\psis(i)|}{(j-i)^3} \sqrt{\lambda - i} \sqrt{j - \lambda} \preceq s \sum_{j = \lambda + \L}^{+ \infty} \frac{\l}{\lambda^{1/2}} \frac{1}{(j-\lambda)^{3/2}} + s \sum_{j= \lambda}^{3 \lambda /2}  \frac{\l}{\lambda^{3/2}} \frac{1}{\sqrt{j-\lambda}}
\preceq s \frac{\l}{\lambda^{1/2} \L^{1/2}} + s \frac{\l}{\lambda}.
\end{equation*}
We obtain
\begin{equation}
\label{farfroml1}
\sum_{j = \lambda}^{+\infty} \sum_{i = \lambda/2}^{\lambda - \L} \frac{|\psis(i)|}{(j-i)^3} \sqrt{\lambda - i} \sqrt{j - \lambda} = s \pto.
\end{equation}

\item Concerning the terms $\psis'(i)$, we use \eqref{bound:psis0} and \eqref{muL0} and write, for any $j \geq \lambda$,
$$
\sum_{i = \lambda/2}^{\lambda - \L} \frac{|\psis'(i)|}{(j-i)^2} \sqrt{\lambda - i} \preceq s \sum_{i = \lambda/2}^{\lambda - \L} \frac{\l}{\lambda^{3/2} \sqrt{\lambda - i} (j-i)^2} \sqrt{\lambda - i}  \preceq s \frac{\l}{\lambda^{1/2}} \frac{1}{(j-\lambda + \L)^2}.
$$
We thus have
$$
\sum_{\lambda}^{+ \infty} \sum_{i = \lambda/2}^{\lambda - \L} \frac{|\psis'(i)|}{(j-i)^2} \sqrt{\lambda - i} \sqrt{j - \lambda} \preceq s\sum_{\lambda}^{+ \infty}\frac{\l}{\lambda^{1/2}} \frac{\sqrt{j - \lambda}}{(j-\lambda + \L)^2} \preceq s\frac{\l}{\lambda^{1/2} \L^{1/2}}
$$
and we obtain
\begin{equation}
\label{farfroml2}
\sum_{j = \lambda}^{+\infty} \sum_{i = \lambda/2}^{\lambda - \L} \frac{|\psis'(i)|}{(j-i)^2} \sqrt{\lambda - i} \sqrt{j - \lambda} = s \pto.
\end{equation}
\end{enumerate}
Combining \eqref{farfroml1} and \eqref{farfroml2}, we see that the contribution in \eqref{CSA} coming from the terms “$i$ close to $\lambda$”, i.e. here $\lambda/2 \leq i \leq \lambda - \L$, is $\pto$.
\end{description}
This concludes the proof of \eqref{MTDF}, which, combined with \eqref{BTDF}, yields \eqref{petitDF}.
\end{proof}

\subsection{Proof of Lemma \ref{lem:controlonGam}}
\label{sec:proofcontrolGam}
\begin{proof}[Proof of Lemma \ref{lem:controlonGam}]
\textbf{For $3\lambda/4 \leq |x| \leq 4 \lambda$.}
For simplicity we consider $\Gamma(\lambda)$, the proof extends readily to $\Gamma(x)$ for $3\lambda/4 \leq |x| \leq 4 \lambda$. We have, by definition
$$
\Gamma(\lambda) = \int - \log(\lambda -y) \mulambda(y) dy.
$$

\textit{The $|y| \leq \hal \lambda$ part.}
We want to show:
\begin{equation} \label{yleqlambda2A}
\int_{-\lambda/2}^{\lambda/2} \log(\lambda - y) \mulambda(y) dy \preceq \frac{ \l \log^2(\lambda)}{\lambda}.
\end{equation}
We recall that, by definition,
$$
\mulambda(y) = \frac{-1}{\pi} \frac{1}{\sqrt{\lambda^2-y^2}} \HLP(y) = \frac{-1}{\pi} \frac{1}{\sqrt{\lambda^2-y^2}} \PV \int \frac{\varphi'(t) \sqrt{\lambda^2-t^2}}{y-t} dt.
$$
For $|y| \leq \hal \lambda$, we write
\begin{multline*}
\log |\lambda - y| \frac{1}{\sqrt{\lambda^2-y^2}} = \left( \log(\lambda) + \Oun\left( \frac{|y|}{\lambda} \right) \right) \left( \frac{1}{\lambda} + \Oun\left(\frac{|y|^2}{\lambda^3} \right) \right)  \\
= \frac{\log(\lambda)}{\lambda} + \Oun \left( \frac{|y|}{\lambda^2} + \frac{|y|^2 \log(\lambda)}{\lambda^3} \right) = \frac{\log(\lambda)}{\lambda} +  \Oun \left( \frac{\log(\lambda) |y|}{\lambda^2}\right) .
\end{multline*}
and thus
\begin{equation}
\label{yleqlambda2}
\log |\lambda - y| \frac{1}{\sqrt{\lambda^2-y^2}} \HLP(y) = 
\frac{\log(\lambda)}{\lambda} \HLP(y)  + \Oun \left( \frac{\log (\lambda) |y|}{\lambda^2}\right) |\HLP(y)|.
\end{equation}

Using \eqref{integraledeH}, we get
\begin{equation}
\label{yleqlambda2PV} 
\int_{-\lambda/2}^{\lambda/2} \frac{\log(\lambda)}{\lambda} \HLP(y)  \preceq \frac{\l \log (\lambda)}{\lambda}.
\end{equation}
Using the bounds \eqref{mulambda0}, we can check that
$$
\int_{-\lambda/2}^{\lambda/2} |y| |\HLP(y)| dy \preceq \l \lambda \log(\lambda),
$$
and thus
\begin{equation}
 \int_{-\lambda/2}^{\lambda/2}  \frac{\log (\lambda)}{\lambda^2} |y| |\HLP(y)| dy  \preceq  \frac{\l \log^2(\lambda)}{\lambda}.
\end{equation}
We obtain \eqref{yleqlambda2A}.

\textit{The $|y| \geq \lambda / 2$ part.}
We want to show:
\begin{equation}
 \label{ygeqlambda2}
\int_{|y| \geq \lambda / 2} \log |\lambda -y| \mulambda(y) dy \preceq  \frac{\l \log(\lambda)}{\lambda}.
\end{equation}
We use \eqref{mulambda0} and an elementary computation. The mass of $\mulambda$ outside $[-\lambda/2, \lambda/2]$ is indeed $\Oun\left(\frac{\l}{\lambda}\right)$.

Combining \eqref{yleqlambda2A} and \eqref{ygeqlambda2}, we obtain \eqref{Gambord}.

\textbf{For $\lambda \leq |x| \leq 4 \lambda$.}
We can e.g. assume that $\lambda \leq x \leq 4\lambda$. write $\Gam'(x)$ as
$$
\Gam'(x) = \int \frac{1}{x-t} \mulambda(t) dt, 
$$
and we use \eqref{mulambda0}. We have
$$
\int_{-\lambda}^{\lambda^2} \frac{1}{x-t} |\mulambda(t)| \leq \frac{1}{\lambda},
$$
and we focus on the remaining part $t \in [\lambda/2, \lambda]$. We write
$$
\int_{\lambda^2}{\lambda} \frac{1}{x-t} |\mulambda(t)| dt \preceq \int_{\lambda^2}^{\lambda} \frac{1}{x-\lambda + \lambda-t} \frac{\l}{\lambda^{3/2} \sqrt{\lambda - t}} dt = \frac{\l}{\lambda^{3/2}} \int_{0}^{\lambda/2} \frac{1}{(x-\lambda + v) \sqrt{v}} dv.
$$
An elementary computation shows that
$$
 \int_{0}^{\lambda/2} \frac{1}{(x-\lambda + v) \sqrt{v}} \preceq \frac{1}{\sqrt{x-\lambda}}, 
$$
which yields \eqref{Gampprimepres}.

\textbf{For $|x| \geq 4 \lambda$.}
We write $\Gam'(x)$ as
$$
\Gam'(x) = \int \frac{1}{x-t} \mulambda(t) dt, 
$$
and, since $\mulambda$ has total mass $0$, a first-order expansion yields
$$
\Gam'(x) \preceq \frac{1}{x^2} \int |t| |\mulambda(t)|.
$$
We can use \eqref{mulambda0} to compute 
$\int |t| |\mulambda(t)| dt \preceq \l \log(\lambda),$
which yields \eqref{Gamprimeloin}.
\end{proof}

\bibliographystyle{alpha}
\bibliography{TCL}

\end{document}